\documentclass[reqno]{amsart}
\expandafter\ifx\csname ProvidesPackage\endcsname\relax\toks0=\expandafter{%
\fi\ProvidesPackage{diagrams}[2024/11/20 v3.97 Paul Taylor's commutative
diagrams]
\toks0=\bgroup}
\ifx\diagram\isundefined\else\expandafter\endinput
\fi

\edef\cdrestoreat{
\noexpand\catcode`\noexpand\@=\the\catcode`\@
\noexpand\catcode`\noexpand\#=\the\catcode`\#
\noexpand\catcode`\noexpand\$=\the\catcode`\$
\noexpand\catcode`\noexpand\<=\the\catcode`\<
\noexpand\catcode`\noexpand\>=\the\catcode`\>
\noexpand\catcode`\noexpand\:=\the\catcode`\:
\noexpand\catcode`\noexpand\;=\the\catcode`\;
\noexpand\catcode`\noexpand\!=\the\catcode`\!
\noexpand\catcode`\noexpand\?=\the\catcode`\?
\noexpand\catcode`\noexpand\+=\the\catcode'53
}\catcode`\@=11 \catcode`\#=6 \catcode`\<=12 \catcode`\>=12 \catcode'53=12
\catcode`\:=12 \catcode`\;=12 \catcode`\!=12 \catcode`\?=12

\ifx\diagram@help@messages\CD@qK\let\diagram@help@messages y\fi

\def\cdps@Rokicki#1{\special{ps:#1}}\let\cdps@dvips\cdps@Rokicki\let
\cdps@RadicalEye\cdps@Rokicki\let\CD@HB\cdps@Rokicki\let\CD@IK\cdps@Rokicki
\let\CD@HB\cdps@Rokicki
\def\cdps@Bechtolsheim#1{\special{dvitps: Literal "#1"}}%
\let\cdps@dvitps\cdps@Bechtolsheim\let\cdps@IntegratedComputerSystems
\cdps@Bechtolsheim
\def\cdps@Clark#1{\special{dvitops: inline #1}}
\let\cdps@dvitops\cdps@Clark
\let\cdps@OzTeX\empty\let\cdps@oztex\empty\let\cdps@Trevorrow\empty
\def\cdps@Coombes#1{\special{ps-string #1}}

\count@=\year\multiply\count@12 \advance\count@\month
\ifnum\count@>24382 
\ifnum\count@>24385 
\errhelp{You may press RETURN and carry on for the time being.}\errmessage{! remain
compatible with your files. Please get it NOW.}\fi\fi

\def\CD@DE{\global\let}\def\CD@RH{\outer\def}

{\escapechar\m@ne\xdef\CD@o{\string\{}\xdef\CD@yC{\string\}}
\catcode`\&=4 \CD@DE\CD@Q=&\xdef\CD@S{\string\&}
\catcode`\$=3 \CD@DE\CD@k=$\CD@DE\CD@ND=$
\xdef\CD@nC{\string\$}\gdef\CD@LG{$$}
\catcode`\_=8 \CD@DE\CD@lJ=_
\obeylines\catcode`\^=7 \CD@DE\@super=^
\ifnum\newlinechar=10 \gdef\CD@uG{^^J}
\else\ifnum\newlinechar=13 \gdef\CD@uG{^^M}
\else\ifnum\newlinechar=-1 \gdef\CD@uG{^^J}
\else\CD@DE\CD@uG\space\expandafter\message{! input error: \noexpand
\newlinechar\space is ASCII \the\newlinechar, not LF=10 or CR=13.}
\fi\fi\fi}

\mathchardef\lessthan='30474 \mathchardef\greaterthan='30476

\ifx\tenln\CD@qK
\font\tenln=line10\relax
\fi\ifx\tenlnw\CD@qK\ifx\tenln\nullfont\let\tenlnw\nullfont\else
\font\tenlnw=linew10\relax
\fi\fi

\ifx\inputlineno\CD@qK\csname newcount\endcsname\inputlineno\inputlineno\m@ne
\fi

\def\cd@shouldnt#1{\CD@KB{* THIS (#1) SHOULD NEVER HAPPEN! *}}

\def\get@round@pair#1(#2,#3){#1{#2}{#3}}
\def\get@square@arg#1[#2]{#1{#2}}
\def\CD@AE#1{\CD@PK\let\CD@DH\CD@@E\CD@@E#1,],}
\def\CD@m{[}\def\CD@RD{]}\def\commdiag#1{{\let\enddiagram\relax\diagram[]#1%
\enddiagram}}

\def\CD@BF{{\ifx\CD@EH[\aftergroup\get@square@arg\aftergroup\CD@YH\else
\aftergroup\CD@JH\fi}}
\def\CD@CF#1#2{\def\CD@YH{#1}\def\CD@JH{#2}\futurelet\CD@EH\CD@BF}

\def\CD@KK{|}

\def\CD@PB{
\tokcase\CD@DD:\CD@y\break@args;\catcase\@super:\upper@label;\catcase\CD@lJ:%
\lower@label;\tokcase{~}:\middle@label;
\tokcase<:\CD@iF;
\tokcase>:\CD@iI;
\tokcase(:\CD@BC;
\tokcase[:\optional@;
\tokcase.:\CD@JJ;
\catcase\space:\eat@space;\catcase\bgroup:\positional@;\default:\CD@@A
\break@args;\endswitch}

\def\switch@arg{
\catcase\@super:\upper@label;\catcase\CD@lJ:\lower@label;\tokcase[:\optional@
;
\tokcase.:\CD@JJ;
\catcase\space:\eat@space;\catcase\bgroup:\positional@;\tokcase{~}:%
\middle@label;
\default:\CD@y\break@args;\endswitch}

\let\CD@tJ\relax\ifx\protect\CD@qK\let\protect\relax\fi\ifx\AtEndDocument
\CD@qK\def\CD@PG{\CD@gB}\def\CD@GF#1#2{}\else\def\CD@PG#1{\edef\CD@CH{#1}%
\expandafter\CD@oC\CD@CH\CD@OD}\def\CD@oC#1\CD@OD{\AtEndDocument{\typeout{%
\CD@tA: #1}}}\def\CD@GF#1#2{\gdef#1{#2}\AtEndDocument{#1}}\fi\def\CD@ZA#1#2{%
\def#1{\CD@PG{#2\CD@mD\CD@W}\CD@DE#1\relax}}\def\CD@uF#1\repeat{\def\CD@p{#1}%
\CD@OF}\def\CD@OF{\CD@p\relax\expandafter\CD@OF\fi}\def\CD@sF#1\repeat{\def
\CD@q{#1}\CD@PF}\def\CD@PF{\CD@q\relax\expandafter\CD@PF\fi}\def\CD@tF#1%
\repeat{\def\CD@r{#1}\CD@QF}\def\CD@QF{\CD@r\relax\expandafter\CD@QF\fi}\def
\CD@tG#1#2#3{\def#2{\let#1\iftrue}\def#3{\let#1\iffalse}#3}\if y%
\diagram@help@messages\def\CD@rG#1#2{\csname newtoks\endcsname#1#1=%
\expandafter{\csname#2\endcsname}}\else\csname newtoks\endcsname\no@cd@help
\no@cd@help{See the manual}\def\CD@rG#1#2{\let#1\no@cd@help}\fi\chardef\CD@lF
=1 \chardef\CD@lI=2 \chardef\CD@MH=5 \chardef\CD@tH=6 \chardef\CD@sH=7
\chardef\CD@PC=9 \dimendef\CD@hI=2 \dimendef\CD@hF=3 \dimendef\CD@mF=4
\dimendef\CD@mI=5 \dimendef\CD@wJ=6 \dimendef\CD@tI=8 \dimendef\CD@sI=9
\skipdef\CD@uB=1 \skipdef\CD@NF=2 \skipdef\CD@tB=3 \skipdef\CD@ZE=4 \skipdef
\countdef\CD@JC=9 \countdef\CD@gD=8 \countdef\CD@A=7 \def\sdef#1#2{\def#1{#2}%
}\def\CD@L#1{\expandafter\aftergroup\csname#1\endcsname}\def\CD@RC#1{%
\expandafter\def\csname#1\endcsname}\def\CD@sD#1{\expandafter\gdef\csname#1%
\endcsname}\def\CD@vC#1{\expandafter\edef\csname#1\endcsname}\def\CD@nF#1#2{%
\expandafter\let\csname#1\expandafter\endcsname\csname#2\endcsname}\def\CD@EE
#1#2{\expandafter\CD@DE\csname#1\expandafter\endcsname\csname#2\endcsname}%
\def\CD@AK#1{\csname#1\endcsname}\def\CD@XJ#1{\expandafter\show\csname#1%
\endcsname}\def\CD@ZJ#1{\expandafter\showthe\csname#1\endcsname}\def\CD@WJ#1{%
\expandafter\showbox\csname#1\endcsname}\def\CD@tA{Commutative Diagram}\edef
\CD@kH{\string\par}\edef\CD@dC{\string\diagram}\edef\CD@HD{\string\enddiagram
}\edef\CD@EC{\string\\}\def\CD@eF{LaTeX}\ifx\@ignoretrue\CD@qK\expandafter
\CD@tG\csname if@ignore\endcsname\ignore@true\ignore@false\def\@ignoretrue{%
\global\ignore@true}\def\@ignorefalse{\global\ignore@false}\fi

\def\CD@g{{\ifnum0=`}\fi}\def\CD@wC{\ifnum0=`{\fi}}\def\catcase#1:{\ifcat
\noexpand\CD@EH#1\CD@tJ\expandafter\CD@kC\else\expandafter\CD@dJ\fi}\def
\tokcase#1:{\ifx\CD@EH#1\CD@tJ\expandafter\CD@kC\else\expandafter\CD@dJ\fi}%
\def\CD@kC#1;#2\endswitch{#1}\def\CD@dJ#1;{}\let\endswitch\relax\def\default:%
#1;#2\endswitch{#1}\ifx\at@\CD@qK\def\at@{@}\fi\edef\CD@P{\CD@o pt\CD@yC}%
\CD@RC{\CD@P>}#1>#2>{\CD@z\rTo\sp{#1}\sb{#2}\CD@z}\CD@RC{\CD@P<}#1<#2<{\CD@z
\lTo\sp{#1}\sb{#2}\CD@z}\CD@RC{\CD@P)}#1)#2){\CD@z\rTo\sp{#1}\sb{#2}\CD@z}%
\CD@RC{\CD@P(}#1(#2({\CD@z\lTo\sp{#1}\sb{#2}\CD@z}
\def\CD@O{\def\endCD{\enddiagram}\CD@RC{\CD@P A}##1A##2A{\uTo<{##1}>{##2}%
\CD@z\CD@z}\CD@RC{\CD@P V}##1V##2V{\dTo<{##1}>{##2}\CD@z\CD@z}\CD@RC{\CD@P=}{%
\CD@z\hEq\CD@z}\CD@RC{\CD@P\CD@KK}{\vEq\CD@z\CD@z}\CD@RC{\CD@P\string\vert}{%
\vEq\CD@z\CD@z}\CD@RC{\CD@P.}{\CD@z\CD@z}\let\CD@z\CD@Q}\def\CD@IE{\let\tmp
\CD@JE\ifcat A\noexpand\CD@CH\else\ifcat=\noexpand\CD@CH\else\ifcat\relax
\noexpand\CD@CH\else\let\tmp\at@\fi\fi\fi\tmp}\def\CD@JE#1{\CD@nF{tmp}{\CD@P
\string#1}\ifx\tmp\relax\def\tmp{\at@#1}\fi\tmp}\def\CD@z{}\begingroup
\aftergroup\def\aftergroup\CD@T\aftergroup{\aftergroup\def\catcode`\@\active
\aftergroup @\endgroup{\futurelet\CD@CH\CD@IE}}\def\CD@uK#1{\CD@nF{x}{tex_#1:%
D}\ifx\x\relax\else\CD@nF{#1}{x}\fi} \def\CD@vK{\CD@uK{par}\CD@uK{everypar}%
\CD@uK{noindent}\CD@uK{parskip}\CD@uK{unskip}\CD@uK{hskip}\CD@uK{indent}}

\newcount\CD@uA\newcount\CD@vA\newcount\CD@wA\newcount\CD@xA\newdimen\CD@OA
\newdimen\CD@PA\CD@tG\CD@gE\CD@@A\CD@y\CD@tG\CD@hE\CD@EA\CD@BA\newdimen\CD@RA
\newdimen\CD@SA\newcount\CD@yA\newcount\CD@zA\newdimen\CD@QA\newbox\CD@DA
\CD@tG\CD@lE\CD@dA\CD@bA\newcount\CD@LH\newcount\CD@TC\def\CD@V#1#2{\ifdim#1<%
#2\relax#1=#2\relax\fi}\def\CD@X#1#2{\ifdim#1>#2\relax#1=#2\relax\fi}%
\newdimen\CD@XH\CD@XH=1sp \newdimen\CD@zC\CD@zC\z@\def\CD@cJ{\ifdim\CD@zC=1em%
\else\CD@nJ\fi}\def\CD@nJ{\CD@zC1em\def\CD@NC{\fontdimen8\textfont3 }\CD@@J
\CD@NJ\setbox0=\vbox{\CD@t\noindent\CD@k\null\penalty-9993\null\CD@ND\null
\endgraf\setbox0=\lastbox\unskip\unpenalty\setbox1=\lastbox\global\setbox
\CD@IG=\hbox{\unhbox0\unskip\unskip\unpenalty\setbox0=\lastbox}\global\setbox
\CD@KG=\hbox{\unhbox1\unskip\unpenalty\setbox1=\lastbox}}}\newdimen\CD@@I
\CD@@I=1true in \divide\CD@@I300 \def\CD@zH#1{\multiply#1\tw@\advance#1\ifnum
#1<\z@-\else+\fi\CD@@I\divide#1\tw@\divide#1\CD@@I\multiply#1\CD@@I}\def
\MapBreadth{\afterassignment\CD@gI\CD@LF}\newdimen\CD@LF\newdimen\CD@oI\def
\CD@gI{\CD@oI\CD@LF\CD@V\CD@@I{4\CD@XH}\CD@X\CD@@I\p@\CD@zH\CD@oI\ifdim\CD@LF
>\z@\CD@V\CD@oI\CD@@I\fi\CD@cJ}\def\CD@RJ#1{\CD@zD\count@\CD@@I#1\ifnum
\count@>\z@\divide\CD@@I\count@\fi\CD@gI\CD@NJ}\def\CD@NJ{\dimen@\CD@QC
\count@\dimen@\divide\count@5\divide\count@\CD@@I\edef\CD@OC{\the\count@}}%
\def\CD@AJ{\CD@QJ\z@}\def\CD@QJ#1{\CD@tI\axisheight\advance\CD@tI#1\relax
\advance\CD@tI-.5\CD@oI\CD@zH\CD@tI\CD@sI-\CD@tI\advance\CD@tI\CD@LF}%
\newdimen\CD@DC\CD@DC\z@\newdimen\CD@eJ\CD@eJ\z@\def\CD@CJ#1{\CD@sI#1\relax
\CD@tI\CD@sI\advance\CD@tI\CD@LF\relax}\def\horizhtdp{height\CD@tI depth%
\CD@sI}\def\axisheight{\fontdimen22\the\textfont\tw@}\def\script@axisheight{%
\fontdimen22\the\scriptfont\tw@}\def\ss@axisheight{\fontdimen22\the
\scriptscriptfont\tw@}\def\CD@NC{0.4pt}\def\CD@VK{\fontdimen3\textfont\z@}%
\def\CD@UK{\fontdimen3\textfont\z@}\newdimen\PileSpacing\newdimen\CD@nA\CD@nA
\z@\def\CD@RG{\ifincommdiag1.3em\else2em\fi}\newdimen\CD@YB\def\CellSize{%
\afterassignment\CD@kB\DiagramCellHeight}\newdimen\DiagramCellHeight
\DiagramCellHeight-\maxdimen\newdimen\DiagramCellWidth\DiagramCellWidth-%
\maxdimen\def\CD@kB{\DiagramCellWidth\DiagramCellHeight}\def\CD@QC{3em}%
\newdimen\MapShortFall\def\MapsAbut{\MapShortFall\z@\objectheight\z@
\objectwidth\z@}\newdimen\CD@iA\CD@iA\z@\CD@tG\CD@vE\CD@aB\CD@ZB\expandafter
\ifx\expandafter\iftrue\csname ifUglyObsoleteDiagrams\endcsname\CD@ZB\else
\CD@aB\fi\CD@nF{ifUglyObsoleteDiagrams}{relax}\newif\ifUglyObsoleteDiagrams
\def\CD@nK{\CD@aB\UglyObsoleteDiagramsfalse}\def\CD@oK{\CD@ZB
\UglyObsoleteDiagramstrue}\CD@vE\CD@nK\else\CD@oK\fi\CD@tG\CD@hK\CD@dK\CD@cK
\CD@cK\def\CD@sK{\ifx\pdfoutput\CD@qK\else\ifx\pdfoutput\relax\else\ifnum
\pdfoutput>\z@\CD@pK\fi\fi\fi} \def\CD@pK{\global\CD@dK\global\CD@aB\global
\UglyObsoleteDiagramsfalse\global\let\CD@n\empty\global\let\CD@oK\relax
\global\let\CD@pK\relax\global\let\CD@sK\relax}\def\CD@tK#1{}\ifx\pdfliteral\CD@qK\else\ifx\pdfliteral\relax\else\let\CD@tK
\pdfliteral\fi\fi\ifx\XeTeXrevision\CD@qK\else\ifx\XeTeXrevision\relax\else
\ifdim\XeTeXrevision pt<.996pt \expandafter\message{! XeTeX version
\XeTeXrevision\space does not support PDF literals, so diagonals will not work%
!}\else\expandafter\message{RUNNING UNDER XETEX \XeTeXrevision}\CD@pK\fi\fi
\fi\CD@sK\def\newarrowhead{\CD@mG h\CD@BG\CD@GG>}\def\newarrowtail{\CD@mG t%
\CD@BG\CD@GG>}\def\newarrowmiddle{\CD@mG m\CD@BG\hbox@maths\empty}\def
\newarrowfiller{\CD@mG f\CD@bE\CD@MK-}\def\CD@mG#1#2#3#4#5#6#7#8#9{\CD@RC{r#1%
:#5}{#2{#6}}\CD@RC{l#1:#5}{#2{#7}}\CD@RC{d#1:#5}{#3{#8}}\CD@RC{u#1:#5}{#3{#9}%
}\CD@vC{-#1:#5}{\expandafter\noexpand\csname-#1:#4\endcsname\noexpand\CD@MC}%
\CD@vC{+#1:#5}{\expandafter\noexpand\csname+#1:#4\endcsname\noexpand\CD@MC}}%
\CD@ZA\CD@MC{\CD@eF\space diagonals are used unless PostScript is set}\def
\defaultarrowhead#1{\edef\CD@sJ{#1}\CD@@J}\def\CD@@J{\CD@IJ\CD@sJ<>ht\CD@IJ
\CD@sJ<>th}\def\CD@IJ#1#2#3#4#5{\CD@HJ{r#4}{#3}{l#5}{#2}{r#4:#1}\CD@HJ{r#5}{#%
2}{l#4}{#3}{l#4:#1}\CD@HJ{d#4}{#3}{u#5}{#2}{d#4:#1}\CD@HJ{d#5}{#2}{u#4}{#3}{u%
#4:#1}}\def\CD@HJ#1#2#3#4#5{\begingroup\aftergroup\CD@GJ\CD@L{#1+:#2}\CD@L{#1%
:#2}\CD@L{#3:#4}\CD@L{#5}\endgroup}\def\CD@GJ#1#2#3#4{\csname newbox%
\endcsname#1\def#2{\copy#1}\def#3{\copy#1}\setbox#1=\box\voidb@x}\def\CD@sJ{}%
\CD@@J\def\CD@GJ#1#2#3#4{\setbox#1=#4}\ifx\tenln\nullfont\def\CD@sJ{vee}\else
\let\CD@sJ\CD@eF\fi\def\CD@xF#1#2#3{\begingroup\aftergroup\CD@wF\CD@L{#1#2:#3%
#3}\CD@L{#1#2:#3}\aftergroup\CD@yF\CD@L{#1#2:#3-#3}\CD@L{#1#2:#3}\endgroup}%
\def\CD@wF#1#2{\def#1{\hbox{\rlap{#2}\kern.4\CD@zC#2}}}\def\CD@yF#1#2{\def#1{%
\hbox{\rlap{#2}\kern.4\CD@zC#2\kern-.4\CD@zC}}}\CD@xF lh>\CD@xF rt>\CD@xF rh<%
\CD@xF rt<\def\CD@yF#1#2{\def#1{\hbox{\kern-.4\CD@zC\rlap{#2}\kern.4\CD@zC#2}%
}}\CD@xF rh>\CD@xF lh<\CD@xF lt>\CD@xF lt<\def\CD@wF#1#2{\def#1{\vbox{\vbox to%
\z@{#2\vss}\nointerlineskip\kern.4\CD@zC#2}}}\def\CD@yF#1#2{\def#1{\vbox{%
\vbox to\z@{#2\vss}\nointerlineskip\kern.4\CD@zC#2\kern-.4\CD@zC}}}\CD@xF uh>%
\CD@xF dt>\CD@xF dh<\CD@xF dt<\def\CD@yF#1#2{\def#1{\vbox{\kern-.4\CD@zC\vbox
to\z@{#2\vss}\nointerlineskip\kern.4\CD@zC#2}}}\CD@xF dh>\CD@xF ut>\CD@xF uh<%
\CD@xF ut<\def\CD@BG#1{\hbox{\mathsurround\z@\offinterlineskip\CD@k\mkern-1.5%
mu{#1}\mkern-1.5mu\CD@ND}}\def\hbox@maths#1{\hbox{\CD@k#1\CD@ND}}\def\CD@GG#1%
{\hbox to\CD@LF{\setbox0=\hbox{\offinterlineskip\mathsurround\z@\CD@k{#1}%
\CD@ND}\dimen0.5\wd0\advance\dimen0-.5\CD@oI\CD@zH{\dimen0}\kern-\dimen0%
\unhbox0\hss}}\def\CD@sB#1{\hbox to2\CD@LF{\hss\offinterlineskip\mathsurround
\z@\CD@k{#1}\CD@ND\hss}}\def\CD@vF#1{\hbox{\mathsurround\z@\CD@k{#1}\CD@ND}}%
\def\CD@bE#1{\hbox{\kern-.15\CD@zC\CD@k{#1}\CD@ND\kern-.15\CD@zC}}\def\CD@MK#%
1{\vbox{\offinterlineskip\kern-.2ex\CD@GG{#1}\kern-.2ex}}\def\@fillh{%
\xleaders\vrule\horizhtdp}\def\@fillv{\xleaders\hrule width\CD@LF}\CD@nF{rf:-%
}{@fillh}\CD@nF{lf:-}{@fillh}\CD@nF{df:-}{@fillv}\CD@nF{uf:-}{@fillv}\CD@nF{%
rh:}{null}\CD@nF{rm:}{null}\CD@nF{rt:}{null}\CD@nF{lh:}{null}\CD@nF{lm:}{null%
}\CD@nF{lt:}{null}\CD@nF{dh:}{null}\CD@nF{dm:}{null}\CD@nF{dt:}{null}\CD@nF{%
uh:}{null}\CD@nF{um:}{null}\CD@nF{ut:}{null}\CD@nF{+h:}{null}\CD@nF{+m:}{null%
}\CD@nF{+t:}{null}\CD@nF{-h:}{null}\CD@nF{-m:}{null}\CD@nF{-t:}{null}\def
\CD@@D{\hbox{\vrule height 1pt depth-1pt width 1pt}}\CD@RC{rf:}{\CD@@D}\CD@nF
{lf:}{rf:}\CD@nF{+f:}{rf:}\CD@RC{df:}{\CD@@D}\CD@nF{uf:}{df:}\CD@nF{-f:}{df:}%
\def\CD@BD{\CD@U\null\CD@@D\null\CD@@D\null}\edef\CD@lG{\string\newarrow}\def
\newarrow#1#2#3#4#5#6{\begingroup\edef\@name{#1}\edef\CD@oJ{#2}\edef\CD@iD{#3%
}\edef\CD@QG{#4}\edef\CD@jD{#5}\edef\CD@LE{#6}\let\CD@HE\CD@sG\let\CD@FK
\CD@BH\let\@x\CD@AH\ifx\CD@oJ\CD@iD\let\CD@oJ\empty\fi\ifx\CD@LE\CD@jD\let
\CD@LE\empty\fi\def\CD@LI{r}\def\CD@SF{l}\def\CD@IC{d}\def\CD@yJ{u}\def\CD@gH
{+}\def\@m{-}\ifx\CD@iD\CD@jD\ifx\CD@QG\CD@iD\let\CD@QG\empty\fi\ifx\CD@LE
\empty\ifx\CD@iD\CD@aE\let\@x\CD@yG\else\let\@x\CD@zG\fi\fi\else\edef\CD@a{%
\CD@iD\CD@oJ}\ifx\CD@a\empty\ifx\CD@QG\CD@jD\let\CD@QG\empty\fi\fi\fi\ifmmode
\aftergroup\CD@kG\else\CD@@A\CD@oB rh{head\space\space}\CD@LE\CD@oB rf{filler%
}\CD@iD\CD@oB rm{middle}\CD@QG\ifx\CD@jD\CD@iD\else\CD@oB rf{filler}\CD@jD\fi
\CD@oB rt{tail\space\space}\CD@oJ\CD@gE\CD@HE\CD@FK\@x\CD@nG l-2+2{lu}{nw}%
\NorthWest\CD@nG r+2+2{ru}{ne}\NorthEast\CD@nG l-2-2{ld}{sw}\SouthWest\CD@nG r%
+2-2{rd}{se}\SouthEast\else\aftergroup\CD@b\CD@L{r\@name}\fi\fi\endgroup}\def
\CD@sG{\CD@vG\CD@LI\CD@SF rl\Horizontal@Map}\def\CD@BH{\CD@vG\CD@IC\CD@yJ du%
\Vertical@Map}\def\CD@AH{\CD@vG\CD@gH\@m+-\Vector@Map}\def\CD@yG{\CD@vG\CD@gH
\@m+-\Slant@Map}\def\CD@zG{\CD@vG\CD@gH\@m+-\Slope@Map}\catcode`\/=\active
\def\CD@vG#1#2#3#4#5{\CD@jG#1#3#5t:\CD@oJ/f:\CD@iD/m:\CD@QG/f:\CD@jD/h:\CD@LE
//\CD@jG#2#4#5h:\CD@LE/f:\CD@jD/m:\CD@QG/f:\CD@iD/t:\CD@oJ//}\def\CD@jG#1#2#3%
#4//{\edef\CD@fG{#2}\aftergroup\sdef\CD@L{#1\@name}\aftergroup{\aftergroup#3%
\CD@M#4//\aftergroup}}\def\CD@M#1/{\edef\CD@EH{#1}\ifx\CD@EH\empty\else\CD@L{%
\CD@fG#1}\expandafter\CD@M\fi}\catcode`\/=12 \def\CD@nG#1#2#3#4#5#6#7#8{%
\aftergroup\sdef\CD@L{#6\@name}\aftergroup{\CD@L{#2\@name}\if#2#4\aftergroup
\CD@CI\else\aftergroup\CD@BI\fi\CD@L{#1\@name}%
\aftergroup(\aftergroup#3\aftergroup,\aftergroup#5\aftergroup)\aftergroup}}%
\def\CD@oB#1#2#3#4{\expandafter\ifx\csname#1#2:#4\endcsname\relax\CD@y\CD@gB{%
arrow#3 "#4" undefined}\fi}\CD@rG\CD@VE{All five components must be defined
before an arrow.}\CD@rG\CD@SE{\CD@lG, unlike \string\HorizontalMap, is a
declaration.}\def\CD@b#1{\CD@YA{Arrows \string#1 etc could not be defined}%
\CD@VE}\def\CD@kG{\CD@YA{misplaced \CD@lG}\CD@SE}\def\newdiagramgrid#1#2#3{%
\CD@RC{cdgh@#1}{#2,],}
\CD@RC{cdgv@#1}{#3,],}}
\CD@tG\ifincommdiag\incommdiagtrue\incommdiagfalse\CD@tG\CD@@F\CD@IF\CD@HF
\newcount\CD@VA\CD@VA=0 \def\CD@yH{\CD@VA6 }\def\CD@OB{\CD@VA1 \global\CD@yA1
\CD@DE\CD@YF\empty}\def\CD@YF{}\def\CD@nB#1{\relax\CD@MD\edef\CD@vJ{#1}%
\begingroup\CD@rE\else\ifcase\CD@VA\ifmmode\else\CD@YG\CD@E0\fi\or\CD@cE5\or
\CD@YG\CD@F5\or\CD@YG\CD@B5\or\CD@YG\CD@B5\or\CD@YG\CD@C5\or\CD@cE7\or\CD@YG
\CD@D7\fi\fi\endgroup\xdef\CD@YF{#1}}\def\CD@pB#1#2#3#4#5{\relax\CD@MD\xdef
\CD@vJ{#4}\begingroup\ifnum\CD@VA<#1 \expandafter\CD@cE\ifcase\CD@VA0\or#2\or
#3\else#2\fi\else\ifnum\CD@VA<6 \CD@tJ\CD@YG\CD@B#2\else\CD@YG\CD@G#2\fi\fi
\endgroup\CD@DE\CD@YF\CD@vJ\ifincommdiag\let\CD@ZD#5\else\let\CD@ZD\CD@LK\fi}%
\def\CD@yI{\global\CD@yA=\ifnum\CD@VA<5 1\else2\fi\relax}\def\CD@OI{\CD@VA
\CD@yA}\def\CD@cE#1{\aftergroup\CD@VA\aftergroup#1\aftergroup\relax}\def
\CD@HH{\def\CD@nB##1{\relax}\let\CD@pB\CD@FH\let\CD@yH\relax\let\CD@OB\relax
\let\CD@yI\relax\let\CD@OI\relax}\def\CD@FH#1#2#3#4#5{\ifincommdiag\let\CD@ZD
#5\else\xdef\CD@vJ{#4}\let\CD@ZD\CD@LK\fi}\def\CD@YG#1{\aftergroup#1%
\aftergroup\relax\CD@cE}\def\CD@B{\CD@YE\CD@S\CD@ME\CD@Q}\def\CD@G{\CD@YE{%
\CD@yC\CD@S}\CD@XE\CD@QD\CD@Q}\def\CD@F{\CD@YE{*\CD@S}\CD@RE\clubsuit\CD@Q}%
\def\CD@C{\CD@YE{\CD@S*\CD@S}\CD@RE\CD@Q\clubsuit\CD@Q}\def\CD@D{\CD@YE\CD@EC
\CD@TE\\}\def\CD@E{\CD@YE\CD@nC\CD@QE\CD@k}\def\CD@LK{\CD@YA{\CD@vJ\space
ignored \CD@dH}\CD@WE}\def\CD@FE{}\def\CD@d{\CD@YA{maps must never be enclosed
in braces}\CD@OE}\def\CD@dH{outside diagram}\def\CD@FC{\string\HonV, \string
\VonH\space and \string\HmeetV}\CD@rG\CD@ME{The way that horizontal and
vertical arrows are terminated implicitly means\CD@uG that they cannot be
mixed with each other or with \CD@FC.}\CD@rG\CD@XE{\string\pile\space is for
parallel horizontal arrows; verticals can just be put together in\CD@uG a cell%
. \CD@FC\space are not meaningful in a \string\pile.}\CD@rG\CD@RE{The
horizontal maps must point to an object, not each other (I've put in\CD@uG one
which you're unlikely to want). Use \string\pile\space if you want them
parallel.}\CD@rG\CD@TE{Parallel horizontal arrows must be in separate layers
of a \string\pile.}\CD@rG\CD@QE{Horizontal arrows may be used \CD@dH s, but
must still be in maths.}\CD@rG\CD@WE{Vertical arrows, \CD@FC\space\CD@dH s don%
't know where\CD@uG where to terminate.}\CD@rG\CD@OE{This prevents them from
stretching correctly.}\def\CD@YE#1{\CD@YA{"#1" inserted \ifx\CD@YF\empty
before \CD@vJ\else between \CD@YF\ifx\CD@YF\CD@vJ s\else\space and \CD@vJ\fi
\fi}}\count@=\year\multiply\count@12 \advance\count@\month\ifnum\count@>24391
\expandafter\endinput\fi\def
\Horizontal@Map{\CD@nB{horizontal map}\CD@LC\CD@TJ\CD@qD}\def\CD@TJ{\CD@GB-%
9999 \let\CD@ZD\CD@XD\ifincommdiag\else\CD@cJ\ifinpile\else\skip2\z@ plus 1.5%
\CD@VK minus .5\CD@UK\skip4\skip2 \fi\fi\let\CD@kD\@fillh\CD@nF{fill@dot}{rf:%
.}}\def\Vector@Map{\CD@HK4}\def\Slant@Map{\CD@HK{\CD@EF255\else6\fi}}\def
\Slope@Map{\CD@HK\CD@OC}\def\CD@HK#1#2#3#4#5#6{\CD@LC\def\CD@WK{2}\def\CD@aK{%
2}\def\CD@ZK{1}\def\CD@bK{1}\let\Horizontal@Map\CD@nI\def\CD@OG{#1}\def\CD@NI
{\CD@U#2#3#4#5#6}}\def\CD@nI{\CD@TJ\CD@JB\let\CD@ZD\CD@TD\CD@qD}\CD@tG\CD@pE
\CD@rA\CD@qA\CD@rA\def\cds@missives{\CD@rA}\def\CD@TD{\CD@vE\let\CD@OG\CD@OC
\CD@x\CD@zE\CD@WF\fi\setbox0\hbox{\incommdiagfalse\CD@HI}\CD@pE\CD@aD\else
\global\CD@YC\CD@bD\fi\ifvoid6 \ifvoid7 \CD@eE\fi\fi\CD@zE\else\CD@BD\global
\CD@YC\let\CD@CG\CD@IH\CD@YD\fi\else\CD@NI\CD@MI\global\CD@YC\CD@YD\fi}\def
\CD@LC{\begingroup\dimen1=\MapShortFall\dimen2=\CD@RG\dimen5=\MapShortFall
\setbox3=\box\voidb@x\setbox6=\box\voidb@x\setbox7=\box\voidb@x\CD@pD
\mathsurround\z@\skip2\z@ plus1fill minus 1000pt\skip4\skip2 \CD@TB}\CD@tG
\CD@tE\CD@UB\CD@TB\def\CD@U#1#2#3#4#5{\let\CD@oJ#1\let\CD@iD#2\let\CD@QG#3%
\let\CD@jD#4\let\CD@LE#5\CD@TB\ifx\CD@iD\CD@jD\CD@UB\fi}\def\CD@qD#1#2#3#4#5{%
\CD@U#1#2#3#4#5\CD@tD}\def\Vertical@Map{\CD@pB433{vertical map}\CD@cD\CD@LC
\CD@GB-9995 \let\CD@kD\@fillv\CD@nF{fill@dot}{df:.}\CD@qD}\def\break@args{%
\def\CD@tD{\CD@ZD}\CD@ZD\endgroup\aftergroup\CD@FE}\def\CD@MJ{\setbox1=\CD@oJ
\setbox5=\CD@LE\ifvoid3 \ifx\CD@QG\null\else\setbox3=\CD@QG\fi\fi\CD@@G2%
\CD@iD\CD@@G4\CD@jD}\def\CD@pF#1{\ifvoid1\else\CD@oF1#1\fi\ifvoid2\else\CD@oF
2#1\fi\ifvoid3\else\CD@oF3#1\fi\ifvoid4\else\CD@oF4#1\fi\ifvoid5\else\CD@oF5#%
1\fi} \def\CD@oF#1#2{\setbox#1\vbox{\offinterlineskip\box#1\dimen@\prevdepth
\advance\dimen@-#2\relax\setbox0\null\dp0\dimen@\ht0-\dimen@\box0}}\def\CD@@G
#1#2{\ifx#2\CD@kD\setbox#1=\box\voidb@x\else\setbox#1=#2\def#2{\xleaders\box#%
1}\fi}\CD@ZA\CD@BK{\string\HorizontalMap, \string\VerticalMap\space and
\string\DiagonalMap\CD@uG are obsolete - use \CD@lG\space to pre-define maps}%
\def\HorizontalMap#1#2#3#4#5{\CD@BK\CD@nB{old horizontal map}\CD@LC\CD@TJ\def
\CD@oJ{\CD@UH{#1}}\CD@SH\CD@iD{#2}\def\CD@QG{\CD@UH{#3}}\CD@SH\CD@jD{#4}\def
\CD@LE{\CD@UH{#5}}\CD@tD}\def\VerticalMap#1#2#3#4#5{\CD@BK\CD@pB433{vertical
map}\CD@cD\CD@LC\CD@GB-9995 \let\CD@kD\@fillv\def\CD@oJ{\CD@GG{#1}}\CD@VH
\CD@iD{#2}\def\CD@QG{\CD@GG{#3}}\CD@VH\CD@jD{#4}\def\CD@LE{\CD@GG{#5}}\CD@tD}%
\def\DiagonalMap#1#2#3#4#5{\CD@BK\CD@LC\def\CD@OG{4}\let\CD@kD\CD@qK\let
\CD@ZD\CD@YD\def\CD@WK{2}\def\CD@aK{2}\def\CD@ZK{1}\def\CD@bK{1}\def\CD@QG{%
\CD@vF{#3}}\ifPositiveGradient\let\mv\raise\def\CD@oJ{\CD@vF{#5}}\def\CD@iD{%
\CD@vF{#4}}\def\CD@jD{\CD@vF{#2}}\def\CD@LE{\CD@vF{#1}}\else\let\mv\lower\def
\CD@oJ{\CD@vF{#1}}\def\CD@iD{\CD@vF{#2}}\def\CD@jD{\CD@vF{#4}}\def\CD@LE{%
\CD@vF{#5}}\fi\CD@tD}\def\CD@aE{-}\def\CD@AD{\empty}\def\CD@SH{\CD@EG\CD@bE
\CD@aE\@fillh}\def\CD@VH{\CD@EG\CD@MK\CD@KK\@fillv}\def\CD@EG#1#2#3#4#5{\def
\CD@CH{#5}\ifx\CD@CH#2\let#4#3\else\let#4\null\ifx\CD@CH\empty\else\ifx\CD@CH
\CD@AD\else\let#4\CD@CH\fi\fi\fi}\def\CD@UH#1{\hbox{\mathsurround\z@
\offinterlineskip\def\CD@CH{#1}\ifx\CD@CH\empty\else\ifx\CD@CH\CD@AD\else
\CD@k\mkern-1.5mu{\CD@CH}\mkern-1.5mu\CD@ND\fi\fi}}\def\CD@yD#1#2{\setbox#1=%
\hbox\bgroup\setbox0=\hbox{\CD@k\labelstyle()\CD@ND}
\setbox1=\null\ht1\ht0\dp1\dp0\box1 \kern.1\CD@zC\CD@k\bgroup\labelstyle
\aftergroup\CD@LD\CD@xD}\def\CD@LD{\CD@ND\kern.1\CD@zC\egroup\CD@tD}\def
\CD@xD{\futurelet\CD@EH\CD@mJ}\def\CD@mJ{
\catcase\bgroup:\CD@v;\catcase\egroup:\missing@label;\catcase\space:\CD@TF;%
\tokcase[:\CD@XF;
\default:\CD@zJ;\endswitch}\def\CD@v{\let\CD@MD\CD@c\let\CD@CH}\def\CD@zJ#1{%
\let\CD@UF\egroup{\let\actually@braces@missing@around@macro@in@label\CD@ZH
\let\CD@MD\CD@xC\let\CD@UF\CD@VF#1%
\actually@braces@missing@around@macro@in@label}\CD@UF}\def
\actually@braces@missing@around@macro@in@label{\let\CD@CH=}\def\missing@label
{\egroup\CD@YA{missing label}\CD@PE}\def\CD@xC{\egroup\missing@label}\outer
\def\CD@ZH{}\def\CD@UF{}\def\CD@VF{\CD@wC\CD@UF}\def\CD@MD{}\def\CD@XF{\let
\CD@N\CD@xD\get@square@arg\CD@AE}\CD@rG\CD@PE{The text which has just been
read is not allowed within map labels.}\def\CD@c{\egroup\CD@YA{missing \CD@yC
\space inserted after label}\CD@PE}\def\upper@label{\CD@oD\CD@yD6}\def
\lower@label{\def\positional@{\CD@@A\break@args}\CD@yD7}\def\middle@label{%
\CD@yD3}\CD@tG\CD@yE\CD@pD\CD@oD\def\CD@iF{\ifPositiveGradient\CD@tJ
\expandafter\upper@label\else\expandafter\lower@label\fi}\def\CD@iI{%
\ifPositiveGradient\CD@tJ\expandafter\lower@label\else\expandafter
\upper@label\fi}\def\positional@{\CD@gB{labels as positional arguments are
obsolete}\CD@yE\CD@tJ\expandafter\upper@label\else\expandafter\lower@label\fi
-}\def\CD@tD{\futurelet\CD@EH\switch@arg}\def\eat@space{\afterassignment
\CD@tD\let\CD@EH= }\def\CD@TF{\afterassignment\CD@xD\let\CD@EH= }\def\CD@BC{%
\get@round@pair\CD@uD}\def\CD@uD#1#2{\def\CD@WK{#1}\def\CD@aK{#2}\CD@tD}\def
\optional@{\let\CD@N\CD@tD\get@square@arg\CD@AE}\def\CD@JJ.{\CD@sC\CD@tD}\def
\CD@sC{\let\CD@iD\fill@dot\let\CD@jD\fill@dot\def\CD@MI{\let\CD@iD\dfdot\let
\CD@jD\dfdot}}\def\CD@MI{}\def\CD@@E#1,{\CD@nH#1,\begingroup\ifx\@name\CD@RD
\CD@FF\aftergroup\CD@e\fi\aftergroup\CD@jC\else\expandafter\def\expandafter
\CD@RF\expandafter{\csname\@name\endcsname}\expandafter\CD@vD\CD@RF\CD@KD\ifx
\CD@RF\empty\aftergroup\CD@pC\expandafter\aftergroup\csname\CD@FB\@name
\endcsname\expandafter\aftergroup\csname\CD@FB @\@name\endcsname\else\gdef
\CD@GE{#1}\CD@gB{\string\relax\space inserted before `[\CD@GE'}\message{(I was
trying to read this as a \CD@tA\ option.)}\aftergroup\CD@H\fi\fi\endgroup}%
\def\CD@vD#1#2\CD@KD{\def\CD@RF{#2}}\def\CD@jC{\let\CD@CH\CD@N\let\CD@N\relax
\CD@CH}\def\CD@H#1],{
\CD@jC\relax\def\CD@RF{#1}\ifx\CD@RF\empty\def\CD@RF{[\CD@GE]}%
\else\def\CD@RF{[\CD@GE,#1]}
\fi\CD@RF}\def\CD@pC#1#2{\ifx#2\CD@qK\ifx#1\CD@qK\CD@gB{option `\@name'
undefined}\else#1\fi\else\CD@FF\expandafter#2\CD@GK\CD@PK\else\CD@QK\fi\fi
\CD@DH}\CD@tG\CD@FF\CD@QK\CD@PK\def\CD@nH#1,{\CD@FF\ifx\CD@GK\CD@qK\CD@e\else
\expandafter\CD@oH\CD@GK,#1,(,),(,)[]%
\fi\fi\CD@FF\else\CD@mH#1==,\fi}\def\CD@e{\CD@gB{option `\@name' needs (x,y)
value}\CD@PK\let\@name\empty}\def\CD@mH#1=#2=#3,{\def\@name{#1}\def\CD@GK{#2}%
\def\CD@RF{#3}\ifx\CD@RF\empty\let\CD@GK\CD@qK\fi}%
\def\CD@oH#1(#2,#3)#4,(#5,#6)#7[]{\def\CD@GK{{#2}{#3}}\def\CD@RF{#1#4#5#6}%
\ifx\CD@RF\empty\def\CD@RF{#7}\ifx\CD@RF\empty\CD@e\fi\else\CD@e\fi}\def
\CD@FB{cds@}\let\CD@N\relax\def\CD@zD#1{\ifx\CD@GK\CD@qK\CD@gB{option `\@name
' needs a value}\else#1\CD@GK\relax\fi}\def\CD@BE#1#2{\ifx\CD@GK\CD@qK#1#2%
\relax\else#1\CD@GK\relax\fi}\def\cds@@showpair#1#2{\message{x=#1,y=#2}}\def
\cds@@diagonalbase#1#2{\edef\CD@ZK{#1}\edef\CD@bK{#2}}\def\CD@DI#1{\def\CD@CH
{#1}\CD@nF{@x}{cdps@#1}\ifx\CD@CH\empty\CD@f\CD@CH{cannot be used}\else\ifx
\CD@CH\relax\CD@f\CD@CH{unknown}\else\let\CD@IK\@x\fi\fi}\def\CD@f#1#2{\CD@gB
{PostScript translator `#1' #2}}\def\CD@PH{}\def\CD@PJ{\CD@fA\edef\CD@PH{%
\noexpand\CD@KB{\@name\space ignored within maths}}}\def\diagramstyle{\CD@cJ
\let\CD@N\relax\CD@CF\CD@AE\CD@AE}\CD@tG\CD@sE
\CD@SB\CD@RB\CD@tG\CD@qE\CD@EB\CD@DB\CD@tG\CD@oE\CD@pA\CD@oA\CD@tG\CD@iE
\CD@HA\CD@GA\CD@HA\CD@tG\CD@jE\CD@JA\CD@IA\CD@tG\CD@kE\CD@LA\CD@KA\CD@tG
\CD@EF\CD@DK\CD@CK\CD@tG\CD@rE\CD@JB\CD@IB\CD@tG\CD@mE\CD@gA\CD@fA\CD@tG
\CD@nE\CD@kA\CD@jA\CD@tG\CD@AF\CD@iG\CD@hG\CD@RC{cds@ }{}\CD@RC{cds@}{}\CD@RC
{cds@1em}{\CellSize1\CD@zC}\CD@RC{cds@1.5em}{\CellSize1.5\CD@zC}\CD@RC{cds@2%
em}{\CellSize2\CD@zC}\CD@RC{cds@2.5em}{\CellSize2.5\CD@zC}\CD@RC{cds@3em}{%
\CellSize3\CD@zC}\CD@RC{cds@3.5em}{\CellSize3.5\CD@zC}\CD@RC{cds@4em}{%
\CellSize4\CD@zC}\CD@RC{cds@4.5em}{\CellSize4.5\CD@zC}\CD@RC{cds@5em}{%
\CellSize5\CD@zC}\CD@RC{cds@6em}{\CellSize6\CD@zC}\CD@RC{cds@7em}{\CellSize7%
\CD@zC}\CD@RC{cds@8em}{\CellSize8\CD@zC}\def\cds@abut{\MapsAbut\dimen1\z@
\dimen5\z@}\def\cds@alignlabels{\CD@IA\CD@KA}\def\cds@amstex{\ifincommdiag
\CD@O\else\def\CD{\diagram[amstex]}
\fi\CD@T\catcode`\@\active}\def\cds@b{\let\CD@dB\CD@bB}\def\cds@balance{\let
\CD@hA\CD@AA}\let\cds@bottom\cds@b\def\cds@center{\cds@vcentre\cds@nobalance}%
\let\cds@centre\cds@center\def\cds@centerdisplay{\CD@HA\CD@PJ\cds@balance}%
\let\cds@centredisplay\cds@centerdisplay\def\cds@crab{\CD@BE\CD@DC{.5%
\PileSpacing}}\CD@RC{cds@crab-}{\CD@DC-.5\PileSpacing}\CD@RC{cds@crab+}{%
\CD@DC.5\PileSpacing}\CD@RC{cds@crab++}{\CD@DC1.5\PileSpacing}\CD@RC{cds@crab%
--}{\CD@DC-1.5\PileSpacing}\def\cds@defaultsize{\CD@BE{\let\CD@QC}{3em}\CD@NJ
}\def\cds@displayoneliner{\CD@DB}\let\cds@dotted\CD@sC\def\cds@dpi{\CD@RJ{1%
truein}}\def\cds@dpm{\CD@RJ{100truecm}}\let\CD@XA\CD@qK\def\cds@eqno{\let
\CD@XA\CD@GK\let\CD@EJ\empty}\def\cds@fixed{\CD@qA}\CD@tG\CD@fE\CD@J\CD@I\def
\cds@flushleft{\CD@I\CD@GA\CD@PJ\cds@nobalance\CD@BE\CD@nA\CD@nA}\def\cds@gap
{\CD@AJ\setbox3=\null\ht3=\CD@tI\dp3=\CD@sI\CD@BE{\wd3=}\MapShortFall} \def
\cds@grid{\ifx\CD@GK\CD@qK\let\h@grid\relax\let\v@grid\relax\else\CD@nF{%
h@grid}{cdgh@\CD@GK}\CD@nF{v@grid}{cdgv@\CD@GK}\ifx\h@grid\relax\CD@gB{%
unknown grid `\CD@GK'}\else\CD@WB\fi\fi}\let\h@grid\relax\let\v@grid\relax
\def\cds@gridx{\ifx\CD@GK\CD@qK\else\cds@grid\fi\let\CD@CH\h@grid\let\h@grid
\v@grid\let\v@grid\CD@CH}\def\cds@h{\CD@zD\DiagramCellHeight}\def\cds@hcenter
{\let\CD@hA\CD@aA}\let\cds@hcentre\cds@hcenter\def\cds@heads{\CD@BE{\let
\CD@sJ}\CD@sJ\CD@@J\CD@vE\else\ifx\CD@sJ\CD@eF\else\CD@MC\fi\fi}\let
\cds@height\cds@h\let\cds@hmiddle\cds@balance\def\cds@htriangleheight{\CD@BE
\DiagramCellHeight\DiagramCellHeight\DiagramCellWidth1.73205%
\DiagramCellHeight}\def\cds@htrianglewidth{\CD@BE\DiagramCellWidth
\DiagramCellWidth\DiagramCellHeight.57735\DiagramCellWidth}\CD@tG\CD@zE\CD@eE
\CD@dE\CD@eE\def\cds@hug{\CD@eE} \def\cds@inline{\CD@gA\let\CD@PH\empty}\def
\cds@inlineoneliner{\CD@EB}\CD@RC{cds@l>}{\CD@zD{\let\CD@RG}\dimen2=\CD@RG}%
\def\cds@labelstyle{\CD@zD{\let\labelstyle}}\def\cds@landscape{\CD@kA}\def
\cds@large{\CellSize5\CD@zC}\let\CD@EJ\empty\def\CD@FJ{\refstepcounter{%
equation}\def\CD@XA{\hbox{\@eqnnum}}}\def\cds@LaTeXeqno{\let\CD@EJ\CD@FJ}\def
\cds@lefteqno{\CD@pA}\def\cds@leftflush{\cds@flushleft\CD@J}\def
\cds@leftshortfall{\CD@zD{\dimen1 }}\def\cds@lowershortfall{%
\ifPositiveGradient\cds@leftshortfall\else\cds@rightshortfall\fi}\def
\cds@loose{\CD@VB}\def\cds@midhshaft{\CD@JA}\def\cds@midshaft{\CD@JA}\def
\cds@midvshaft{\CD@LA}\def\cds@moreoptions{\CD@@A}\let\cds@nobalance
\cds@hcenter\def\cds@nohcheck{\CD@HH}\def\cds@nohug{\CD@dE} \def
\cds@nooptions{\def\CD@aC{\CD@WD}}\let\cds@noorigin\cds@nobalance\def
\cds@nopixel{\CD@@I4\CD@XH\CD@cJ}\def\cds@UO{\CD@oK\global\let\CD@n\empty}%
\def\cds@UglyObsolete{\cds@UO\let\cds@PS\empty}\def\CD@rK#1{\CD@gB{option `#1%
' renamed as `UglyObsolete'}}\def\cds@noPostScript{\CD@rK{noPostScript}}\def
\cds@noPS{\CD@rK{noPostScript}}\def\cds@notextflow{\CD@RB}\def\cds@noTPIC{%
\CD@CK}\def\cds@objectstyle{\CD@zD{\let\objectstyle}}\def\cds@origin{\let
\CD@hA\CD@iB}\def\cds@p{\CD@zD\PileSpacing}\let\cds@pilespacing\cds@p\def
\cds@pixelsize{\CD@zD\CD@@I\CD@gI}\def\cds@portrait{\CD@jA}\def
\cds@PostScript{\CD@nK\global\let\CD@n\empty\CD@BE\CD@DI\empty}\def\cds@PS{%
\CD@nK\global\let\CD@n\empty}\CD@GF\CD@n{\typeout{\CD@tA: try the PostScript
option for better results}}\def\cds@repositionpullbacks{\let\make@pbk\CD@fH
\let\CD@qH\CD@pH}\def\cds@righteqno{\CD@oA}\def\cds@rightshortfall{\CD@zD{%
\dimen5 }}\def\cds@ruleaxis{\CD@zD{\let\axisheight}}\def\cds@cmex{\let\CD@GG
\CD@sB\let\CD@QJ\CD@CJ}\def\cds@s{\cds@height\DiagramCellWidth
\DiagramCellHeight}\def\cds@scriptlabels{\let\labelstyle\scriptstyle}\def
\cds@shortfall{\CD@zD\MapShortFall\dimen1\MapShortFall\dimen5\MapShortFall}%
\def\cds@showfirstpass{\CD@BE{\let\CD@nD}\z@}\def\cds@silent{\def\CD@KB##1{}%
\def\CD@gB##1{}}\let\cds@size\cds@s\def\cds@small{\CellSize2\CD@zC}\def
\cds@snake{\CD@BE\CD@eJ\z@}\def\cds@t{\let\CD@dB\CD@fB}\def\cds@textflow{%
\CD@SB\CD@PJ}\def\cds@thick{\let\CD@rF\tenlnw\CD@LF\CD@NC\CD@BE\MapBreadth{2%
\CD@LF}\CD@@J}\def\cds@thin{\let\CD@rF\tenln\CD@BE\MapBreadth{\CD@NC}\CD@@J}%
\def\cds@tight{\CD@WB}\let\cds@top\cds@t\def\cds@TPIC{\CD@DK}\def
\cds@uppershortfall{\ifPositiveGradient\cds@rightshortfall\else
\cds@leftshortfall\fi}\def\cds@vcenter{\let\CD@dB\CD@cB}\let\cds@vcentre
\cds@vcenter\def\cds@vtriangleheight{\CD@BE\DiagramCellHeight
\DiagramCellHeight\DiagramCellWidth.577035\DiagramCellHeight}\def
\cds@vtrianglewidth{\CD@BE\DiagramCellWidth\DiagramCellWidth
\DiagramCellHeight1.73205\DiagramCellWidth}\def\cds@vmiddle{\let\CD@dB\CD@eB}%
\def\cds@w{\CD@zD\DiagramCellWidth}\let\cds@width\cds@w\def\diagram{\relax
\protect\CD@bC}\def\enddiagram{\protect\CD@SG}\def\CD@bC{\CD@g\CD@uI
\incommdiagtrue\edef\CD@wI{\the\CD@NB}\global\CD@NB\z@\boxmaxdepth\maxdimen
\everycr{}\CD@sK\everymath{}\everyhbox{}\ifx\pdfsyncstop\CD@qK\else
\pdfsyncstop\fi\CD@aC}\def\CD@aC{\CD@y\let\CD@N\CD@ZC\CD@CF\CD@AE\CD@WD}\def
\CD@ZC{\CD@gE\expandafter\CD@aC\else\expandafter\CD@WD\fi}\def\CD@WD{\let
\CD@EH\relax\CD@nE\CD@vE\else\CD@hK\else\CD@KB{landscape ignored without
PostScript}\CD@jA\fi\fi\fi\CD@EJ\setbox2=\vbox\bgroup\CD@JF\CD@VD}\def\CD@cH{%
\CD@nE\CD@fB\else\CD@dB\fi\CD@hA\nointerlineskip\setbox0=\null\ht0-\CD@pI\dp0%
\CD@pI\wd0\CD@kI\box0 \global\CD@QA\CD@kF\global\CD@yA\CD@XB\ifx\CD@NK\CD@qK
\global\CD@RA\CD@kF\else\global\CD@RA\CD@NK\fi\egroup\CD@zF\CD@nE\setbox2=%
\hbox to\dp2{\vrule height\wd2 depth\CD@QA width\z@\global\CD@QA\ht2\ht2\z@
\dp2\z@\wd2\z@\CD@hK\CD@tK{q 0 1 -1 0 0 0 cm}\else\global\CD@iG\CD@IK{0 1
bturn}\fi\box2\CD@gK\hss}\CD@DB\fi\ifnum\CD@yA=1 \else\CD@DB\fi\global
\@ignorefalse\CD@mE\leavevmode\fi\ifvmode\CD@TA\else\ifmmode\CD@PH\CD@GI\else
\CD@qE\CD@gA\fi\ifinner\CD@gA\fi\CD@mE\CD@GI\else\CD@sE\CD@QB\else\CD@TA\fi
\fi\fi\fi\CD@dD}\def\CD@dD{\global\CD@NB\CD@wI\relax\CD@xE\global\CD@ID\else
\aftergroup\CD@mC\fi\if@ignore\aftergroup\ignorespaces\fi\CD@wC\ignorespaces}%
\def\CD@fB{\advance\CD@pI\dimen1\relax}\def\CD@eB{\advance\CD@pI.5\dimen1%
\relax}\def\CD@bB{}\def\CD@cB{\CD@fB\advance\CD@pI\CD@YB\divide\CD@pI2
\advance\CD@pI-\axisheight\relax}\def\CD@aA{}\def\CD@iB{\CD@kF\z@}\def\CD@AA{%
\ifdim\dimen2>\CD@kF\CD@kF\dimen2 \else\dimen2\CD@kF\CD@kI\dimen0 \advance
\CD@kI\dimen2 \fi}\def\CD@QB{\skip0\z@\relax\loop\skip1\lastskip\ifdim\skip1>%
\z@\unskip\advance\skip0\skip1 \repeat\vadjust{\prevdepth\dp\strutbox\penalty
\predisplaypenalty\vskip\abovedisplayskip\CD@UA\penalty\postdisplaypenalty
\vskip\belowdisplayskip}\ifdim\skip0=\z@\else\hskip\skip0 \global\@ignoretrue
\fi}\def\CD@TA{\CD@LG\kern-\displayindent\CD@UA\CD@LG\global\@ignoretrue}\def
\CD@UA{\hbox to\hsize{\CD@fE\ifdim\CD@RA=\z@\else\advance\CD@QA-\CD@RA\setbox
2=\hbox{\kern\CD@RA\box2}\fi\fi\setbox1=\hbox{\ifx\CD@XA\CD@qK\else\CD@k
\CD@XA\CD@ND\fi}\CD@oE\CD@iE\else\advance\CD@QA\wd1 \fi\wd1\z@\box1 \fi\dimen
0\wd2 \advance\dimen0\wd1 \advance\dimen0-\hsize\ifdim\dimen0>-\CD@nA\CD@HA
\fi\advance\dimen0\CD@QA\ifdim\dimen0>\z@\CD@KB{wider than the page by \the
\dimen0 }\CD@HA\fi\CD@iE\hss\else\CD@V\CD@QA\CD@nA\fi\CD@GI\hss\kern-\wd1\box
1 }}\def\CD@GI{\CD@AF\CD@@F\else\CD@SC\global\CD@hG\fi\fi\kern\CD@QA\box2 }%
\CD@tG\CD@wE\CD@YC\CD@XC\def\CD@JF{\CD@cJ\ifdim\DiagramCellHeight=-\maxdimen
\DiagramCellHeight\CD@QC\fi\ifdim\DiagramCellWidth=-\maxdimen
\DiagramCellWidth\CD@QC\fi\global\CD@XC\CD@IF\let\CD@FE\empty\let\CD@z\CD@Q
\let\overprint\CD@eH\let\CD@s\CD@rJ\let\enddiagram\CD@ED\let\\\CD@cC\let\par
\CD@jH\let\CD@MD\empty\let\switch@arg\CD@PB\let\shift\CD@iA\baselineskip
\DiagramCellHeight\lineskip\z@\lineskiplimit\z@\mathsurround\z@\tabskip\z@
\CD@OB}\def\CD@VD{\penalty-123 \begingroup\CD@jA\aftergroup\CD@K\halign
\bgroup\global\advance\CD@NB1 \vadjust{\penalty1}\global\CD@FA\z@\CD@OB\CD@j#%
#\CD@DD\CD@Q\CD@Q\CD@OI\CD@j##\CD@DD\cr}\def\CD@ED{\CD@MD\CD@GD\crcr\egroup
\global\CD@JD\endgroup}\def\CD@j{\global\advance\CD@FA1 \futurelet\CD@EH\CD@i
}\def\CD@i{\ifx\CD@EH\CD@DD\CD@tJ\hskip1sp plus 1fil \relax\let\CD@DD\relax
\CD@vI\else\hfil\CD@k\objectstyle\let\CD@FE\CD@d\fi}\def\CD@DD{\CD@MD\relax
\CD@yI\CD@vI\global\CD@QA\CD@iA\penalty-9993 \CD@ND\hfil\null\kern-2\CD@QA
\null}\def\CD@cC{\cr}\def\across#1{\span\omit\mscount=#1 \global\advance
\CD@FA\mscount\global\advance\CD@FA\m@ne\CD@sF\ifnum\mscount>2 \CD@fJ\repeat
\ignorespaces}\def\CD@fJ{\relax\span\omit\advance\mscount\m@ne}\def\CD@qJ{%
\ifincommdiag\ifx\CD@iD\@fillh\ifx\CD@jD\@fillh\ifdim\dimen3>\z@\else\ifdim
\dimen2>93\CD@@I\ifdim\dimen2>18\p@\ifdim\CD@LF>\z@\count@\CD@bJ\advance
\count@\m@ne\ifnum\count@<\z@\count@20\let\CD@aJ\CD@uJ\fi\xdef\CD@bJ{\the
\count@}\fi\fi\fi\fi\fi\fi\fi}\def\CD@cG#1{\vrule\horizhtdp width#1\dimen@
\kern2\dimen@}\def\CD@uJ{\rlap{\dimen@\CD@@I\CD@V\dimen@{.182\p@}\CD@zH
\dimen@\advance\CD@tI\dimen@\CD@cG0\CD@cG0\CD@cG2\CD@cG6\CD@cG6\CD@cG2\CD@cG0%
\CD@cG0\CD@cG2\CD@cG6\CD@cG0\CD@cG0\CD@cG2\CD@cG2\CD@cG6\CD@cG0\CD@cG0\CD@cG2%
\CD@cG6\CD@cG2\CD@cG2\CD@cG0\CD@cG0}}\def\CD@bJ{10}\def\CD@aJ{}\def\CD@XD{%
\CD@gE\CD@TB\fi\CD@x\CD@WF\CD@HI}\def\CD@x{\CD@QJ\CD@DC\CD@MJ\ifdim\CD@DC=\z@
\else\CD@pF\CD@DC\fi\ifvoid3 \setbox3=\null\ht3\CD@tI\dp3\CD@sI\else\CD@V{\ht
3}\CD@tI\CD@V{\dp3}\CD@sI\fi\dimen3=.5\wd3 \ifdim\dimen3=\z@\CD@tE\else\dimen
3-\CD@XH\fi\else\CD@TB\fi\CD@V{\dimen2}{\wd7}\CD@V{\dimen2}{\wd6}\CD@qJ
\advance\dimen2-2\dimen3 \dimen4.5\dimen2 \dimen2\dimen4 \advance\dimen2%
\CD@eJ\advance\dimen4-\CD@eJ\advance\dimen2-\wd1 \advance\dimen4-\wd5 \ifvoid
2 \else\CD@V{\ht3}{\ht2}\CD@V{\dp3}{\dp2}\CD@V{\dimen2}{\wd2}\fi\ifvoid4 \else
\CD@V{\ht3}{\ht4}\CD@V{\dp3}{\dp4}\CD@V{\dimen4}{\wd4}\fi\advance\skip2\dimen
2 \advance\skip4\dimen4 \CD@tE\advance\skip2\skip4 \dimen0\dimen5 \advance
\dimen0\wd5 \skip3-\skip4 \advance\skip3-\dimen0 \let\CD@jD\empty\else\skip3%
\z@\relax\dimen0\z@\fi}\def\CD@WF{\offinterlineskip\lineskip.2\CD@zC\ifvoid6
\else\setbox3=\vbox{\hbox to2\dimen3{\hss\box6\hss}\box3}\fi\ifvoid7 \else
\setbox3=\vtop{\box3 \hbox to2\dimen3{\hss\box7\hss}}\fi}\def\CD@HI{\kern
\dimen1 \box1 \CD@aJ\CD@iD\hskip\skip2 \kern\dimen0 \ifincommdiag\CD@jE
\penalty1\fi\kern\dimen3 \penalty\CD@GB\hskip\skip3 \null\kern-\dimen3 \else
\hskip\skip3 \fi\box3 \CD@jD\hskip\skip4 \box5 \kern\dimen5}\def\CD@MF{\ifnum
\CD@LH>\CD@TC\CD@V{\dimen1}\objectheight\CD@V{\dimen5}\objectheight\else\CD@V
{\dimen1}\objectwidth\CD@V{\dimen5}\objectwidth\fi}\def\CD@Y{\begingroup
\ifdim\dimen7=\z@\kern\dimen8 \else\ifdim\dimen6=\z@\kern\dimen9 \else\dimen5%
\dimen6 \dimen6\dimen9 \CD@KJ\dimen4\dimen2 \CD@dG{\dimen4}\dimen6\dimen5
\dimen7\dimen8 \CD@KJ\CD@iC{\dimen2}\ifdim\dimen2<\dimen4 \kern\dimen2 \else
\kern\dimen4 \fi\fi\fi\endgroup}\def\CD@jJ{\CD@JI\setbox\z@\hbox{\lower
\axisheight\hbox to\dimen2{\CD@DF\ifPositiveGradient\dimen8\ht\CD@MH\dimen9%
\CD@mI\else\dimen8\dp3 \dimen9\dimen1 \fi\else\dimen8 \ifPositiveGradient
\objectheight\else\z@\fi\dimen9\objectwidth\fi\advance\dimen8
\ifPositiveGradient-\fi\axisheight\CD@Y\unhbox\z@\CD@DF\ifPositiveGradient
\dimen8\dp3 \dimen9\dimen0 \else\dimen8\ht\CD@MH\dimen9\CD@mF\fi\else\dimen8
\ifPositiveGradient\z@\else\objectheight\fi\dimen9\objectwidth\fi\advance
\dimen8 \ifPositiveGradient\else-\fi\axisheight\CD@Y}}}\def\CD@bD{\dimen6
\CD@aK\DiagramCellHeight\dimen7 \CD@WK\DiagramCellWidth\CD@jJ
\ifPositiveGradient\advance\dimen7-\CD@ZK\DiagramCellWidth\else\dimen7 \CD@ZK
\DiagramCellWidth\dimen6\z@\fi\advance\dimen6-\CD@bK\DiagramCellHeight\CD@mK
\setbox0=\rlap{\kern-\dimen7 \lower\dimen6\box\z@}\ht0\z@\dp0\z@\raise
\axisheight\box0 }\def\CD@mK{\setbox0\hbox{\ht\z@\z@\dp\z@\z@\wd\z@\z@\CD@hK
\expandafter\CD@tK{q \CD@eK\space\CD@lK\space\CD@kK\space\CD@eK\space0 0 cm}%
\else\global\CD@iG\CD@eD{\the\CD@TC\space\ifPositiveGradient\else-\fi\the
\CD@LH\space bturn}\fi\box\z@\CD@gK}}\def\CD@vB{\advance\CD@hF-\CD@mI\CD@wJ
\CD@hF\advance\CD@wJ\CD@hI\ifvoid\CD@sH\ifdim\CD@wJ<.1em\ifnum\CD@gD=\@m\else
\CD@aG h\CD@wJ<.1em:objects overprint:\CD@FA\CD@gD\fi\fi\else\ifhbox\CD@sH
\CD@SK\else\CD@TK\fi\advance\CD@wJ\CD@mI\CD@bH{-\CD@mI}{\box\CD@sH}{\CD@wJ}%
\z@\fi\CD@hF-\CD@mF\CD@gD\CD@FA\CD@hI\z@}\def\CD@SK{\setbox\CD@sH=\hbox{%
\unhbox\CD@sH\unskip\unpenalty}\setbox\CD@tH=\hbox{\unhbox\CD@tH\unskip
\unpenalty}\setbox\CD@sH=\hbox to\CD@wJ{\CD@OA\wd\CD@sH\unhbox\CD@sH\CD@PA
\lastkern\unkern\ifdim\CD@PA=\z@\CD@UB\advance\CD@OA-\wd\CD@tH\else\CD@TB\fi
\ifnum\lastpenalty=\z@\else\CD@JA\unpenalty\fi\kern\CD@PA\ifdim\CD@hF<\CD@OA
\CD@JA\fi\ifdim\CD@hI<\wd\CD@tH\CD@JA\fi\CD@jE\CD@hI\CD@wJ\advance\CD@hI-%
\CD@OA\advance\CD@hI\wd\CD@tH\ifdim\CD@hI<2\wd\CD@tH\CD@aG h\CD@hI<2\wd\CD@tH
:arrow too short:\CD@FA\CD@gD\fi\divide\CD@hI\tw@\CD@hF\CD@wJ\advance\CD@hF-%
\CD@hI\fi\CD@tE\kern-\CD@hI\fi\hbox to\CD@hI{\unhbox\CD@tH}\CD@HG}}\CD@tG
\ifinpile\inpiletrue\inpilefalse\inpilefalse\def\pile{\protect\CD@UJ\protect
\CD@uH}\def\CD@uH#1{\CD@l#1\CD@QD}\def\CD@UJ{\CD@nB{pile}\setbox0=\vtop
\bgroup\aftergroup\CD@lD\inpiletrue\let\CD@FE\empty\let\pile\CD@KF\let\CD@QD
\CD@PD\let\CD@GD\CD@FD\CD@yH\baselineskip.5\PileSpacing\lineskip.1\CD@zC
\relax\lineskiplimit\lineskip\mathsurround\z@\tabskip\z@\let\\\CD@wH}\def
\CD@l{\CD@DE\CD@YF\empty\halign\bgroup\hfil\CD@k\let\CD@FE\CD@d\let\\\CD@vH##%
\CD@MD\CD@ND\hfil\CD@Q\CD@R##\cr}\CD@rG\CD@NE{pile only allows one column.}%
\CD@rG\CD@UE{you left it out!}\def\CD@R{\CD@QD\CD@Q\relax\CD@YA{missing \CD@yC
\space inserted after \string\pile}\CD@NE}\def\CD@PD{\CD@MD\crcr\egroup
\egroup}\def\CD@GD{\CD@MD}\def\CD@FD{\CD@MD\relax\CD@QD\CD@YA{missing \CD@yC
\space inserted between \string\pile\space and \CD@HD}\CD@UE}\def\CD@QD{%
\CD@MD}\def\CD@lD{\vbox{\dimen1\dp0 \unvbox0 \setbox0=\lastbox\advance\dimen1%
\dp0 \nointerlineskip\box0 \nointerlineskip\setbox0=\null\dp0.5\dimen1\ht0-%
\dp0 \box0}\ifincommdiag\CD@tJ\penalty-9998 \fi\xdef\CD@YF{pile}}\def\CD@vH{%
\cr}\def\CD@wH{\noalign{\skip@\prevdepth\advance\skip@-\baselineskip
\prevdepth\skip@}}\def\CD@KF#1{#1}\def\CD@TK{\setbox\CD@sH=\vbox{\unvbox
\CD@sH\setbox1=\lastbox\setbox0=\box\voidb@x\CD@tF\setbox\CD@sH=\lastbox
\ifhbox\CD@sH\CD@rC\repeat\unvbox0 \global\CD@QA\CD@ZE}\CD@ZE\CD@QA}\def
\CD@rC{\CD@jE\setbox\CD@sH=\hbox{\unhbox\CD@sH\unskip\setbox\CD@sH=\lastbox
\unskip\unhbox\CD@sH}\ifdim\CD@wJ<\wd\CD@sH\CD@aG h\CD@wJ<\wd\CD@sH:arrow in
pile too short:\CD@FA\CD@gD\else\setbox\CD@sH=\hbox to\CD@wJ{\unhbox\CD@sH}%
\fi\else\CD@gJ\fi\setbox0=\vbox{\box\CD@sH\nointerlineskip\ifvoid0 \CD@tJ\box
1 \else\vskip\skip0 \unvbox0 \fi}\skip0=\lastskip\unskip}\def\CD@gJ{\penalty7
\noindent\unhbox\CD@sH\unskip\setbox\CD@sH=\lastbox\unskip\unhbox\CD@sH
\endgraf\setbox\CD@tH=\lastbox\unskip\setbox\CD@tH=\hbox{\CD@JG\unhbox\CD@tH
\unskip\unskip\unpenalty}\ifcase\prevgraf\cd@shouldnt P\or\ifdim\CD@wJ<\wd
\CD@tH\CD@aG h\CD@wJ<\wd\CD@sH:object in pile too wide:\CD@FA\CD@gD\setbox
\CD@sH=\hbox to\CD@wJ{\hss\unhbox\CD@tH\hss}\else\setbox\CD@sH=\hbox to\CD@wJ
{\hss\kern\CD@hF\unhbox\CD@tH\kern\CD@hI\hss}\fi\or\setbox\CD@sH=\lastbox
\unskip\CD@SK\else\cd@shouldnt Q\fi\unskip\unpenalty}\def\CD@cD{\CD@MJ\ifvoid
3 \setbox3=\null\ht3\axisheight\dp3-\ht3 \dimen3.5\CD@LF\else\dimen4\dp3
\dimen3.5\wd3 \setbox3=\CD@GG{\box3}\dp3\dimen4 \ifdim\ht3=-\dp3 \else\CD@TB
\fi\fi\dimen0\dimen3 \advance\dimen0-.5\CD@LF\setbox0\null\ht0\ht3\dp0\dp3\wd
0\wd3 \ifvoid6\else\setbox6\hbox{\unhbox6\kern\dimen0\kern2pt}\dimen0\wd6 \fi
\ifvoid7\else\setbox7\hbox{\kern2pt\kern\dimen3\unhbox7}\dimen3\wd7 \fi
\setbox3\hbox{\ifvoid6\else\kern-\dimen0\unhbox6\fi\unhbox3 \ifvoid7\else
\unhbox7\kern-\dimen3\fi}\ht3\ht0\dp3\dp0\wd3\wd0 \CD@tE\dimen4=\ht\CD@MH
\advance\dimen4\dp5 \advance\dimen4\dimen1 \let\CD@jD\empty\else\dimen4\ht3
\fi\setbox0\null\ht0\dimen4 \offinterlineskip\setbox8=\vbox spread2ex{\kern
\dimen5 \box1 \CD@iD\vfill\CD@tE\else\kern\CD@eJ\fi\box0}\ht8=\z@\setbox9=%
\vtop spread2ex{\kern-\ht3 \kern-\CD@eJ\box3 \CD@jD\vfill\box5 \kern\dimen1}%
\dp9=\z@\hskip\dimen0plus.0001fil \box9 \kern-\CD@LF\box8 \CD@kE\penalty2 \fi
\CD@tE\penalty1 \fi\kern\PileSpacing\kern-\PileSpacing\kern-.5\CD@LF\penalty
\CD@GB\null\kern\dimen3}\def\CD@cI{\ifhbox\CD@VA\CD@KB{clashing verticals}\ht
\CD@MH.5\dp\CD@VA\dp\CD@MH-\ht5 \CD@yB\ht\CD@MH\z@\dp\CD@MH\z@\fi\dimen1\dp
\CD@VA\CD@xA\prevgraf\unvbox\CD@VA\CD@wA\lastpenalty\unpenalty\setbox\CD@VA=%
\null\setbox\CD@lI=\hbox{\CD@JG\unhbox\CD@lI\unskip\unpenalty\dimen0\lastkern
\unkern\unkern\unkern\kern\dimen0 \CD@HG}\setbox\CD@lF=\hbox{\unhbox\CD@lF
\dimen0\lastkern\unkern\unkern\global\CD@QA\lastkern\unkern\kern\dimen0 }%
\CD@tF\ifnum\CD@xA>4 \CD@zI\repeat\unskip\unskip\advance\CD@mF.5\wd\CD@VA
\advance\CD@mF\wd\CD@lF\advance\CD@mI.5\wd\CD@VA\advance\CD@mI\wd\CD@lI\ifnum
\CD@FA=\CD@lA\CD@OA.5\wd\CD@VA\edef\CD@NK{\the\CD@OA}\fi\setbox\CD@VA=\hbox{%
\kern-\CD@mF\box\CD@lF\unhbox\CD@VA\box\CD@lI\kern-\CD@mI\penalty\CD@wA
\penalty\CD@NB}\ht\CD@VA\dimen1 \dp\CD@VA\z@\wd\CD@VA\CD@tB\CD@vB}\def\CD@zI{%
\ifdim\wd\CD@lF<\CD@QA\setbox\CD@lF=\hbox to\CD@QA{\CD@JG\unhbox\CD@lF}\fi
\advance\CD@xA\m@ne\setbox\CD@VA=\hbox{\box\CD@lF\unhbox\CD@VA}\unskip\setbox
\CD@lF=\lastbox\setbox\CD@lF=\hbox{\unhbox\CD@lF\unskip\unpenalty\dimen0%
\lastkern\unkern\unkern\global\CD@QA\lastkern\unkern\kern\dimen0 }}\def\CD@yB
{\dimen1\dp\CD@VA\ifhbox\CD@VA\CD@xB\else\CD@zB\fi\setbox\CD@VA=\vbox{%
\penalty\CD@NB}\dp\CD@VA-\dp\CD@MH\wd\CD@VA\CD@tB}\def\CD@zB{\unvbox\CD@VA
\CD@wA\lastpenalty\unpenalty\ifdim\dimen1<\ht\CD@MH\CD@aG v\dimen1<\ht\CD@MH:%
rows overprint:\CD@NB\CD@wA\fi}\def\CD@xB{\dimen0=\ht\CD@VA\setbox\CD@VA=%
\hbox\bgroup\advance\dimen1-\ht\CD@MH\unhbox\CD@VA\CD@xA\lastpenalty
\unpenalty\CD@wA\lastpenalty\unpenalty\global\CD@RA-\lastkern\unkern\setbox0=%
\lastbox\CD@tF\setbox\CD@VA=\hbox{\box0\unhbox\CD@VA}\setbox0=\lastbox\ifhbox
0 \CD@kJ\repeat\global\CD@SA-\lastkern\unkern\global\CD@QA\CD@JK\unhbox\CD@VA
\egroup\CD@JK\CD@QA\CD@bH{\CD@SA}{\box\CD@VA}{\CD@RA}{\dimen1}}\def\CD@kJ{%
\setbox0=\hbox to\wd0\bgroup\unhbox0 \unskip\unpenalty\dimen7\lastkern\unkern
\ifnum\lastpenalty=1 \unpenalty\CD@UB\else\CD@TB\fi\ifnum\lastpenalty=2
\unpenalty\dimen2.5\dimen0\advance\dimen2-.5\dimen1\advance\dimen2-%
\axisheight\else\dimen2\z@\fi\setbox0=\lastbox\dimen6\lastkern\unkern\setbox1%
=\lastbox\setbox0=\vbox{\unvbox0 \CD@tE\kern-\dimen1 \else\ifdim\dimen2=\z@
\else\kern\dimen2 \fi\fi}\ifdim\dimen0<\ht0 \CD@aG v\dimen0<\ht0:upper part of
vertical too short:{\CD@tE\CD@NB\else\CD@wA\fi}\CD@xA\else\setbox0=\vbox to%
\dimen0{\unvbox0}\fi\setbox1=\vtop{\unvbox1}\ifdim\dimen1<\dp1 \CD@aG v\dimen
1<\dp1:lower part of vertical too short:\CD@NB\CD@wA\else\setbox1=\vtop to%
\dimen1{\ifdim\dimen2=\z@\else\kern-\dimen2 \fi\unvbox1 }\fi\box1 \kern\dimen
6 \box0 \kern\dimen7 \CD@HG\global\CD@QA\CD@JK\egroup\CD@JK\CD@QA\relax}%
\countdef\CD@u=14 \newcount\CD@CA\newcount\CD@XB\newcount\CD@NB\let\CD@LB
\insc@unt\newcount\CD@FA\newcount\CD@lA\let\CD@mA\CD@XB\newcount\CD@MB\CD@tG
\CD@DF\CD@bI\CD@aI\CD@aI\def\CD@nD{-1}\def\CD@K{\CD@t\ifnum\CD@nD<\z@\else
\begingroup\scrollmode\showboxdepth\CD@nD\showboxbreadth\maxdimen\showlists
\endgroup\fi\CD@bI\CD@zF\CD@CA=\CD@u\advance\CD@CA1 \CD@XB=\CD@CA\ifnum\CD@NB
=1 \CD@JA\fi\advance\CD@XB\CD@NB\dimen1\z@\skip0\z@\count@=\insc@unt\advance
\count@\CD@u\divide\count@2 \ifnum\CD@XB>\count@\CD@KB{The diagram has too
many rows! It can't be reformatted.}\else\CD@NG\CD@WI\fi\CD@cH}\def\CD@NG{%
\CD@NB\CD@CA\CD@uF\ifnum\CD@NB<\CD@XB\setbox\CD@NB\box\voidb@x\advance\CD@NB1%
\relax\repeat\CD@NB\CD@CA\skip\z@\z@\CD@uF\CD@GB\lastpenalty\unpenalty\ifnum
\CD@GB>\z@\CD@KE\repeat\ifnum\CD@GB=-123 \CD@tJ\unpenalty\else\cd@shouldnt D%
\fi\ifx\v@grid\relax\else\CD@NB\CD@XB\advance\CD@NB\m@ne\expandafter\CD@VJ
\v@grid\fi\CD@MB\CD@mA\CD@tB\z@\CD@XG\ifx\h@grid\relax\else\expandafter\CD@LJ
\h@grid\fi\count@\CD@XB\advance\count@\m@ne\CD@YB\ht\count@}\def\CD@KE{%
\ifcase\CD@GB\or\CD@MG\else\CD@uA-\lastpenalty\unpenalty\CD@vA\lastpenalty
\unpenalty\setbox0=\lastbox\CD@WG\fi\CD@wD}\def\CD@wD{\skip1\lastskip\unskip
\advance\skip0\skip1 \ifdim\skip1=\z@\else\expandafter\CD@wD\fi}\def\CD@MG{%
\setbox0=\lastbox\CD@pI\dp0 \advance\CD@pI\skip\z@\skip\z@\z@\advance\CD@NF
\CD@pI\CD@uE\ifnum\CD@NB>\CD@CA\CD@NF\DiagramCellHeight\CD@pI\CD@NF\advance
\CD@pI-\CD@qI\fi\fi\CD@qI\ht0 \CD@NF\CD@qI\setbox\CD@NB\hbox{\unhbox\CD@NB
\unhbox0}\dp\CD@NB\CD@pI\ht\CD@NB\CD@qI\advance\CD@NB1 }\def\CD@WG{\ifnum
\CD@uA<\z@\advance\CD@uA\CD@XB\ifnum\CD@uA<\CD@CA\CD@UG\else\CD@OA\dp\CD@uA
\CD@PA\ht\CD@uA\setbox\CD@uA\hbox{\box\z@\penalty\CD@vA\penalty\CD@GB\unhbox
\CD@uA}\dp\CD@uA\CD@OA\ht\CD@uA\CD@PA\fi\else\CD@UG\fi}\def\CD@UG{\CD@KB{%
diagonal goes outside diagram (lost)}}\def\CD@fI{\advance\CD@uA\CD@XB\ifnum
\CD@uA<\CD@CA\CD@UG\else\ifnum\CD@uA=\CD@NB\CD@VG\else\ifnum\CD@uA>\CD@NB
\cd@shouldnt M\else\CD@OA\dp\CD@uA\CD@PA\ht\CD@uA\setbox\CD@uA\hbox{\box\z@
\penalty\CD@vA\penalty\CD@GB\unhbox\CD@uA}\dp\CD@uA\CD@OA\ht\CD@uA\CD@PA\fi
\fi\fi}\def\CD@WI{\CD@AJ\setbox\CD@PC=\hbox{\CD@k A\@super f\CD@lJ f\CD@ND}%
\CD@ZE\z@\CD@JK\z@\CD@kI\z@\CD@kF\z@\CD@NB=\CD@XB\CD@NF\z@\CD@uB\z@\CD@uF
\ifnum\CD@NB>\CD@CA\advance\CD@NB\m@ne\CD@qI\ht\CD@NB\CD@pI\dp\CD@NB\advance
\CD@NF\CD@qI\CD@rI\advance\CD@uB\CD@NF\CD@KC\CD@ZI\CD@w\ht\CD@NB\CD@qI\dp
\CD@NB\CD@pI\nointerlineskip\box\CD@NB\CD@NF\CD@pI\setbox\CD@NB\null\ht\CD@NB
\CD@uB\repeat\CD@wB\nointerlineskip\box\CD@NB\CD@gG\CD@ZE\DiagramCellWidth{%
width}\CD@gG\CD@JK\DiagramCellHeight{height}\CD@VA\CD@LB\advance\CD@VA-\CD@lA
\advance\CD@VA\m@ne\advance\CD@VA\CD@mA\dimen0\wd\CD@VA\CD@tI\axisheight
\dimen1\CD@uB\advance\dimen1-\CD@YB\dimen2\CD@kI\advance\dimen2-\dimen0
\advance\CD@XB-\CD@CA\advance\CD@LB-\CD@lA}\count@\year\multiply\count@12
\advance\count@\month\ifnum\count@>24398 \loop\iftrue\repeat\fi\def\CD@wB{\CD@qI-\CD@NF\CD@pI\CD@NF\setbox\CD@MH=\null\dp
\CD@MH\CD@NF\ht\CD@MH-\CD@NF\CD@mF\z@\CD@mI\z@\CD@lA\CD@LB\advance\CD@lA-%
\CD@MB\advance\CD@lA\CD@mA\CD@FA\CD@LB\CD@VA\CD@MB\CD@sF\ifnum\CD@FA>\CD@lA
\advance\CD@FA\m@ne\advance\CD@VA\m@ne\CD@tB\wd\CD@VA\setbox\CD@FA=\box
\voidb@x\CD@yB\repeat\CD@w\ht\CD@NB\CD@qI\dp\CD@NB\CD@pI}\def\CD@gG#1#2#3{%
\ifdim#1>.01\CD@zC\CD@PA#2\relax\advance\CD@PA#1\relax\advance\CD@PA.99\CD@zC
\count@\CD@PA\divide\count@\CD@zC\CD@KB{increase cell #3 to \the\count@ em}%
\fi}\def\CD@rI{\CD@FA=\CD@LB\penalty4 \noindent\unhbox\CD@NB\CD@sF\unskip
\setbox0=\lastbox\ifhbox0 \advance\CD@FA\m@ne\setbox\CD@FA\hbox to\wd0{\null
\penalty-9990\null\unhbox0}\repeat\CD@lA\CD@FA\advance\CD@FA\CD@MB\advance
\CD@FA-\CD@mA\ifnum\CD@FA<\CD@LB\count@\CD@FA\advance\count@\m@ne\dimen0=\wd
\count@\count@\CD@MB\advance\count@\m@ne\CD@tB\wd\count@\CD@sF\ifnum\CD@FA<%
\CD@LB\CD@DJ\CD@XG\dimen0\wd\CD@FA\advance\CD@FA1 \repeat\fi\CD@sF\CD@GB
\lastpenalty\unpenalty\ifnum\CD@GB>\z@\CD@vA\lastpenalty\unpenalty\CD@VG
\repeat\endgraf\unskip\ifnum\lastpenalty=4 \unpenalty\else\cd@shouldnt S\fi}%
\def\CD@VG{\advance\CD@vA\CD@lA\advance\CD@vA\m@ne\setbox0=\lastbox\ifnum
\CD@vA<\CD@LB\setbox\CD@vA\hbox{\box0\penalty\CD@GB\unhbox\CD@vA}\else\CD@UG
\fi}\def\CD@bG{}\CD@tG\CD@uE\CD@WB\CD@VB\def\CD@DJ{\advance\dimen0\wd\CD@FA
\divide\dimen0\tw@\CD@uE\dimen0\DiagramCellWidth\else\CD@V{\dimen0}%
\DiagramCellWidth\CD@pJ\fi\advance\CD@tB\dimen0 }\def\CD@XG{\setbox\CD@MB=%
\vbox{}\dp\CD@MB=\CD@uB\wd\CD@MB\CD@tB\advance\CD@MB1 }\def\CD@LJ#1,{\def
\CD@GK{#1}\ifx\CD@GK\CD@RD\else\advance\CD@tB\CD@GK\DiagramCellWidth\CD@XG
\expandafter\CD@LJ\fi}\def\CD@VJ#1,{\def\CD@GK{#1}\ifx\CD@GK\CD@RD\else\ifnum
\CD@NB>\CD@CA\CD@NF\CD@GK\DiagramCellHeight\advance\CD@NF-\dp\CD@NB\advance
\CD@NB\m@ne\ht\CD@NB\CD@NF\fi\expandafter\CD@VJ\fi}\def\CD@pJ{\CD@wE\CD@OA
\dimen0 \advance\CD@OA-\DiagramCellWidth\ifdim\CD@OA>2\MapShortFall\CD@KB{%
badly drawn diagonals (see manual)}\let\CD@pJ\empty\fi\else\let\CD@pJ\empty
\fi}\def\CD@KC{\CD@VA\CD@mA\CD@sF\ifnum\CD@VA<\CD@MB\dimen0\dp\CD@VA\advance
\dimen0\CD@NF\dp\CD@VA\dimen0 \advance\CD@VA1 \repeat}\def\CD@bH#1#2#3#4{%
\ifnum\CD@FA<\CD@LB\CD@OA=#1\relax\setbox\CD@FA=\hbox{\setbox0=#2\dimen7=#4%
\relax\dimen8=#3\relax\ifhbox\CD@FA\unhbox\CD@FA\advance\CD@OA-\lastkern
\unkern\fi\ifdim\CD@OA=\z@\else\kern-\CD@OA\fi\raise\dimen7\box0 \kern-\dimen
8 }\ifnum\CD@FA=\CD@lA\CD@V\CD@kF\CD@OA\fi\else\cd@shouldnt O\fi}\def\CD@w{%
\setbox\CD@NB=\hbox{\CD@FA\CD@lA\CD@VA\CD@mA\CD@PA\z@\relax\CD@sF\ifnum\CD@FA
<\CD@LB\CD@tB\wd\CD@VA\relax\CD@eI\advance\CD@FA1 \advance\CD@VA1 \repeat}%
\CD@V\CD@kI{\wd\CD@NB}\wd\CD@NB\z@}\def\CD@eI{\ifhbox\CD@FA\CD@OA\CD@tB\relax
\advance\CD@OA-\CD@PA\relax\ifdim\CD@OA=\z@\else\kern\CD@OA\fi\CD@PA\CD@tB
\advance\CD@PA\wd\CD@FA\relax\unhbox\CD@FA\advance\CD@PA-\lastkern\unkern\fi}%
\def\CD@ZI{\setbox\CD@sH=\box\voidb@x\CD@VA=\CD@MB\CD@FA\CD@LB\CD@VA\CD@mA
\advance\CD@VA\CD@FA\advance\CD@VA-\CD@lA\advance\CD@VA\m@ne\CD@tB\wd\CD@VA
\count@\CD@LB\advance\count@\m@ne\CD@hF.5\wd\count@\advance\CD@hF\CD@tB\CD@A
\m@ne\CD@gD\@m\CD@sF\ifnum\CD@FA>\CD@lA\advance\CD@FA\m@ne\advance\CD@hF-%
\CD@tB\CD@PI\wd\CD@VA\CD@tB\advance\CD@hF\CD@tB\advance\CD@VA\m@ne\CD@tB\wd
\CD@VA\repeat\CD@mF\CD@kF\CD@mI-\CD@mF\CD@vB}\newcount\CD@GB\def\CD@s{}\def
\CD@t{\CD@vK\mathsurround\z@\hsize\z@\rightskip\z@ plus1fil minus\maxdimen
\parfillskip\z@\linepenalty9000 \looseness0 \hfuzz\maxdimen\hbadness10000
\clubpenalty0 \widowpenalty0 \displaywidowpenalty0 \interlinepenalty0
\predisplaypenalty0 \postdisplaypenalty0 \interdisplaylinepenalty0
\interfootnotelinepenalty0 \floatingpenalty0 \brokenpenalty0 \everypar{}%
\leftskip\z@\parskip\z@\parindent\z@\pretolerance10000 \tolerance10000
\hyphenpenalty10000 \exhyphenpenalty10000 \binoppenalty10000 \relpenalty10000
\adjdemerits0 \doublehyphendemerits0 \finalhyphendemerits0 \baselineskip\z@
\CD@IA\def\parskip{\cd@shouldnt{PS}}\prevdepth\z@}\newbox\CD@KG\newbox\CD@IG
\def\CD@JG{\unhcopy\CD@KG}\def\CD@HG{\unhcopy\CD@IG}\def\CD@iJ{\hbox{}%
\penalty1\nointerlineskip}\def\CD@PI{\penalty5 \noindent\setbox\CD@MH=\null
\CD@mF\z@\CD@mI\z@\ifnum\CD@FA<\CD@LB\ht\CD@MH\ht\CD@FA\dp\CD@MH\dp\CD@FA
\unhbox\CD@FA\skip0=\lastskip\unskip\else\CD@OK\skip0=\z@\fi\endgraf\ifcase
\prevgraf\cd@shouldnt Y \or\cd@shouldnt Z \or\CD@RI\or\CD@XI\else\CD@QI\fi
\unskip\setbox0=\lastbox\unskip\unskip\unpenalty\noindent\unhbox0\setbox0%
\lastbox\unpenalty\unskip\unskip\unpenalty\setbox0\lastbox\CD@tF\CD@GB
\lastpenalty\unpenalty\ifnum\CD@GB>\z@\setbox\z@\lastbox\CD@lB\repeat\endgraf
\unskip\unskip\unpenalty}\def\CD@YJ{\CD@uA\CD@XB\advance\CD@uA-\CD@NB\CD@vA
\CD@FA\advance\CD@vA-\CD@lA\advance\CD@vA1 \expandafter\message{prevgraf=\the
\prevgraf at (\the\CD@uA,\the\CD@vA)}}\def\CD@XI{\CD@CE\setbox\CD@lI=\lastbox
\setbox\CD@lI=\hbox{\unhbox\CD@lI\unskip\unpenalty}\unskip\ifdim\ht\CD@lI>\ht
\CD@PC\setbox\CD@MH=\copy\CD@lI\else\ifdim\dp\CD@lI>\dp\CD@PC\setbox\CD@MH=%
\copy\CD@lI\else\CD@FG\CD@lI\fi\fi\advance\CD@mF.5\wd\CD@lI\advance\CD@mI.5%
\wd\CD@lI\setbox\CD@lI=\hbox{\unhbox\CD@lI\CD@HG}\CD@bH\CD@mF{\box\CD@lI}%
\CD@mI\z@\CD@yB\CD@vB}\def\CD@CE{\ifnum\CD@A>0 \advance\dimen0-\CD@tB\CD@iA-.%
5\dimen0 \CD@A-\CD@A\else\CD@A0 \CD@iA\z@\fi\setbox\CD@MH=\lastbox\setbox
\CD@MH=\hbox{\unhbox\CD@MH\unskip\unskip\unpenalty\setbox0=\lastbox\global
\CD@QA\lastkern\unkern}\advance\CD@iA-.5\CD@QA\unskip\setbox\CD@MH=\null
\CD@mI\CD@iA\CD@mF-\CD@iA}\def\CD@Z{\ht\CD@MH\CD@tI\dp\CD@MH\CD@sI}\def\CD@FG
#1{\setbox\CD@MH=\hbox{\CD@V{\ht\CD@MH}{\ht#1}\CD@V{\dp\CD@MH}{\dp#1}\CD@V{%
\wd\CD@MH}{\wd#1}\vrule height\ht\CD@MH depth\dp\CD@MH width\wd\CD@MH}}\def
\CD@QI{\CD@CE\CD@Z\setbox\CD@lI=\lastbox\unskip\setbox\CD@lF=\lastbox\unskip
\setbox\CD@lF=\hbox{\unhbox\CD@lF\unskip\global\CD@yA\lastpenalty\unpenalty}%
\advance\CD@yA9999 \ifcase\CD@yA\CD@VI\or\CD@YI\or\CD@TI\or\CD@dI\or\CD@cI\or
\CD@SI\else\cd@shouldnt9\fi}\def\CD@VI{\CD@FG\CD@lI\CD@UI\setbox\CD@sH=\box
\CD@lF\setbox\CD@tH=\box\CD@lI}\def\CD@YI{\CD@FG\CD@lF\setbox\CD@lI\hbox{%
\penalty8 \unhbox\CD@lI\unskip\unpenalty\ifnum\lastpenalty=8 \else\CD@xH\fi}%
\CD@UI\setbox\CD@lF=\hbox{\unhbox\CD@lF\unskip\unpenalty\global\setbox\CD@DA=%
\lastbox}\ifdim\wd\CD@lF=\z@\else\CD@xH\fi\setbox\CD@sH=\box\CD@DA}\def\CD@xH
{\CD@KB{extra material in \string\pile\space cell (lost)}}\def\CD@UI{\CD@yB
\ifvoid\CD@sH\else\CD@KB{Clashing horizontal arrows}\CD@mI.5\CD@hF\CD@mF-%
\CD@mI\CD@vB\CD@mI\z@\CD@mF\z@\fi\CD@hI\CD@hF\advance\CD@hI-\CD@mI\CD@hF-%
\CD@mF\CD@JC\CD@FA}\def\CD@RI{\setbox0\lastbox\unskip\CD@iA\z@\CD@Z\ifdim
\skip0>\z@\CD@tJ\CD@A0 \else\ifnum\CD@A<1 \CD@A0 \dimen0\CD@tB\fi\advance
\CD@A1 \fi}\def\VonH{\CD@MA46\VonH{.5\CD@LF}}\def\HonV{\CD@MA57\HonV{.5\CD@LF
}}\def\HmeetV{\CD@MA44\HmeetV{-\MapShortFall}}\def\CD@MA#1#2#3#4{\CD@pB34#1{%
\string#3}\CD@SD\CD@GB-999#2 \dimen0=#4\CD@tI\dimen0\advance\CD@tI\axisheight
\CD@sI\dimen0\advance\CD@sI-\axisheight\CD@CF\CD@HC\CD@ZD}\def\CD@HC#1{%
\setbox0=\hbox{\CD@k#1\CD@ND}\dimen0.5\wd0 \CD@tI\ht0 \CD@sI\dp0 \CD@ZD}\def
\CD@SD{\setbox0=\null\ht0=\CD@tI\dp0=\CD@sI\wd0=\dimen0 \copy0\penalty\CD@GB
\box0 }\def\CD@TI{\CD@GC\CD@yB}\def\CD@dI{\CD@GC\CD@vB}\def\CD@SI{\CD@GC
\CD@yB\CD@vB}\def\CD@GC{\setbox\CD@lI=\hbox{\unhbox\CD@lI}\setbox\CD@lF=\hbox
{\unhbox\CD@lF\global\setbox\CD@DA=\lastbox}\ht\CD@MH\ht\CD@DA\dp\CD@MH\dp
\CD@DA\advance\CD@mF\wd\CD@DA\advance\CD@mI\wd\CD@lI}\CD@tG
\ifPositiveGradient\CD@CI\CD@BI\CD@CI\CD@tG\ifClimbing\CD@rB\CD@qB\CD@rB
\newcount\DiagonalChoice\DiagonalChoice\m@ne\ifx\tenln\nullfont\CD@tJ\def
\CD@qF{\CD@KH\ifPositiveGradient/\else\CD@k\backslash\CD@ND\fi}\else\def
\CD@qF{\CD@rF\char\count@}\fi\let\CD@rF\tenln\def\Use@line@char#1{\hbox{#1%
\CD@rF\ifPositiveGradient\else\advance\count@64 \fi\char\count@}}\def\CD@cF{%
\Use@line@char{\count@\CD@TC\multiply\count@8\advance\count@-9\advance\count@
\CD@LH}}\def\CD@ZF{\Use@line@char{\ifcase\DiagonalChoice\CD@gF\or\CD@fF\or
\CD@fF\else\CD@gF\fi}}\def\CD@gF{\ifnum\CD@TC=\z@\count@'33 \else\count@
\CD@TC\multiply\count@\sixt@@n\advance\count@-9\advance\count@\CD@LH\advance
\count@\CD@LH\fi}\def\CD@fF{\count@'\ifcase\CD@LH55\or\ifcase\CD@TC66\or22\or
52\or61\or72\fi\or\ifcase\CD@TC66\or25\or22\or63\or52\fi\or\ifcase\CD@TC66\or
16\or36\or22\or76\fi\or\ifcase\CD@TC66\or27\or25\or67\or22\fi\fi\relax}\def
\CD@uC#1{\hbox{#1\setbox0=\Use@line@char{#1}\ifPositiveGradient\else\raise.3%
\ht0\fi\copy0 \kern-.7\wd0 \ifPositiveGradient\raise.3\ht0\fi\box0}}\def
\CD@jF#1{\hbox{\setbox0=#1\kern-.75\wd0 \vbox to.25\ht0{\ifPositiveGradient
\else\vss\fi\box0 \ifPositiveGradient\vss\fi}}}\def\CD@jI#1{\hbox{\setbox0=#1%
\dimen0=\wd0 \vbox to.25\ht0{\ifPositiveGradient\vss\fi\box0
\ifPositiveGradient\else\vss\fi}\kern-.75\dimen0 }}\CD@RC{+h:>}{%
\Use@line@char\CD@fF}\CD@RC{-h:>}{\Use@line@char\CD@gF}\CD@nF{+t:<}{-h:>}%
\CD@nF{-t:<}{+h:>}\CD@RC{+t:>}{\CD@jF{\Use@line@char\CD@fF}}\CD@RC{-t:>}{%
\CD@jI{\Use@line@char\CD@gF}}\CD@nF{+h:<}{-t:>}\CD@nF{-h:<}{+t:>}\CD@RC{+h:>>%
}{\CD@uC\CD@fF}\CD@RC{-h:>>}{\CD@uC\CD@gF}\CD@nF{+t:<<}{-h:>>}\CD@nF{-t:<<}{+%
h:>>}\CD@nF{+h:>->}{+h:>>}\CD@nF{-h:>->}{-h:>>}\CD@nF{+t:<-<}{-h:>>}\CD@nF{-t%
:<-<}{+h:>>}\CD@RC{+t:>>}{\CD@jF{\CD@uC\CD@fF}}\CD@RC{-t:>>}{\CD@jI{\CD@uC
\CD@gF}}\CD@nF{+h:<<}{-t:>>}\CD@nF{-h:<<}{+t:>>}\CD@nF{+t:>->}{+t:>>}\CD@nF{-%
t:>->}{-t:>>}\CD@nF{+h:<-<}{-t:>>}\CD@nF{-h:<-<}{+t:>>}\CD@RC{+f:-}{\CD@EF
\null\else\CD@cF\fi}\CD@nF{-f:-}{+f:-}\def\CD@tC#1#2{\vbox to#1{\vss\hbox to#%
2{\hss.\hss}\vss}}\def\hfdot{\CD@tC{2\axisheight}{.5em}}%
\def\vfdot{\CD@tC{1ex}\z@}
\def\CD@bF{\hbox{\dimen0=.3\CD@zC\dimen1\dimen0 \ifnum\CD@LH>\CD@TC\CD@iC{%
\dimen1}\else\CD@dG{\dimen0}\fi\CD@tC{\dimen0}{\dimen1}}}\newarrowfiller{.}%
\hfdot\hfdot\vfdot\vfdot\def\dfdot{\CD@bF\CD@CK}\CD@RC{+f:.}{\dfdot}\CD@RC{-f%
:.}{\dfdot}\def\CD@@K#1{\hbox\bgroup\def\CD@CH{#1\egroup}\afterassignment
\CD@CH
\count@='}\def\lnchar{\CD@@K\CD@qF}\def\CD@dF#1{\setbox#1=\hbox{\dimen5\dimen
#1 \setbox8=\box#1 \dimen1\wd8 \count@\dimen5 \divide\count@\dimen1 \ifnum
\count@=0 \box8 \ifdim\dimen5<.95\dimen1 \CD@gB{diagonal line too short}\fi
\else\dimen3=\dimen5 \advance\dimen3-\dimen1 \divide\dimen3\count@\dimen4%
\dimen3 \CD@dG{\dimen4}\ifPositiveGradient\multiply\dimen4\m@ne\fi\dimen6%
\dimen1 \advance\dimen6-\dimen3 \loop\raise\count@\dimen4\copy8 \ifnum\count@
>0 \kern-\dimen6 \advance\count@\m@ne\repeat\fi}}\def\CD@CG#1{\CD@EF\CD@xJ{#1%
}\else\CD@dF{#1}\fi}\def\CD@IH#1{}\newdimen\objectheight\objectheight1.8ex
\newdimen\objectwidth\objectwidth1em \def\CD@YD{\dimen6=\CD@aK
\DiagramCellHeight\dimen7=\CD@WK\DiagramCellWidth\CD@KJ\ifnum\CD@LH>0 \ifnum
\CD@TC>0 \CD@aF\else\aftergroup\CD@VC\fi\else\aftergroup\CD@UC\fi}\def\CD@VC{%
\CD@YA{diagonal map is nearly vertical}\CD@NA}\def\CD@UC{\CD@YA{diagonal map
is nearly horizontal}\CD@NA}\CD@rG\CD@NA{Use an orthogonal map instead}\def
\CD@aF{\CD@MJ\dimen3\dimen7\dimen7\dimen6\CD@iC{\dimen7}\advance\dimen3-%
\dimen7 \CD@MF\ifnum\CD@LH>\CD@TC\advance\dimen6-\dimen1\advance\dimen6-%
\dimen5 \CD@iC{\dimen1}\CD@iC{\dimen5}\else\dimen0\dimen1\advance\dimen0%
\dimen5\CD@dG{\dimen0}\advance\dimen6-\dimen0 \fi\dimen2.5\dimen7\advance
\dimen2-\dimen1 \dimen4.5\dimen7\advance\dimen4-\dimen5 \ifPositiveGradient
\dimen0\dimen5 \advance\dimen1-\CD@WK\DiagramCellWidth\advance\dimen1 \CD@ZK
\DiagramCellWidth\setbox6=\llap{\unhbox6\kern.1\ht2}\setbox7=\rlap{\kern.1\ht
2\unhbox7}\else\dimen0\dimen1 \advance\dimen1-\CD@ZK\DiagramCellWidth\setbox7%
=\llap{\unhbox7\kern.1\ht2}\setbox6=\rlap{\kern.1\ht2\unhbox6}\fi\setbox6=%
\vbox{\box6\kern.1\wd2}\setbox7=\vtop{\kern.1\wd2\box7}\CD@dG{\dimen0}%
\advance\dimen0-\axisheight\advance\dimen0-\CD@bK\DiagramCellHeight\dimen5-%
\dimen0 \advance\dimen0\dimen6 \advance\dimen1.5\dimen3 \ifdim\wd3>\z@\ifdim
\ht3>-\dp3\CD@TB\fi\fi\dimen3\dimen2 \dimen7\dimen2\advance\dimen7\dimen4
\ifvoid3 \else\CD@tE\else\dimen8\ht3\advance\dimen8-\axisheight\CD@iC{\dimen8%
}\CD@X{\dimen8}{.5\wd3}\dimen9\dp3\advance\dimen9\axisheight\CD@iC{\dimen9}%
\CD@X{\dimen9}{.5\wd3}\ifPositiveGradient\advance\dimen2-\dimen9\advance
\dimen4-\dimen8 \else\advance\dimen4-\dimen9\advance\dimen2-\dimen8 \fi\fi
\advance\dimen3-.5\wd3 \fi\dimen9=\CD@aK\DiagramCellHeight\advance\dimen9-2%
\DiagramCellHeight\CD@tE\advance\dimen2\dimen4 \CD@CG{2}\dimen2-\dimen0%
\advance\dimen2\dp2 \else\CD@CG{2}\CD@CG{4}\ifPositiveGradient\dimen2-\dimen0%
\advance\dimen2\dp2 \dimen4\dimen5\advance\dimen4-\ht4 \else\dimen4-\dimen0%
\advance\dimen4\dp4 \dimen2\dimen5\advance\dimen2-\ht2 \fi\fi\setbox0=\hbox to%
\z@{\kern\dimen1 \ifvoid1 \else\ifPositiveGradient\advance\dimen0-\dp1 \lower
\dimen0 \else\advance\dimen5-\ht1 \raise\dimen5 \fi\rlap{\unhbox1}\fi\raise
\dimen2\rlap{\unhbox2}\ifvoid3 \else\lower.5\dimen9\rlap{\kern\dimen3\unhbox3%
}\fi\kern.5\dimen7 \lower.5\dimen9\box6 \lower.5\dimen9\box7 \kern.5\dimen7
\CD@tE\else\raise\dimen4\llap{\unhbox4}\fi\ifvoid5 \else\ifPositiveGradient
\advance\dimen5-\ht5 \raise\dimen5 \else\advance\dimen0-\dp5 \lower\dimen0 \fi
\llap{\unhbox5}\fi\hss}\ht0=\axisheight\dp0=-\ht0\box0 }\def\NorthWest{\CD@BI
\CD@rB\DiagonalChoice0 }\def\NorthEast{\CD@CI\CD@rB\DiagonalChoice1 }\def
\SouthWest{\CD@CI\CD@qB\DiagonalChoice3 }\def\SouthEast{\CD@BI\CD@qB
\DiagonalChoice2 }\def\CD@aD{\vadjust{\CD@uA\CD@FA\advance\CD@uA
\ifPositiveGradient\else-\fi\CD@ZK\relax\CD@vA\CD@NB\advance\CD@vA-\CD@bK
\relax\hbox{\advance\CD@uA\ifPositiveGradient-\fi\CD@WK\advance\CD@vA\CD@aK
\hbox{\box6 \kern\CD@DC\kern\CD@eJ\penalty1 \box7 \box\z@}\penalty\CD@uA
\penalty\CD@vA}\penalty\CD@uA\penalty\CD@vA\penalty104}}\def\CD@eH#1{\relax
\vadjust{\hbox@maths{#1}\penalty\CD@FA\penalty\CD@NB\penalty\tw@}}\def\CD@lB{%
\ifcase\CD@GB\or\or\CD@bH{.5\wd0}{\box0}{.5\wd0}\z@\or\unhbox\z@\setbox\z@
\lastbox\CD@bH{.5\wd0}{\box0}{.5\wd0}\z@\unpenalty\unpenalty\setbox\z@
\lastbox\or\CD@TG\else\advance\CD@GB-100 \ifnum\CD@GB<\z@\cd@shouldnt B\fi
\setbox\z@\hbox{\kern\CD@mF\copy\CD@MH\kern\CD@mI\CD@uA\CD@XB\advance\CD@uA-%
\CD@NB\penalty\CD@uA\CD@uA\CD@FA\advance\CD@uA-\CD@lA\penalty\CD@uA\unhbox\z@
\global\CD@yA\lastpenalty\unpenalty\global\CD@zA\lastpenalty\unpenalty}\CD@uA
-\CD@yA\CD@vA\CD@zA\CD@fI\fi}\def\CD@TG{\unhbox\z@\setbox\z@\lastbox\CD@uA
\lastpenalty\unpenalty\advance\CD@uA\CD@mA\CD@vA\CD@XB\advance\CD@vA-%
\lastpenalty\unpenalty\dimen1\lastkern\unkern\setbox3\lastbox\dimen0\lastkern
\unkern\setbox0=\hbox to\z@{\unhbox0\setbox0\lastbox\setbox7\lastbox
\unpenalty\CD@eJ\lastkern\unkern\CD@DC\lastkern\unkern\setbox6\lastbox\dimen7%
\CD@tB\advance\dimen7-\wd\CD@uA\ifdim\dimen7<\z@\CD@CI\multiply\dimen7\m@ne
\let\mv\empty\else\CD@BI\def\mv{\raise\ht1}\kern-\dimen7 \fi\ifnum\CD@vA>%
\CD@NB\dimen6\CD@uB\advance\dimen6-\ht\CD@vA\else\dimen6\z@\fi\CD@jJ\CD@mK
\setbox1\null\ht1\dimen6\wd1\dimen7 \dimen7\dimen2 \dimen6\wd1 \CD@KJ\CD@uA
\CD@LH\CD@vA\CD@TC\dimen6\ht1 \CD@KJ\setbox2\null\divide\dimen2\tw@\advance
\dimen2\CD@eJ\CD@eG{\dimen2}\wd2\dimen2 \dimen0.5\dimen7 \advance\dimen0%
\ifPositiveGradient\else-\fi\CD@eJ\CD@dG{\dimen0}\advance\dimen0-\axisheight
\ht2\dimen0 \dimen0\CD@DC\CD@eG{\dimen0}\advance\dimen0\ht2\ht2\dimen0 \dimen
0\ifPositiveGradient-\fi\CD@DC\CD@dG{\dimen0}\advance\dimen0\wd2\wd2\dimen0
\setbox4\null\dimen0 .6\CD@zC\CD@eG{\dimen0}\ht4\dimen0 \dimen0 .2\CD@zC
\CD@dG{\dimen0}\wd4\dimen0 \dimen0\wd2 \ifvoid6\else\dimen1\ht4 \advance
\dimen1\ht2 \CD@CC6+-\raise\dimen1\rlap{\ifPositiveGradient\advance\dimen0-%
\wd6\advance\dimen0-\wd4 \else\advance\dimen0\wd4 \fi\kern\dimen0\box6}\fi
\dimen0\wd2 \ifvoid7\else\dimen1\ht4 \advance\dimen1-\ht2 \CD@CC7-+\lower
\dimen1\rlap{\ifPositiveGradient\advance\dimen0\wd4 \else\advance\dimen0-\wd7%
\advance\dimen0-\wd4 \fi\kern\dimen0\box7}\fi\mv\box0\hss}\ht0\z@\dp0\z@
\CD@bH{\z@}{\box\z@}{\z@}{\axisheight}}\def\CD@CC#1#2#3{\dimen4.5\wd#1 \ifdim
\dimen4>.25\dimen7\dimen4=.25\dimen7\fi\ifdim\dimen4>\CD@zC\dimen4.4\dimen4
\advance\dimen4.6\CD@zC\fi\CD@eG{\dimen4}\dimen5\axisheight\CD@dG{\dimen5}%
\advance\dimen4-\dimen5 \dimen5\dimen4\CD@eG{\dimen5}\advance\dimen0%
\ifPositiveGradient#2\else#3\fi\dimen5 \CD@dG{\dimen4}\advance\dimen1\dimen4 }
\def\CD@eD#1{\expandafter\CD@IK{#1}}\CD@ZA\CD@EK{output is PostScript
dependent}\def\CD@SC{\CD@IK{/bturn {gsave currentpoint currentpoint translate
4 2 roll neg exch atan rotate neg exch neg exch translate } def /eturn {%
currentpoint grestore moveto} def}}\def\CD@gK{\relax\CD@hK\CD@tK{Q}\else
\CD@IK{eturn}\fi} \def\CD@OJ#1{\count@#1\relax\multiply\count@7\advance
\count@16577\divide\count@33154 }\def\CD@fD#1{\expandafter\special{#1}} \def
\CD@xJ#1{\setbox#1=\hbox{\dimen0\dimen#1\CD@dG{\dimen0}\CD@OJ{\dimen0}\setbox
0=\null\ifPositiveGradient\count@-\count@\ht0\dimen0 \else\dp0\dimen0 \fi\box
0 \CD@uA\count@\CD@OJ\CD@LF\CD@fD{pn \the\count@}\CD@fD{pa 0 0}\CD@OJ{\dimen#%
1}\CD@fD{pa \the\count@\space\the\CD@uA}\CD@fD{fp}\kern\dimen#1}}\def\CD@JI{%
\CD@KJ\begingroup\ifdim\dimen7<\dimen6 \dimen2=\dimen6 \dimen6=\dimen7 \dimen
7=\dimen2 \count@\CD@LH\CD@LH\CD@TC\CD@TC\count@\else\dimen2=\dimen7 \fi
\ifdim\dimen6>.01\p@\CD@KI\global\CD@QA\dimen0 \else\global\CD@QA\dimen7 \fi
\endgroup\dimen2\CD@QA\CD@iK\CD@lK{\ifPositiveGradient\else-\fi\dimen6}\CD@iK
\CD@kK{\ifPositiveGradient-\fi\dimen6}\CD@iK\CD@eK{\dimen7}}\def\CD@KI{\CD@hJ
\ifdim\dimen7>1.73\dimen6 \divide\dimen2 4 \multiply\CD@TC2 \else\dimen2=0.%
353553\dimen2 \advance\CD@LH-\CD@TC\multiply\CD@TC4 \fi\dimen0=4\dimen2 \CD@ZG
4\CD@ZG{-2}\CD@ZG2\CD@ZG{-2.5}}\def\CD@AI{\begingroup\count@\dimen0 \dimen2 45%
pt \divide\count@\dimen2 \ifdim\dimen0<\z@\advance\count@\m@ne\fi\ifodd
\count@\advance\count@1\CD@@A\else\CD@y\fi\advance\dimen0-\count@\dimen2
\CD@gE\multiply\dimen0\m@ne\fi\ifnum\count@<0 \multiply\count@-7 \fi\dimen3%
\dimen1 \dimen6\dimen0 \dimen7 3754936sp \ifdim\dimen0<6\p@\def\CD@OG{4000}%
\fi\CD@KJ\dimen2\dimen3\CD@dG{\dimen2}\CD@hJ\multiply\CD@TC-6 \dimen0\dimen2
\CD@ZG1\CD@ZG{0.3}\dimen1\dimen0 \dimen2\dimen3 \dimen0\dimen3 \CD@ZG3\CD@ZG{%
1.5}\CD@ZG{0.3}\divide\count@2 \CD@gE\multiply\dimen1\m@ne\fi\ifodd\count@
\dimen2\dimen1\dimen1\dimen0\dimen0-\dimen2 \fi\divide\count@2 \ifodd\count@
\multiply\dimen0\m@ne\multiply\dimen1\m@ne\fi\global\CD@QA\dimen0\global
\CD@RA\dimen1\endgroup\dimen6\CD@QA\dimen7\CD@RA}\def\CD@OC{255}\let\CD@OG
\CD@OC\def\CD@KJ{\begingroup\ifdim\dimen7<\dimen6 \dimen9\dimen7\dimen7\dimen
6\dimen6\dimen9\CD@@A\else\CD@y\fi\dimen2\z@\dimen3\CD@XH\dimen4\CD@XH\dimen0%
\z@\dimen8=\CD@OG\CD@XH\CD@lC\global\CD@yA\dimen\CD@gE0\else3\fi\global\CD@zA
\dimen\CD@gE3\else0\fi\endgroup\CD@LH\CD@yA\CD@TC\CD@zA}\def\CD@lC{\count@
\dimen6 \divide\count@\dimen7 \advance\dimen6-\count@\dimen7 \dimen9\dimen4
\advance\dimen9\count@\dimen0 \ifdim\dimen9>\dimen8 \CD@@C\else\CD@AC\ifdim
\dimen6>\z@\dimen9\dimen6 \dimen6\dimen7 \dimen7\dimen9 \expandafter
\expandafter\expandafter\CD@lC\fi\fi}\def\CD@@C{\ifdim\dimen0=\z@\ifdim\dimen
9<2\dimen8 \dimen0\dimen8 \fi\else\advance\dimen8-\dimen4 \divide\dimen8%
\dimen0 \ifdim\count@\CD@XH<2\dimen8 \count@\dimen8 \dimen9\dimen4 \advance
\dimen9\count@\dimen0 \CD@AC\fi\fi}\def\CD@AC{\dimen4\dimen0 \dimen0\dimen9
\advance\dimen2\count@\dimen3 \dimen9\dimen2 \dimen2\dimen3 \dimen3\dimen9 }%
\def\CD@ZG#1{\CD@dG{\dimen2}\advance\dimen0 #1\dimen2 }\def\CD@dG#1{\divide#1%
\CD@TC\multiply#1\CD@LH}\def\CD@eG#1{\divide#1\CD@vA\multiply#1\CD@uA}\def
\CD@iC#1{\divide#1\CD@LH\multiply#1\CD@TC}\def\CD@hJ{\dimen6\CD@LH\CD@XH
\multiply\dimen6\CD@LH\dimen7\CD@TC\CD@XH\multiply\dimen7\CD@TC\CD@KJ}\def
\CD@iK#1#2{\begingroup\dimen@#2\relax\loop\ifdim\dimen2<.4\maxdimen\multiply
\dimen2\tw@\multiply\dimen@\tw@\repeat\divide\dimen2\@cclvi\divide\dimen@
\dimen2\relax\multiply\dimen@\@cclvi\expandafter\CD@jK\the\dimen@\endgroup
\let#1\CD@fK}{\catcode`p=12 \catcode`0=12 \catcode`.=12 \catcode`t=12 \gdef
\CD@jK#1pt{\gdef\CD@fK{#1}}}\ifx\errorcontextlines\CD@qK\CD@tJ\let\CD@GH
\relax\else\def\CD@GH{\errorcontextlines\m@ne}\fi\ifnum\inputlineno<0 \let
\CD@CD\empty\let\CD@W\empty\let\CD@mD\relax\let\CD@uI\relax\let\CD@vI\relax
\let\CD@zF\relax\else\def\CD@W{ at line \number\inputlineno}\def\CD@mD{ - first occurred%
}\def\CD@uI{\edef\CD@h{\the\inputlineno}\global\let\CD@jB\CD@h}\def\CD@h{9999%
}\def\CD@vI{\xdef\CD@jB{\the\inputlineno}}\def\CD@jB{\CD@h}\def\CD@zF{\ifnum
\CD@h<\inputlineno\edef\CD@CD{\space at lines \CD@h--\the\inputlineno}\else
\edef\CD@CD{\CD@W}\fi}\fi\let\CD@CD\empty\def\CD@YA#1#2{\CD@GH\errhelp=#2%
\expandafter\errmessage{\CD@tA: #1}}\def\CD@KB#1{\begingroup\expandafter
\message{! \CD@tA: #1\CD@CD}\ifnum\CD@XB>\CD@NB\ifnum\CD@CA>\CD@NB\else\ifnum
\CD@lA>\CD@FA\else\ifnum\CD@LB>\CD@FA\advance\CD@XB-\CD@NB\advance\CD@FA-%
\CD@lA\advance\CD@FA1\relax\expandafter\message{! (error detected at row \the
\CD@XB, column \the\CD@FA, but probably caused elsewhere)}\fi\fi\fi\fi
\endgroup}\def\CD@gB#1{{\expandafter\message{\CD@tA\space Warning: #1\CD@W}}}%
\def\CD@CB#1#2{\CD@gB{#1 \string#2 is obsolete\CD@mD}}\def\CD@AB#1{\CD@CB{%
Dimension}{#1}\CD@DE#1\CD@BB\CD@BB}\def\CD@BB{\CD@OA=}\def\CD@@B#1{\CD@CB{%
Count}{#1}\CD@DE#1\CD@OH\CD@OH}\def\CD@OH{\count@=}\def\HorizontalMapLength{%
\CD@AB\HorizontalMapLength}\def\VerticalMapHeight{\CD@AB\VerticalMapHeight}%
\def\VerticalMapDepth{\CD@AB\VerticalMapDepth}\def\VerticalMapExtraHeight{%
\CD@AB\VerticalMapExtraHeight}\def\VerticalMapExtraDepth{\CD@AB
\VerticalMapExtraDepth}\def\DiagonalLineSegments{\CD@@B\DiagonalLineSegments}%
\ifx\tenln\nullfont\CD@ZA\CD@KH{\CD@eF\space diagonal line and arrow font not
available}\else\let\CD@KH\relax\fi\def\CD@aG#1#2<#3:#4:#5#6{\begingroup\CD@PA
#3\relax\advance\CD@PA-#2\relax\ifdim.1em<\CD@PA\CD@uA#5\relax\CD@vA#6\relax
\ifnum\CD@uA<\CD@vA\count@\CD@vA\advance\count@-\CD@uA\CD@KB{#4 by \the\CD@PA
}\if#1v\let\CD@CH\CD@JK\edef\tmp{\the\CD@uA--\the\CD@vA,\the\CD@FA}\else
\advance\count@\count@\if#1l\advance\count@-\CD@A\else\if#1r\advance\count@
\CD@A\fi\fi\advance\CD@PA\CD@PA\let\CD@CH\CD@ZE\edef\tmp{\the\CD@NB,\the
\CD@uA--\the\CD@vA}\fi\divide\CD@PA\count@\ifdim\CD@CH<\CD@PA\global\CD@CH
\CD@PA\fi\fi\fi\endgroup}\CD@tG\CD@xE\CD@JD\CD@ID\CD@rG\CD@xI{See the message
above.}\CD@rG\CD@lH{Perhaps you've forgotten to end the diagram before
resuming the text, in\CD@uG which case some garbage may be added to the
diagram, but we should be ok now.\CD@uG Alternatively you've left a blank line
in the middle - TeX will now complain\CD@uG that the remaining \CD@S s are
misplaced - so please use comments for layout.}\CD@rG\CD@hD{You have already
closed too many brace pairs or environments; an \CD@HD\CD@uG command was (%
over)due.}\CD@rG\CD@hH{\CD@dC\space and \CD@HD\space commands must match.}%
\def\CD@jH{\ifnum\inputlineno=0 \else\expandafter\CD@iH\fi}\def\CD@iH{\CD@MD
\CD@GD\crcr\CD@YA{missing \CD@HD\space inserted before \CD@kH- type "h"}%
\CD@lH\enddiagram\CD@AG\CD@kH\par}\def\CD@AG#1{\edef\enddiagram{\noexpand
\CD@rD{#1\CD@W}}}\def\CD@rD#1{\CD@YA{\CD@HD\space(anticipated by #1) ignored}%
\CD@xI\let\enddiagram\CD@SG}\def\CD@SG{\CD@YA{misplaced \CD@HD\space ignored}%
\CD@hH}\def\CD@mC{\CD@YA{missing \CD@HD\space inserted.}\CD@hD\CD@AG{closing
group}}\ifx\ifx
\fi\fi

\catcode`\$=3 
\def\vboxtoz{\vbox to\z@}

\def\scriptaxis#1{\@scriptaxis{$\scriptstyle#1$}}
\def\ssaxis#1{\@ssaxis{$\scriptscriptstyle#1$}}
\def\@scriptaxis#1{\dimen0\axisheight\advance\dimen0-\ss@axisheight\raise
\dimen0\hbox{#1}}\def\@ssaxis#1{\dimen0\axisheight\advance\dimen0-%
\ss@axisheight\raise\dimen0\hbox{#1}}

\ifx\boldmath\CD@qK
\let\boldscriptaxis\scriptaxis
\def\boldscript#1{\hbox{$\scriptstyle#1$}}
\else\def\boldscriptaxis#1{\@scriptaxis{\boldmath$\scriptstyle#1$}}
\def\boldscript#1{\hbox{\boldmath$\scriptstyle#1$}}
\fi

\def\raisehook#1#2#3{\hbox{\setbox3=\hbox{#1$\scriptscriptstyle#3$}%
\dimen0\ss@axisheight
\dimen1\axisheight\advance\dimen1-\dimen0
\dimen2\ht3\advance\dimen2-\dimen0%
\advance\dimen2-0.021em\advance\dimen1 #2\dimen2%
\raise\dimen1\box3}}
\def\shifthook#1#2#3{\setbox1=\hbox{#1$\scriptscriptstyle#3$}\dimen0\wd1%
\divide\dimen0 12\CD@zH{\dimen0}
\dimen1\wd1\advance\dimen1-2\dimen0 \advance\dimen1-2\CD@oI\CD@zH{\dimen1}%
\kern#2\dimen1\box1}

\def\@cmex{\mathchar"03}

\def\make@pbk#1{\setbox\tw@\hbox to\z@{#1}\ht\tw@\z@\dp\tw@\z@\box\tw@}\def
\CD@fH#1{\overprint{\hbox to\z@{#1}}}\def\CD@qH{\kern0.11em}\def\CD@pH{\kern0%
.35em}

\def\dblvert{\def\CD@rH{\kern.5\PileSpacing}}\def\CD@rH{}

\def\SEpbk{\make@pbk{\CD@qH\CD@rH\vrule depth 2.87ex height -2.75ex width 0.%
95em \vrule height -0.66ex depth 2.87ex width 0.05em \hss}}

\def\SWpbk{\make@pbk{\hss\vrule height -0.66ex depth 2.87ex width 0.05em
\vrule depth 2.87ex height -2.75ex width 0.95em \CD@qH\CD@rH}}

\def\NEpbk{\make@pbk{\CD@qH\CD@rH\vrule depth -3.81ex height 4.00ex width 0.%
95em \vrule height 4.00ex depth -1.72ex width 0.05em \hss}}

\def\NWpbk{\make@pbk{\hss\vrule height 4.00ex depth -1.72ex width 0.05em
\vrule depth -3.81ex height 4.00ex width 0.95em \CD@qH\CD@rH}}

\def\puncture{{\setbox0\hbox{A}\vrule height.53\ht0 depth-.47\ht0 width.35\ht
0 \kern.12\ht0 \vrule height\ht0 depth-.65\ht0 width.06\ht0 \kern-.06\ht0
\vrule height.35\ht0 depth0pt width.06\ht0 \kern.12\ht0 \vrule height.53\ht0
depth-.47\ht0 width.35\ht0 }}

\def\NEclck{\overprint{\raise2.5ex\rlap{ \CD@rH$\scriptstyle\searrow$}}}
\def\NEanti{\overprint{\raise2.5ex\rlap{ \CD@rH$\scriptstyle\nwarrow$}}}
\def\NWclck{\overprint{\raise2.5ex\llap{$\scriptstyle\nearrow$ \CD@rH}}}
\def\NWanti{\overprint{\raise2.5ex\llap{$\scriptstyle\swarrow$ \CD@rH}}}
\def\SEclck{\overprint{\lower1ex\rlap{ \CD@rH$\scriptstyle\swarrow$}}}
\def\SEanti{\overprint{\lower1ex\rlap{ \CD@rH$\scriptstyle\nearrow$}}}
\def\SWclck{\overprint{\lower1ex\llap{$\scriptstyle\nwarrow$ \CD@rH}}}
\def\SWanti{\overprint{\lower1ex\llap{$\scriptstyle\searrow$ \CD@rH}}}

\def\rhvee{\mkern-10mu\greaterthan}
\def\lhvee{\lessthan\mkern-10mu}
\def\dhvee{\vboxtoz{\vss\hbox{$\vee$}\kern0pt}}
\def\uhvee{\vboxtoz{\hbox{$\wedge$}\vss}}
\newarrowhead{vee}\rhvee\lhvee\dhvee\uhvee

\def\dhlvee{\vboxtoz{\vss\hbox{$\scriptstyle\vee$}\kern0pt}}
\def\uhlvee{\vboxtoz{\hbox{$\scriptstyle\wedge$}\vss}}
\newarrowhead{littlevee}{\mkern1mu\scriptaxis\rhvee}{\scriptaxis\lhvee}%
\dhlvee\uhlvee\ifx\boldmath\CD@qK
\newarrowhead{boldlittlevee}{\mkern1mu\scriptaxis\rhvee}{\scriptaxis\lhvee}%
\dhlvee\uhlvee\else
\def\dhblvee{\vboxtoz{\vss\boldscript\vee\kern0pt}}
\def\uhblvee{\vboxtoz{\boldscript\wedge\vss}}
\newarrowhead{boldlittlevee}{\mkern1mu\boldscriptaxis\rhvee}{\boldscriptaxis
\lhvee}\dhblvee\uhblvee
\fi

\def\rhcvee{\mkern-10mu\succ}
\def\lhcvee{\prec\mkern-10mu}
\def\dhcvee{\vboxtoz{\vss\hbox{$\curlyvee$}\kern0pt}}
\def\uhcvee{\vboxtoz{\hbox{$\curlywedge$}\vss}}
\newarrowhead{curlyvee}\rhcvee\lhcvee\dhcvee\uhcvee

\def\rhvvee{\mkern-13mu\gg}
\def\lhvvee{\ll\mkern-13mu}
\def\dhvvee{\vboxtoz{\vss\hbox{$\vee$}\kern-.6ex\hbox{$\vee$}\kern0pt}}
\def\uhvvee{\vboxtoz{\hbox{$\wedge$}\kern-.6ex \hbox{$\wedge$}\vss}}
\newarrowhead{doublevee}\rhvvee\lhvvee\dhvvee\uhvvee

\def\rhtriangle{\triangleright\mkern1.2mu}
\def\lhtriangle{\triangleleft\mkern.8mu}
\def\uhtriangle{\vbox{\kern-.2ex \hbox{$\scriptscriptstyle\bigtriangleup$}%
\kern-.25ex}}
\def\dhtriangle{\vbox{\kern-.28ex \hbox{$\scriptscriptstyle\bigtriangledown$}%
\kern-.1ex}}
\def\dhblack{\vbox{\kern-.25ex\nointerlineskip\hbox{$\blacktriangledown$}}}%
\def\uhblack{\vbox{\kern-.25ex\nointerlineskip\hbox{$\blacktriangle$}}}%
\def\dhlblack{\vbox{\kern-.25ex\nointerlineskip\hbox{$\scriptstyle
\blacktriangledown$}}}
\def\uhlblack{\vbox{\kern-.25ex\nointerlineskip\hbox{$\scriptstyle
\blacktriangle$}}}
\newarrowhead{triangle}\rhtriangle\lhtriangle\dhtriangle\uhtriangle
\newarrowhead{blacktriangle}{\mkern-1mu\blacktriangleright\mkern.4mu}{%
\blacktriangleleft}\dhblack\uhblack\newarrowhead{littleblack}{\mkern-1mu%
\scriptaxis\blacktriangleright}{\scriptaxis\blacktriangleleft\mkern-2mu}%
\dhlblack\uhlblack

\def\rhla{\hbox{\setbox0=\lnchar55\dimen0=\wd0\kern-.6\dimen0\ht0\z@\raise
\axisheight\box0\kern.1\dimen0}}
\def\lhla{\hbox{\setbox0=\lnchar33\dimen0=\wd0\kern.05\dimen0\ht0\z@\raise
\axisheight\box0\kern-.5\dimen0}}
\def\dhla{\vboxtoz{\vss\rlap{\lnchar77}}}
\def\uhla{\vboxtoz{\setbox0=\lnchar66 \wd0\z@\kern-.15\ht0\box0\vss}}
\newarrowhead{LaTeX}\rhla\lhla\dhla\uhla

\def\lhlala{\lhla\kern.3em\lhla}
\def\rhlala{\rhla\kern.3em\rhla}
\def\uhlala{\hbox{\uhla\raise-.6ex\uhla}}
\def\dhlala{\hbox{\dhla\lower-.6ex\dhla}}
\newarrowhead{doubleLaTeX}\rhlala\lhlala\dhlala\uhlala

\def\hhO{\scriptaxis\bigcirc\mkern.4mu} \def\hho{{\circ}\mkern1.2mu}%
\newarrowhead{o}\hho\hho\circ\circ
\newarrowhead{O}\hhO\hhO{\scriptstyle\bigcirc}{\scriptstyle\bigcirc}

\def\rhtimes{\mkern-5mu{\times}\mkern-.8mu}\def\lhtimes{\mkern-.8mu{\times}%
\mkern-5mu}\def\uhtimes{\setbox0=\hbox{$\times$}\ht0\axisheight\dp0-\ht0%
\lower\ht0\box0 }\def\dhtimes{\setbox0=\hbox{$\times$}\ht0\axisheight\box0 }%
\newarrowhead{X}\rhtimes\lhtimes\dhtimes\uhtimes\newarrowhead+++++

\newarrowhead{Y}{\mkern-3mu\Yright}{\Yleft\mkern-3mu}\Ydown\Yup

\newarrowhead{->}\rightarrow\leftarrow\downarrow\uparrow

\newarrowhead{=>}\Rightarrow\Leftarrow{\@cmex7F}{\@cmex7E}

\newarrowhead{harpoon}\rightharpoonup\leftharpoonup\downharpoonleft
\upharpoonleft

\def\twoheaddownarrow{\rlap{$\downarrow$}\raise-.5ex\hbox{$\downarrow$}}
\def\twoheaduparrow{\rlap{$\uparrow$}\raise.5ex\hbox{$\uparrow$}}
\newarrowhead{->>}\twoheadrightarrow\twoheadleftarrow\twoheaddownarrow
\twoheaduparrow

\def\ltvee{\mkern-1mu{\lessthan}\mkern.4mu}
\newarrowtail{vee}\greaterthan\ltvee\vee\wedge

\newarrowtail{littlevee}{\scriptaxis\greaterthan}{\mkern-1mu\scriptaxis
\lessthan}{\scriptstyle\vee}{\scriptstyle\wedge}\ifx\boldmath\CD@qK
\newarrowtail{boldlittlevee}{\scriptaxis\greaterthan}{\mkern-1mu\scriptaxis
\lessthan}{\scriptstyle\vee}{\scriptstyle\wedge}\else\newarrowtail{%
boldlittlevee}{\boldscriptaxis\greaterthan}{\mkern-1mu\boldscriptaxis
\lessthan}{\boldscript\vee}{\boldscript\wedge}\fi

\newarrowtail{curlyvee}\succ{\mkern-1mu{\prec}\mkern.4mu}\curlyvee\curlywedge

\def\rttriangle{\mkern1.2mu\triangleright}
\newarrowtail{triangle}\rttriangle\lhtriangle\dhtriangle\uhtriangle
\newarrowtail{blacktriangle}\blacktriangleright{\mkern-1mu\blacktriangleleft
\mkern.4mu}\dhblack\uhblack\newarrowtail{littleblack}{\scriptaxis
\blacktriangleright\mkern-2mu}{\mkern-1mu\scriptaxis\blacktriangleleft}%
\dhlblack\uhlblack

\def\rtla{\hbox{\setbox0=\lnchar55\dimen0=\wd0\kern-.5\dimen0\ht0\z@\raise
\axisheight\box0\kern-.2\dimen0}}
\def\ltla{\hbox{\setbox0=\lnchar33\dimen0=\wd0\kern-.15\dimen0\ht0\z@\raise
\axisheight\box0\kern-.5\dimen0}}
\def\dtla{\vbox{\setbox0=\rlap{\lnchar77}\dimen0=\ht0\kern-.7\dimen0\box0%
\kern-.1\dimen0}}
\def\utla{\vbox{\setbox0=\rlap{\lnchar66}\dimen0=\ht0\kern-.1\dimen0\box0%
\kern-.6\dimen0}}
\newarrowtail{LaTeX}\rtla\ltla\dtla\utla

\def\rtvvee{\gg\mkern-3mu}
\def\ltvvee{\mkern-3mu\ll}
\def\dtvvee{\vbox{\hbox{$\vee$}\kern-.6ex \hbox{$\vee$}\vss}}
\def\utvvee{\vbox{\vss\hbox{$\wedge$}\kern-.6ex \hbox{$\wedge$}\kern\z@}}
\newarrowtail{doublevee}\rtvvee\ltvvee\dtvvee\utvvee

\def\ltlala{\ltla\kern.3em\ltla}
\def\rtlala{\rtla\kern.3em\rtla}
\def\utlala{\hbox{\utla\raise-.6ex\utla}}
\def\dtlala{\hbox{\dtla\lower-.6ex\dtla}}
\newarrowtail{doubleLaTeX}\rtlala\ltlala\dtlala\utlala

\def\utbar{\vrule height 0.093ex depth0pt width 0.4em}
\let\dtbar\utbar
\def\rtbar{\mkern1.5mu\vrule height 1.1ex depth.06ex width .04em\mkern1.5mu}%
\let\ltbar\rtbar
\newarrowtail{mapsto}\rtbar\ltbar\dtbar\utbar
\newarrowtail{|}\rtbar\ltbar\dtbar\utbar

\def\rthooka{\raisehook{}+\subset\mkern-1mu}
\def\lthooka{\mkern-1mu\raisehook{}+\supset}
\def\rthookb{\raisehook{}-\subset\mkern-2mu}
\def\lthookb{\mkern-1mu\raisehook{}-\supset}

\def\dthooka{\shifthook{}+\cap}
\def\dthookb{\shifthook{}-\cap}
\def\uthooka{\shifthook{}+\cup}
\def\uthookb{\shifthook{}-\cup}

\newarrowtail{hooka}\rthooka\lthooka\dthooka\uthooka\newarrowtail{hookb}%
\rthookb\lthookb\dthookb\uthookb

\ifx\boldmath\CD@qK\newarrowtail{boldhooka}\rthooka\lthooka\dthooka\uthooka
\newarrowtail{boldhookb}\rthookb\lthookb\dthookb\uthookb\newarrowtail{%
boldhook}\rthooka\lthooka\dthookb\uthooka\else\def\rtbhooka{\raisehook
\boldmath+\subset\mkern-1mu}
\def\ltbhooka{\mkern-1mu\raisehook\boldmath+\supset}
\def\rtbhookb{\raisehook\boldmath-\subset\mkern-2mu}
\def\ltbhookb{\mkern-1mu\raisehook\boldmath-\supset}
\def\dtbhooka{\shifthook\boldmath+\cap}
\def\dtbhookb{\shifthook\boldmath-\cap}
\def\utbhooka{\shifthook\boldmath+\cup}
\def\utbhookb{\shifthook\boldmath-\cup}
\newarrowtail{boldhooka}\rtbhooka\ltbhooka\dtbhooka\utbhooka\newarrowtail{%
boldhookb}\rtbhookb\ltbhookb\dtbhookb\utbhookb\newarrowtail{boldhook}%
\rtbhooka\ltbhooka\dtbhooka\utbhooka\fi

\def\dtsqhooka{\shifthook{}+\sqcap}
\def\ltsqhooka{\mkern-1mu\raisehook{}+\sqsupset}
\def\rtsqhooka{\raisehook{}+\sqsubset\mkern-1mu}
\def\utsqhooka{\shifthook{}+\sqcup}
\newarrowtail{sqhook}\rtsqhooka\ltsqhooka\dtsqhooka\utsqhooka

\newarrowtail{hook}\rthooka\lthookb\dthooka\uthooka\newarrowtail{C}\rthooka
\lthookb\dthooka\uthooka

\newarrowtail{o}\hho\hho\circ\circ
\newarrowtail{O}\hhO\hhO{\scriptstyle\bigcirc}{\scriptstyle\bigcirc}

\newarrowtail{X}\lhtimes\rhtimes\uhtimes\dhtimes\newarrowtail+++++

\newarrowtail{Y}\Yright\Yleft\Ydown\Yup

\newarrowtail{harpoon}\leftharpoondown\rightharpoondown\upharpoonright
\downharpoonright

\newarrowtail{<=}\Leftarrow\Rightarrow{\@cmex7E}{\@cmex7F}

\newarrowfiller{=}=={\@cmex77}{\@cmex77}
\def\vfthree{\mid\!\!\!\mid\!\!\!\mid}
\newarrowfiller{3}\equiv\equiv\vfthree\vfthree

\def\vfdashstrut{\vrule width0pt height1.3ex depth0.7ex}
\def\vfthedash{\vrule width\CD@LF height0.6ex depth 0pt}
\def\hfthedash{\CD@AJ\vrule\horizhtdp width 0.26em}
\def\hfdash{\mkern5.5mu\hfthedash\mkern5.5mu}
\def\vfdash{\vfdashstrut\vfthedash}
\newarrowfiller{dash}\hfdash\hfdash\vfdash\vfdash

\newarrowmiddle+++++

\iffalse
\newarrow{To}----{vee}
\newarrow{Arr}----{LaTeX}
\newarrow{Dotsto}....{vee}
\newarrow{Dotsarr}....{LaTeX}
\newarrow{Dashto}{}{dash}{}{dash}{vee}
\newarrow{Dasharr}{}{dash}{}{dash}{LaTeX}
\newarrow{Mapsto}{mapsto}---{vee}
\newarrow{Mapsarr}{mapsto}---{LaTeX}
\newarrow{IntoA}{hooka}---{vee}
\newarrow{IntoB}{hookb}---{vee}
\newarrow{Embed}{vee}---{vee}
\newarrow{Emarr}{LaTeX}---{LaTeX}
\newarrow{Onto}----{doublevee}
\newarrow{Dotsonarr}....{doubleLaTeX}
\newarrow{Dotsonto}....{doublevee}
\newarrow{Dotsonarr}....{doubleLaTeX}
\else
\newarrow{To}---->
\newarrow{Arr}---->
\newarrow{Dotsto}....>
\newarrow{Dotsarr}....>
\newarrow{Dashto}{}{dash}{}{dash}>
\newarrow{Dasharr}{}{dash}{}{dash}>
\newarrow{Mapsto}{mapsto}--->
\newarrow{Mapsarr}{mapsto}--->
\newarrow{IntoA}{hooka}--->
\newarrow{IntoB}{hookb}--->
\newarrow{Embed}>--->
\newarrow{Emarr}>--->
\newarrow{Onto}----{>>}
\newarrow{Dotsonarr}....{>>}
\newarrow{Dotsonto}....{>>}
\newarrow{Dotsonarr}....{>>}
\fi

\newarrow{Implies}===={=>}
\newarrow{Project}----{triangle}
\newarrow{Pto}----{harpoon}
\newarrow{Relto}{harpoon}---{harpoon}

\newarrow{Eq}=====
\newarrow{Line}-----
\newarrow{Dots}.....
\newarrow{Dashes}{}{dash}{}{dash}{}

\newarrow{SquareInto}{sqhook}--->

\newarrowhead{cmexbra}{\@cmex7B}{\@cmex7C}{\@cmex3B}{\@cmex38}
\newarrowtail{cmexbra}{\@cmex7A}{\@cmex7D}{\@cmex39}{\@cmex3A}
\newarrowmiddle{cmexbra}{\braceru\bracelu}{\bracerd\braceld}{\vcenter{%
\hbox@maths{\@cmex3D\mkern-2mu}}}
{\vcenter{\hbox@maths{\mkern2mu\@cmex3C}}}
\newarrow{@brace}{cmexbra}-{cmexbra}-{cmexbra}
\newarrow{@parenth}{cmexbra}---{cmexbra}
\def\rightBrace{\d@brace[thick,cmex]}
\def\leftBrace{\u@brace[thick,cmex]}
\def\upperBrace{\r@brace[thick,cmex]}
\def\lowerBrace{\l@brace[thick,cmex]}
\def\rightParenth{\d@parenth[thick,cmex]}
\def\leftParenth{\u@parenth[thick,cmex]}
\def\upperParenth{\r@parenth[thick,cmex]}
\def\lowerParenth{\l@parenth[thick,cmex]}

\let\hEq\rEq
\let\vEq\uEq

\def\labelstyle{
\ifincommdiag
\textstyle
\else
\scriptstyle
\fi}
\let\objectstyle\displaystyle

\newdiagramgrid{pentagon}{0.618034,0.618034,1,1,1,1,0.618034,0.618034}{1.%
17557,1.17557,1.902113,1.902113}

\newdiagramgrid{perspective}{0.75,0.75,1.1,1.1,0.9,0.9,0.95,0.95,0.75,0.75}{0%
.75,0.75,1.1,1.1,0.9,0.9}

\diagramstyle[
dpi=300,
vmiddle,nobalance,
loose,
thin,
pilespacing=10pt,%
shortfall=4pt,
]

\ifx\CD@qK\else\CD@PK\relax\CD@FF\CD@e\fi\fi

\CD@vE\CD@hK\else\fi\else\fi\cdrestoreat
\usepackage{amssymb}
\usepackage{mathrsfs}
\usepackage{cite}
\usepackage{tikz-cd}
\usepackage{MnSymbol}

\DeclareMathOperator\C{\mathbb C}
\DeclareMathOperator\Z{\mathbb Z}
\DeclareMathOperator\R{\mathbb R}

\newtheorem{theorem}{Theorem}[section]
\newtheorem{lemma}[theorem]{Lemma}
\newtheorem{cor}[theorem]{Corollary}
\newtheorem{prop}[theorem]{Proposition}
\theoremstyle{definition}
\newtheorem{definition}[theorem]{Definition}
\newtheorem{example}[theorem]{Example}

\theoremstyle{remark}
\newtheorem{remark}[theorem]{Remark}

\newcommand{\dontprint}[1]\relax

\newcommand{\Ind}{\operatorname{Ind}}
\newcommand{\Der}{\operatorname{Der}}
\newcommand{\Lie}{\operatorname{Lie}}

\newcommand{\La}{\Lambda}
\newcommand{\Om}{\Omega}
\renewcommand{\th}{\theta}
\newcommand{\hra}{\hookrightarrow}
\newcommand{\we}{\wedge}
\renewcommand{\P}{{\mathbb P}}
\newcommand{\A}{{\mathbb A}}
\newcommand{\wt}{\widetilde}
\newcommand{\ot}{\otimes}
\newcommand{\codim}{\operatorname{codim}}

\newcommand{\gr}{\operatorname{gr}}
\newcommand{\und}{\underline}
\newcommand{\Hom}{\operatorname{Hom}}

\newcommand{\Tor}{\operatorname{Tor}}

\newcommand{\TT}{{\mathcal T}}
\newcommand{\JJ}{{\mathcal J}}
\newcommand{\VV}{{\mathcal V}}
\newcommand{\DD}{{\mathcal D}}

\newcommand{\EE}{{\mathcal E}}
\newcommand{\FF}{{\mathcal F}}

\newcommand{\II}{{\mathcal I}}
\newcommand{\KK}{{\mathcal K}}
\newcommand{\LL}{{\mathcal L}}
\newcommand{\NN}{{\mathcal N}}
\newcommand{\OO}{{\mathcal O}}

\newcommand{\si}{\sigma}
\newcommand{\de}{\delta}
\newcommand{\sub}{\subset}

\newcommand{\ov}{\overline}

\newcommand{\om}{\omega}
\newcommand{\la}{\lambda}
\renewcommand{\a}{\alpha}
\renewcommand{\b}{\beta}

\newcommand{\id}{\operatorname{id}}

\newcommand{\End}{\operatorname{End}}
\newcommand{\lan}{\langle}
\newcommand{\ran}{\rangle}
\newcommand{\MM}{{\mathcal M}}
\renewcommand{\SS}{{\mathcal S}}
\renewcommand{\AA}{{\mathcal A}}
\newcommand{\HH}{{\mathcal H}}
\newcommand{\WW}{{\mathcal W}}

\newcommand{\De}{\Delta}

\newcommand{\ga}{\gamma}

\newcommand{\Pf}{\operatorname{Pf}}

\newcommand{\eps}{\epsilon}
\newcommand{\bbL}{{\mathbb L}}
\newcommand{\Ber}{\operatorname{Ber}}
\newcommand{\Sp}{\operatorname{Sp}}
\newcommand{\ad}{\operatorname{ad}}
\newcommand{\fg}{{\mathfrak g}}
\newcommand{\fc}{{\mathfrak c}}

\newcommand{\fp}{{\mathfrak p}}
\renewcommand{\sp}{{\mathfrak sp}}
\newcommand{\Ga}{\Gamma}

\newcommand{\pa}{\partial}
\newcommand{\bos}{\operatorname{bos}}
\newcommand{\str}{\operatorname{str}}

\numberwithin{equation}{section}

\title{Projective connections on super Heisenberg coinvariants. I}

\author{Giovanni Felder}
\address{Department of mathematics,
ETH Zurich, 8092 Zurich, Switzerland}
\email{giovanni.felder@math.ethz.ch}
\author{David Kazhdan}
\address{Einstein Institute of Mathematics,
The Hebrew University of Jerusalem,
Jerusalem 91904, Israel}
\email{kazhdan@math.huji.ac.il}
\author{Alexander Polishchuk}
\address{
    Department of Mathematics, 
    University of Oregon, 
    Eugene, OR 97403, USA; National Research University Higher School of Economics, Moscow, Russia
  }
  \email{apolish@uoregon.edu}

\begin{document}

\begin{abstract}
We study derived coinvariants of isotropic subbundles on modules over super Heisenberg algebras and construct and study certain natural transitive Lie algebroids
acting on them.
\end{abstract}

\maketitle

\section{Introduction}

This paper arose from our attempt to understand projective connections on 
super Heisenberg conformal blocks over
the moduli spaces of SUSY supercurves, which led us to revisit similar structures for classical (even) Heisenberg coinvariants.
The traditional approach for defining projective connections on coinvariants sheaves is via vertex algebras. 
We discovered that there is a simple and natural construction of these projective connections, which uses only finite-dimensional
symplectic and Heisenberg bundles (which come from local data near the punctures).
The goal of this paper is to present the underlying abstract construction of projective connections.
In the sequel we will apply this theory to sheaves of (super) Heisenberg coinvariants on the moduli spaces of (SUSY) curves,
and to $bc/\beta\gamma$-systems.

The central objects of our study are sheaves of derived coinvariants $C_\II(M)$, 
for a module $M$ over a bundle of Heisenberg Lie (super) algebras $\HH$ and an isotropic subbundle $\II\sub \HH$ (over a superscheme).
Among the general results about these coinvariants is their compatibility with the isotropic reduction, 
as well as the relation with the Berezinian bundles under certain transversality assumptions
(see Theorem \ref{coinv-red-thm} and Theorem \ref{Ber-coinv-thm}). For example, in the case when $\II=\LL$ is a Lagrangian subbundle,
and $M=M(\LL')$ is the Fock module of the Heisenberg algebra associated with another Lagrangian subbundle $\LL'$, such that $\LL\cap \LL'$ is locally free, 
then the derived coinvariants are concentrated in one degree (equal to the even rank of $\LL\cap \LL'$) and are isomorphic to $\Ber(\LL\cap \LL')$.

Given a family of isotropic subspaces $\II$ in a constant Heisenberg (super) algebra $H$, parametrized by a superscheme $S$, we construct a transitive Lie algebroid
$\AA_\II$ acting on all sheaves of (derived) coinvariants $C_\II(M)$, where $M$ is a module of weight one over $H$. 
We give two constructions of $\AA_\II$: one as a Lie subalgebroid in the Lie algebroid of infinitesimal symmetries of 
$M^r(\II)$, the right Fock module over $H$ associated with $\II$ (see Sec.\ \ref{constr-main-sec}), 
and another as a pullback of a universal Lie algebroid over the isotropic Grassmannian (see Sec.\ \ref{another-constr-sec}).
More generally, we construct a Lie algebroid $\AA_{\HH,\nabla,\II}$ in the case when the Heisenberg bundle $\HH$ is not constant
but is equipped with a flat connection $\nabla$ (see Sec.\ \ref{flat-Heis-sec}).

As shown in \cite[Lem.\ 2.3.2]{BB}, one can view a flat projective connection as 
a structure of a module over a Picard algebroid, i.e., a transitive Lie algebroid such that the kernel of the anchor map is identified with $\OO$. 
In the case when $\II=\LL$ is Lagrangian, $\AA_\LL$ (or more generally, $\AA_{\HH,\nabla,\LL}$) is a Picard algebroid, so for any Heisenberg module $M$, the derived coinvariants $C_\LL(M)$ get equipped with projective connections. 
For a general isotropic subbundle $\II$, the kernel of the anchor map is $U_{\le 2}(\ov{\HH}_\II)$, where $\ov{\HH}_\II$ is the sheaf of reduced Heisenberg 
algebras associated with $\II$. We show that one can define a Picard subalgebroid $\AA^{Pic}_\II\sub \AA_\II$ in the case when 
$\ov{\HH}_\II$ is equipped with a flat connection 
(see Sec.\ \ref{isotropic-PL-sec}). Such a structure is indeed present in the example corresponding to Heisenberg conformal blocks on the moduli spaces of curves (or SUSY supercurves), where 
$\ov{\HH}_\II$ can be identified with the Heisenberg algebra of topological origin. 

We establish such important properties of the Lie algebroids $\AA_{\HH,\nabla,\II}$ as compatibility with isotropic reduction (see Sec.\ \ref{alg-isotr-red-sec}),
as well as an identification of $\AA_{\HH,\nabla,\LL}$ with the Atiyah algebroid of the square root of the Berezinian in the case when $\LL\sub \VV$ is Lagrangian, and the Heisenberg extension is split (see Theorem \ref{LG-alg-Ber-thm}). Note that a simple proof of the latter result uses derived coinvariants in a crucial way.
We also give a natural construction of local flat connections on $\AA_\II$, and develop a setup for constructing equivariant structures (see Sec.\ \ref{equiv-sec}), 
which will
be used in the case of the moduli of (super-)curves to get rid of a choice of formal parameters at the marked points.
Using local flat connections we give a general formula for the class of the Picard algebroid $\AA_{\HH,\nabla,\LL}$ associated with a Lagrangian $\LL$ in the case when the Heisenberg extension is not necessarily split (see Theorem \ref{alg-class-thm}).

As an application of our Picard algebroids, we prove a general criterion for vanishing of higher derived coinvariants $C_{\LL_1}(M(\LL_2))$,
for a pair of Lagrangian subbundles $\LL_1$ and $\LL_2$, under a certain condition on codimensions of loci with given dimensions of intersection of $\LL_1|_s\cap \LL_2|_s$ (see Theorem \ref{codim-thm}).

The paper is organized as follows. In Sec.\ \ref{der-coinv-sec} we discuss some general results concerning (derived) coinvariants of modules $C_I(M(J))$
over bundles of Heisenberg algebras (where $I$ and $J$ are isotropic subbundles in the Heisenberg algebra).
In particular, in Sec.\ \ref{coinv-basic-sec}--\ref{redII-sec} we discuss the behavior under isotropic reduction and 
calculate coinvariants in the case of a good relative position of $I$ and $J$. In Sec.\ \ref{odd-rk-Pf-sec} we study what happens in the case of the purely odd symplectic bundle,
and show how this leads to a generalization of the Pfaffian (see Proposition \ref{gen-Pf-prop}). Then in Sec.\ \ref{odd-intersection-sec} we study the case when the even parts
of $I$ and $J$ do not intersect (which is relevant for the Heisenberg coinvariants over the moduli of SUSY supercurves). We prove a general result (Theorem \ref{Pf-thm}) on
how the corresponding coinvariants look locally. In Sec.\ \ref{odd-int-sec} we specialize to the case of coinvariants $C_{L_1}(M(L_2))$ where $L_1$ and $L_2$ are Lagrangians whose
even parts have trivial intersection. We show that in this case the line bundle $C_{L_1}(M(L_2))$ with its natural global section is the square root of a certain Berezinian line bundle
and its global section studied in \cite[Thm.\ 4.5]{FKP-reg}.

In Sec.\ \ref{Lie-algebroid-sec} we study transitive Lie algebroids $\AA_\II$ associated with isotropic subbundles $\II$, which act on (derived) $I$-coinvariants of all Heisenberg modules.
In Sec.\ \ref{constr-main-sec} and \ref{another-constr-sec}, we consider isotropic
subbundles in a constant Heisenberg algebra $\HH$, while in Sec.\ \ref{flat-Heis-sec} we generalize the construction to the case when $\HH$ is only locally constant (i.e., is
equipped with a flat connection). In Sec.\ \ref{alg-isotr-red-sec} we check that in some situations our Lie algebroid does not change under the isotropic reduction.
In Sec.\ \ref{Ber-sec} we prove that in the case of a split Heisenberg extension $\HH=\OO\oplus \VV$, where $\VV$ is equipped with a flat (symplectic) connection,
and of a Lagrangian subbundle $\LL\sub \VV$, our Lie algebroid is naturally isomorphic to the Lie algebroid associated with the square root of the Berezinian $\Ber(\LL)$.
In Sec.\ \ref{isotropic-PL-sec}, assuming that the reduced Heisenberg algebra $\ov{\HH}_\II$ is equipped with a flat connection, we give a construction of a (transitive)
Lie subalgebroid in $\AA_\II$, which is a Picard algebroid. In Sec.\ \ref{local-flat-conn-sec} we construct a flat connection on $\AA_\II$ associated with a (locally) constant
isotropic subspace $J$ such that $J^\perp$ is complementary to $\II$. In Sec.\ \ref{alg-class-sec} we give a general formula for the class of the Picard algebroid $\AA_{\HH,\nabla,\LL}$ associated
with a Lagrangian $\LL\sub \HH$ over $S$ in $H^1(S,\Om^{1,cl})$ (assuming existence of local Lagrangian complements preserved by the connection).
In Sec.\ \ref{holonomicity-sec}, as an application of our Lie algebroid's action on coinvariants,
we prove a general criterion for vanishing of the higher derived coinvariants (Theorem \ref{codim-thm}).

In Sec. \ref{equiv-sec} we discuss the natural way in which our data of Lie algebroids acting on coinvariants can be equipped with equivariant data, with respect to an action
of a (super) group.

In appendix \ref{P1-example-sec} we compute explicitly the class of our Picard algebroid over the Lagrangian Grassmannian of subspaces in the (even) 3-dimensional
Heisenberg algebra. In appendix \ref{Pf-app} we prove Prop.\ \ref{gen-Pf-prop}(iv) from Sec.\ \ref{odd-rk-Pf-sec} giving a formula for the generalized Pfaffian.

\bigskip

\noindent
{\it Conventions}. For a homogeneous symbol $x$, its parity is denoted by $\ov{x}$.
We use the parity of differential forms such that the de Rham differential is even, so the de Rham complex can be viewed as a complex of $\Z/2$-graded sheaves.
For $\eta\in \Om^r$ and $\th\in \Om^s$, we have $\eta\we\th=(-1)^{rs+\ov{\eta}\ov{\th}}\th\we\eta$.

\bigskip

\noindent
{\it Acknowledgments}. A.P. is grateful to Alexander Braverman, Pavel Etingof and Mikhail Kapranov for useful discussions.

\section{Heisenberg derived coinvariants}\label{der-coinv-sec}

\subsection{Super Weyl algebras}

Let $\VV$ be a symplectic super vector bundle over a superscheme $S$, i.e., a (finite rank) supervector bundle over $S$, equipped with an (even) nondegenerate skew-symmetric $\OO$-bilinear form $(\cdot,\cdot)$ (we say that a pairing is nondegenerate if it induces an isomorphism $\VV\to \VV^\vee$).
We denote by $\WW(\VV)$ the corresponding super Weyl algebra, which is a quasicoherent sheaf of $\OO$-algebras on $S$, given by the defining relations
$$a\cdot b- (-1)^{\bar{a}\bar{b}} b\cdot a=(a,b),$$
where $a,b\in \VV$. Note that the Weyl algebra $\WW(\VV)$ has an additional, {\it Weyl-degree $\Z/2$-grading}, in which elements of $\VV$ have odd degree.

More generally, we can consider a {\it super Heisenberg $\OO_S$-Lie algebra} $\HH$ over $S$, i.e., a super vector bundle fitting into an extension
$$0\to \OO_S\to \HH\rTo{\pi} \VV\to 0,$$
equipped with an $\OO$-linear super Lie bracket, such that $\OO_S$ is in the center, the induced bracket on $\HH/\OO_S=\VV$ is abelian, and the (super) commutator
pairing $\VV\times \VV\to \OO_S$ is nondegenerate. We denote by $1_\HH\in \HH$ the generator of the center $\OO_S\sub \HH$.
The commutator pairing equips $\VV$ with a structure of a symplectic super vector bundle.

We denote by $U(\HH)$ the quotient of the universal enveloping algebra of $\HH$ by $1_\HH-1$.
A choice of an $\OO_S$-linear splitting $\VV\to \HH$ gives identifications 
$$\WW(\VV)\simeq U(\HH), \ \ \WW_{\le 1}(\VV)\simeq \HH.$$
Note that there is a natural isomorphism
$$U(\HH)^{op}\simeq U(\HH^{op}),$$
where $\HH^{op}$ is $\HH$ with the bracket $-[\cdot,\cdot]$. 
 
An {\it isotropic subbundle} $I\sub \VV$ is a subbundle such that $(\cdot,\cdot)$ is zero on $I\times I$.
A {\it Lagrangian subbundle} $L\sub \VV$ is an isotropic subbundle such that the form induces an isomorphism $\VV/L\simeq L^\vee$.

We say that a subbundle $I\sub \HH$
is isotropic (resp., Lagrangian), if $\pi|_{I}$ is an isomorphism onto an isotropic (resp., Lagrangian) subbundle of $\VV$. Note that a subbundle $I\sub \HH$, transversal to $\OO_S$, 
is isotropic if and only if it is an abelian subalgebra of $\HH$.

With an isotropic subbundle $I\sub \HH$
we associate sheaves of left and right 
$U(\HH)$-modules on $S$,
$$M_\HH(I)=M(I):=U(\HH)/U(\HH)\cdot I, \ \ M^r_\HH(I)=M^r(I):=U(\HH)/I\cdot U(\HH).$$
In the case $\HH=\OO_S\oplus \VV$ and $I\sub \VV$, we also write $M_{\VV}(I)$ and $M^r_{\VV}(I)$ for these.
The module $M(I)$ (resp., $M^r(I)$) has a canonical {\it vacuum vector} $1\in M(I)$ (resp., $1\in M^r(I)$), the image of $1\in U(\HH)$.
In the case when $\HH=\OO_S\oplus \VV$ and $I\sub \VV$, $M(I)$ (resp., $M^r(I)$) inherits the $\Z/2$-grading.
Note that we have an isomorphism
$$M^r_\HH(I)\simeq M_{\HH^{op}}(I)$$ 
of left modules over $U(\HH)^{op}\simeq U(\HH^{op})$.

In the case $I=L$ is a Lagrangian subbundle, one usually calls $M_{\HH}(L)$ the {\it Fock module associated with } $L$.

With an isotropic subbundle $I\sub \HH$, we associate the reduced super Heisenberg $\OO_S$-Lie algebra,
$$\ov{\HH}_I=\ov{\HH}:=\fc(I)/I,$$
where $\fc(I)\sub\HH$ is the centralizer of $I$ in $\HH$. Note that $\fc(I)=\pi^{-1}(\pi(I)^\perp)$, where $\pi(I)^\perp\sub\VV$
is the orthogonal of $\pi(I)$ with respect to the symplectic pairing, so $\ov{\HH}$ fits into an exact sequence,
$$0\to\OO_S\to \ov{\HH}\to \pi(I)^\perp/\pi(I)\to 0.$$

In the case when $\HH=\OO_S\oplus \VV$ and $I\sub \VV$, we have
$\ov{\HH}=\OO_S\oplus I^\perp/I$.

\subsection{Derived coinvariants: statements of basic results}\label{coinv-basic-sec}

Until the end of Section \ref{der-coinv-sec}, we fix a Heisenberg
$\OO_S$-Lie algebra $\HH$ over a superscheme $S$.

For an isotropic subbundle $I\sub \HH$, we define the derived $I$-coinvariants of a left $U(\HH)$-module $M$ by 
$$C_I(M):=M^r(I)\otimes^{\bbL}_{U(\HH)} M$$
(this is an object in the derived category of quasicoherent sheaves of $\OO_S$-modules).



\begin{lemma} $C_I(M)$ is represented by the complex
\begin{equation}\label{C-I-M-complex}
K(I,M)=K_{\HH}(I,M):{\bigwedge}^\bullet(I)\ot M: \ldots\to {\bigwedge}^2(I)\ot M\to I\ot M\to M
\end{equation}
with the differential 
$$d(x_1\we\ldots\we x_n\ot m)=\sum_i(-1)^{i+\ov{x}_i(\ov{x}_{i+1}+\ldots+\ov{x}_n)} x_1\we\ldots\widehat{x_i}\ldots \we x_n\ot x_i(m).$$
\end{lemma}

\begin{proof}
It is enough to check that for $M=U(\HH)$ this complex has no cohomology in degree $\neq 0$.
This is a local statement, so we can assume $\HH$ is split: $\HH=\OO_S\oplus \VV$, where $I\sub \VV$, and
$\VV=I\oplus I^\vee\oplus \VV'$, where the symplectic structure is the sum of the standard symplectic structure on $I\oplus I^\vee$
and of a symplectic structure on $\VV'$. Then $U(\HH)\simeq U(I\oplus I^\vee)\ot_\OO U(\VV')$, so it is enough to
prove the assertion in the case $\VV=I\oplus I^\vee$. Since $U(I\oplus I^\vee)\simeq S(I)\ot_\OO S(I^\vee)$ as
an $S(I)$-module, the assertion reduces to the exactness of the Koszul complex $\bigwedge^\bullet(I)\ot S(I)$ in degrees $\neq 0$.
\end{proof}


In the case when $S$ is a point, for a pair of Lagrangian subbundles $L,L'\sub \VV$, $C_L(M(L'))$ is $1$-dimensional and lives in degree $-\dim(L_-\cap L'_-)$.
We will prove the following analog of this result over superschemes.

\begin{theorem}\label{Ber-coinv-thm}
Let $J,J'\sub \HH$ be isotropic subbundles 
such that $\pi(J\cap J')=\pi(J)\cap \pi(J')$ is a subbundle in $\VV$, and $\pi(J\cap J')^\perp=\pi(J)^\perp+\pi(J')$. 
Then one has an isomorphism in the derived category of $\ov{\HH}_J$-modules,
$$C_J(M(J'))\simeq M_{\ov{\HH}_J}((\fc(J)\cap J')/(J\cap J'))\ot \Ber(J\cap J')[n],$$
where the rank of $J\cap J'$ is $n|m$ (and $(\fc(J)\cap J')/(J\cap J')$ is viewed as an isotropic subbundle in $\ov{\HH}_J=\fc(J)/J$).
\end{theorem}

Let us point out two particular cases for modules associated with Lagrangians.

\begin{cor}\label{JL-cor}
(i) Let $J\sub \HH$ be an isotropic subbundle, $L\sub \HH$ be a Lagrangian subbundle such that $\fc(J)+L=\HH$ (equivalently, $\pi(J)+\pi(L)=\VV$). Then
$$C_J(M(L))\simeq M_{\ov{\HH}_J}(L\cap \fc(J)).$$

\noindent
(ii) If $L,L'\sub \HH$ are Lagrangian subbundles 
such that $\pi(L\cap L')=\pi(L)\cap \pi(L')$, and $\pi(L)+\pi(L')$ is a subbundle in $\VV$, then 
$$C_L(M(L'))\simeq \Ber(L\cap L')[n],$$
where the rank of $L\cap L'$ is $n|m$.
\end{cor}

The following isotropic reduction result is also very useful.

\begin{theorem}\label{coinv-red-thm}
Let $I\sub \HH$ be an isotropic subbundle, and let $J\sub \HH$ be an isotropic subbundle such that 
$\fc(I)+J=\HH$.
Then $\ov{J}:=J\cap \fc(I)$ is an isotropic subbundle in $\ov{\HH}_I=\fc(I)/I$, and for any $U(\HH)$-module on which
$I$ acts locally nilpotently,
there is a natural isomorphism
\begin{equation}\label{redII-conv-isom}
C_{\ov{J}}(M^I)\rTo{\sim} C_{J}(M)
\end{equation}
given by the chain map of complexes
$${\bigwedge}^\bullet(J\cap \fc(I))\ot M^I\to {\bigwedge}^\bullet(J)\ot M.$$
induced by the natural embeddings $J\cap \fc(I)\to J$, $M^I\to M$.
\end{theorem}

\subsection{Some general observations}

Recall that for an isotropic subbundle $I\sub \HH$, we consider the centralizer $\fc(I)\sub \HH$ and the reduced Heisenberg algebra $\ov{\HH}=\fc(I)/I$.
We can also view $\fc(I)$ as a degenerate Heisenberg algebra and consider the corresponding subalgebra $U(\fc(I))\sub U(\HH)$, which is
equipped with a homomorphism $U(\fc(I))\to U(\ov{\HH})$.

For any $U(\HH)$-module, we set 
$$M^I:=\{ x\in M \ | \ I\cdot x=0\},$$
which is equipped with a natural $U(\ov{\HH})$-structure.

\begin{lemma}\label{module-red-lem}
(i) One has an isomorphism of algebras 
\begin{equation}\label{U-red-isom}
U(\ov{\HH})\simeq U(\fc(I))/I\cdot U(\fc(I)), 
\end{equation}
where $I$ is central in $U(\fc(I))$.
This induces an isomorphism of left (resp., right) $U(\HH)$-modules,
\begin{equation}\label{M-I-main-isom}
M(I)\simeq U(\HH)\ot^{\bbL}_{U(\fc(I))} U(\ov{\HH}), \ \ (\text{resp.,} \ M^r(I)\simeq U(\ov{\HH})\ot^{\bbL}_{U(\fc(I))} U(\HH) \ ),
\end{equation}
which equips $M(I)$ (resp., $M^r(I)$) with a structure of a 
$U(\HH)$--$U(\ov{\HH})$-bimodule (resp., $U(\ov{\HH})$--$U(\HH)$-bimodule).

\noindent
(ii) If $I\sub J$, where $J\sub \HH$ is isotropic, then there are canonical isomorphisms
\begin{equation}\label{M-J-red-isom}
M(J)\simeq U(\HH)\ot^{\bbL}_{U(\fc(I))} M_{\ov{\HH}}(\ov{J})\simeq M(I)\ot^{\bbL}_{U(\ov{\HH})} M_{\ov{\HH}}(\ov{J}),
\end{equation}
where $\ov{J}:=J/I$ is the reduced isotropic subbundle in $\ov{\HH}$, and $M(\ov{J})$ is the corresponding $U(\ov{\HH})$-module.
\end{lemma}

\begin{proof} (i) The isomorphism \eqref{U-red-isom} follows immediately from the Lie algebra isomorphism
$\fc(I)/I\simeq \ov{\HH}$. Since $U(\HH)$ is locally free as left (resp., right) $U(\fc(I))$-module, there are no $\Tor$'s in \eqref{M-I-main-isom},
so \eqref{M-I-main-isom} follows immediately from \eqref{U-red-isom}.

\noindent
(ii) The second isomorphism in \eqref{M-J-red-isom} follows immediately from \eqref{M-I-main-isom}. Since $U(\HH)$ is locally free as a right $U(\fc(I))$-module,
there are no higher $\Tor$'s, while the isomorphism
$$M(J)\simeq M(I)/M(I)\ov{J}\simeq M(I)\otimes_{U(\ov{H})}M(\ov{J})$$
follows easily from the definition of the right $U(\ov{H})$-action on $M(I)$.
\end{proof}

Using the above lemma, we can view $C_I(M)$ as an object in the derived category of left $U(\ov{\HH})$-modules, and we
get an identification of functors,
\begin{equation}\label{coinv-main-formula}
C_I(M)\simeq U(\ov{\HH})\otimes^{\bbL}_{U(\fc(I))}M.
\end{equation}

We have the following transitivity for the derived coinvariants functors.

\begin{prop}\label{trans-coinv-prop}
Let $J\sub \HH$ be an isotropic subbundle, and let $I\sub J$ be a subbundle. 
Let $\ov{J}:=J/I\sub \fc(I)/I=\ov{\HH}_I$ denote the reduced isotropic subbundle. Then under the natural isomorphism 
$$\fc_{\ov{\HH}_I}(\ov{J})/\ov{J}\simeq \fc_{\HH}(J)/J=\ov{\HH}_J,$$
we have a natural isomorphism in the derived category of $U(\ov{\HH}_J)-U(\HH)$-bimodules,
$$M^r_{\HH}(J)\simeq M^r_{\ov{\HH}_I}(\ov{J})\ot^{\bbL}_{U(\ov{\HH}_I)}M^r_{\HH}(I).$$
Hence, we have an isomorphism of functors from the derived category of $U(\HH)$-modules to that of $U(\ov{\HH}_J)$-modules,
$$C_J(M)\simeq C_{\ov{J}}(C_I(M)),$$
which is compatible with the canonical maps from $M$.
\end{prop}

\begin{lemma}\label{trans-bimod-lem}
In the situation of Proposition \ref{trans-coinv-prop}, the natural maps $\fc_{\ov{\HH}_I}(\ov{J})\to \ov{\HH}_I$ and $\fc_{\HH}(I)\to \ov{\HH}_I$, induce an isomorphism
$$U(\fc_{\ov{\HH}_I}(\ov{J}))\otimes_{U(\fc_{\HH}(J))}^{\bbL}U(\fc_{\HH}(I))\rTo{\sim} U(\ov{\HH}_I).$$
\end{lemma}

\begin{proof} The statement is local, so we can assume that $\HH=\OO_S\oplus\VV$, $J\sub \VV$, 
$\VV=J\oplus J^*\oplus \ov{\VV}$, and $J=I\oplus I'$. The statement immediately
reduces to the case $\ov{\VV}=0$, so we can assume $\VV=I\oplus I'\oplus I^*\oplus (I')^*$, with the standard form. Then
$$J^\perp=J=I\oplus I', \ \ I^\perp=I\oplus I'\oplus (I')^*, \ \ \ov{J}^\perp=I'.$$
It follows that $\WW(I^\perp)$ is a free module over $\WW(J^\perp)$, more precisely, the natural map of $\WW(J^\perp)$-modules,
$$\WW(J^\perp)\ot S((I')^*)\to \WW(I^\perp)$$
is an isomorphism. Hence, there are no higher derived functors, and the embedding $S((I')^*)\sub \WW(I^\perp)$ induces an isomorphism
$$S(I')\ot_{\OO} S((I')^*)\rTo{\sim} S(I')\ot_{\WW(J^\perp)} \WW(I^\perp).$$
Now the assertion follows from the fact that the natural embeddings induce an isomorphism
$$S(I')\ot_{\OO} S((I')^*)\rTo{\sim}\WW(I'\oplus (I')^*).$$
\end{proof}

\begin{proof}[Proof of Proposition \ref{trans-coinv-prop}]
Applying \eqref{coinv-main-formula}, we can rewrite 
$$C_{\ov{J}}(C_I(M))\simeq U(\fc_{\ov{\HH}_I}(\ov{J})/\ov{J})\ot^{\bbL}_{U(\fc_{\ov{\HH}_I}(\ov{J}))} U(\fc(I)/I)\ot^{\bbL}_{U(\fc(I))}M.$$
Now we can apply Lemma \ref{trans-bimod-lem} to rewrite $U(\ov{\HH}_I)$ and get
\begin{align*}
&C_{\ov{J}}(C_I(M))\simeq U(\fc_{\ov{\HH}_I}(\ov{J})/\ov{J})\ot^{\bbL}_{U(\fc_{\ov{\HH}_I}(\ov{J}))}U(\fc_{\ov{\HH}_I}(\ov{J}))\otimes_{U(\fc(J))}^{\bbL}U(\fc(I))\ot^{\bbL}_{U(\fc(I))}M\\
&\simeq U(\fc_{\ov{\HH}_I}(\ov{J})/\ov{J})\ot^{\bbL}_{U(\fc(J))}M\simeq U(\fc(J)/J)\ot^{\bbL}_{U(\fc(J))}M,
\end{align*}
which is isomorphic to $C_J(M)$.
\end{proof}

\subsection{An equivalence of categories}

Let $I\sub \HH$ be an isotropic subbundle, and let $\ov{\HH}=\ov{\HH}_I$ be the corresponding reduced Heisenberg Lie algebra.

\begin{prop}\label{cat-eq-prop} 
Let $U(\HH)-\mod^I$ denote the category of $U(\HH)$-modules on which $I$ acts locally nilpotently. Then 
the functors 
$$U(\HH)-\mod^I\to U(\ov{\HH})-\mod:M\mapsto M^I,  \text{ and}$$ 
$$U(\ov{\HH})-\mod\to U(\HH)-\mod^I:\ov{M}\mapsto \Ind_{\ov{\HH}}^{\HH}(\ov{M}):=U(\HH)\ot_{U(\fc(I))} \ov{M}$$ 
are mutually inverse equivalences of categories.
Under this equivalence, for an isotropic $J\sub \HH$, such that $I\sub J$, the module 
$M_{\HH}(J)\in U(\HH)-\mod^I$ corresponds to $M_{\ov{H}}(J/I)$. 
\end{prop}

\begin{proof}
Since $M^I=\Hom_{U(\fc(I))}(U(\ov{\HH}),M)$, it is clear that these two functors are adjoint. We have to check that the natural maps
\begin{equation}\label{adj-unit-map}
\ov{M}\to \Ind_{\ov{\VV}}^{\VV}(\ov{M})^I,
\end{equation}
\begin{equation}\label{adj-counit-map}
\Ind_{\ov{\VV}}^{\VV}(M^I)\to M
\end{equation}
are isomorphisms. The question is local, so we can assume that $\HH=\OO_S\oplus\VV$, $I\sub \VV$, and there exists an isotropic subbundle $I'\sub \VV$ complementary to $I^\perp$.
Then we have a natural identification
$$\Ind_{\ov{\VV}}^{\VV}(\ov{M})\simeq S(I')\ot_{\OO} \ov{M},$$
induced by the embedding $S(I')\sub W(\VV)$. Since $I$ acts trivially on $\ov{M}$, the identification $S(I^{\vee})^I=\OO$ implies that \eqref{adj-unit-map}
is an isomorphism.

We claim that the map \eqref{adj-counit-map} is surjective (provided $I$ acts locally nilpotently on $M$). Indeed, this follows from the proof of the (super version of) Kashiwara
theorem about modules over the Weyl algebra supported on a closed subvariety.
Indeed, locally we can choose sections $(x_1,\ldots,x_n,y_1,\ldots,y_n)$, such that $x_1,\ldots,x_n$ is a basis of $I$, $(y_i,y_j)=0$ and $(y_i,x_j)=\de_{ij}$.
Then we can view $M$ as a module over the super Weyl algebra in $n$ variables (where some of the variables are even and some are odd), such that
$x_1,\ldots,x_n$ act locally nilpotently. By Kashiwara's theorem, $M=\OO[y]\cdot M^I$, which implies the surjectivity we want.

Let $K\sub \Ind_{\ov{\VV}}^{\VV}(M^I)$ be the kernel of \eqref{adj-counit-map}. Then $K^I=0$ and $I$ acts locally nilpotently on $K$, hence $K=0$.
Thus, the map \eqref{adj-counit-map} is injective.

The last assertion follows immediately from the isomorphism \eqref{M-J-red-isom}.
\end{proof}

\begin{cor}\label{cat-eq-cor} 
For any $M\in U(\HH)-\mod^I$, the natural morphism
$$U(\HH)\ot_{U(\fc(I))} M^I\to M$$ 
is an isomorphism. For an isotropic $J\sub \HH$, containing $I$, one has
$$M(J)^I\simeq M_{\ov{\HH}_I}(J/I).$$
\end{cor}


In the case when $I$ has purely odd rank, it automatically acts nilpotently on every $U(\HH)$-module, so we deduce the following result. 

\begin{cor}\label{odd-rk-cat-eq-cor}
Let $I\sub \HH$ be an isotropic subbundle of rank $(0|n)$, and let $\ov{\HH}=\ov{\HH}_I$. Then the functors
$M\mapsto M^I$ and $\ov{M}\mapsto U(\HH)\ot_{U(\fc(I))} \ov{M}$ 
are mutually inverse equivalences between the categories $U(\HH)-\mod$ and $U(\ov{\HH})-\mod$.
\end{cor}

\subsection{Isotropic reduction I}\label{redI-sec}

We will deduce Theorem \ref{Ber-coinv-thm} from the following more general statement involving isotropic reduction.

\begin{theorem}\label{Ber-isotropic-thm}
Let $I\sub \HH$
be an isotropic subbundle of rank $n|m$, and let $\ov{\HH}=\ov{\HH}_I$ be the corresponding reduced Heisenberg algebra.
Then for a $U(\HH)$-module $M$ on which $I$ acts locally nilpotently
we get an isomorphism in the derived category of $U(\ov{\HH})$-modules,
$$C_I(M)\simeq M^I\ot_{\OO_S} \Ber(I)[n].$$
Furthermore, if $I\sub J$, where $J\sub \HH$ is an isotropic subbundle, then we get a canonical isomorphism
\begin{equation}\label{der-coinv-red-eq}
C_J(M)\simeq C_{\ov{J}}(M^I)\ot_{\OO_S} \Ber(I)[n],
\end{equation}
where $\ov{J}:=J/I\sub \ov{\HH}$.
\end{theorem}



\begin{lemma}\label{C-I-M-I-lem}
(i) In the situation of Theorem \ref{Ber-isotropic-thm}, locally there is a canonical isomorphism
in the homotopy category of complexes
$$K(I,M(I))\simeq {\bigwedge}^\bullet(I)\ot S(I^\vee)\ot U(\ov{\HH}),$$
(see \eqref{C-I-M-complex}), where the differential on the right is $d_I\ot \id$, with
$$d_I:{\bigwedge}^m(I)\ot S(I^\vee)\to {\bigwedge}^{m-1}(I)\ot S(I^\vee)$$
being the differential induced by the canonical map $I\to \Der_{\OO}(S(I^\vee),S(I^\vee))$.

\noindent
(ii) There is a canonical isomorphism
$$C_I(M(I))\simeq U(\ov{\HH})\ot_{\OO}\Ber(I)[n].$$
\end{lemma}

\begin{proof}
(i) Locally we can choose an isotropic subbundle $I'\sub \HH$, complementary to $\fc(I)$, so that the pairing induces an isomorphism $I'\simeq I^\vee$.
Then the embedding $S(I')\sub U(\HH)$ induces an isomorphism $S(I')\ot_\OO U(\fc(I))\simeq U(\HH)$, and hence by \eqref{M-I-main-isom}, an isomorphism
$$S(I^\vee)\ot_\OO U(\ov{\HH})\rTo{\sim} M(I),$$
which leads to an isomorphism of complexes
$${\bigwedge}^\bullet(I)\ot S(I^\vee)\ot U(\ov{\HH})\rTo{\a_{I'}} K(I,M(I)).$$

We view $U(\ov{\HH})$ as an $S(I')$-module via the isotropic embedding $I'\hra \fc(I)/I=\ov{\HH}$. 
Thus, $N:=S(I^\vee)\ot U(\ov{\HH})$ can be viewed as a module over $S(I)\ot S(I')$, where $S(I)$ acts on $S(I^\vee)$ and $S(I')$ acts on $U(\ov{\HH})$.
We will need a certain basic homotopy relation for the complex ${\bigwedge}^\bullet(I)\ot N$ (with the differential induced by $d_I$). 
For homogeneous elements
$u\in I$ and $v\in I\oplus I'$ of the same parity consider the operator on ${\bigwedge}^\bullet(I)\ot N$ given by
\[
\mathsf H_{u,v}(a\otimes n)
  :=(-1)^{\ov{v}\ov{a}}\,
    u\wedge a\otimes vn .
\]
A direct calculation with the differential $d_I$ gives
\begin{equation}\label{uv-hom-eq}
d_I\mathsf H_{u,v}+\mathsf H_{u,v}d_I=-\id\ot uv,
\end{equation}
where the right-hand side is viewed as an operator on $N$ (coming from the action of $S(I)\ot S(I')$).

Now we will show that for a different choice of an isotropic complement to $\fc(I)$ this isomorphism $\a_{I'}$ will change to a homotopic one.
Indeed, let $\wt{I'}\sub \HH$ be such a complement. Then $\wt{I'}$ is the graph of a morphism $\phi:I'\to \fc(I)$.
Note that we have a canonical decomposition $\fc(I)=\ov{\HH}\oplus I$, where $\ov{\HH}$ is lifted as to $\fc(I)$ as $\fc(I')\cap \fc(I)=\fc(I+I')$.
Thus, we can write $\phi=(\phi_1,\phi_2)$, for $\phi_1:I'\to \ov{\HH}$, $\phi_2:I'\to I$. The condition that $\wt{I'}$ is isotropic means
that $\phi_1$ has isotropic image and $\phi_2$ is self-dual with respect to the natural duality between $I$ and $I'$.

It is enough to consider separately two cases: $\phi=(\phi_1,0)$ and $\phi=(0,\phi_2)$. Indeed, let $I''\sub I+I'$ denote the graph of $\phi_2$. Then $I+I''=I+I'$,
so the decomposition $\fc(I)=\ov{\HH}\oplus I$ associated with $I''$ is the same as the one associated with $I'$. Furthermore, $\wt{I'}$ is the graph of the map
$\phi'_1:I''\to \ov{\HH}$ given as the composition of the projection $I'\to I$ with $\phi_1$.

\noindent
{\bf Case $\phi=(\phi_1,0)$}. 
Let $\a\in I\ot I'$ denote the Casimir element corresponding to the duality between $I$ and $I'$. 
If $(x_i)$ and $(y_i)$ be dual homogeneous bases of $I$ and $I'$ (defined locally) then 
$$\a=\sum_i(-1)^{\ov{x}_i}x_i\ot y_i.$$
Then $A$ denote the even locally nilpotent endomorphism of $N$ given by the action of $\a$ via the above $S(I)\ot S(I')$-module structure.
Then we have
$$\a_{\wt{I'}}-\a_{I'}=\a_{I'}\circ (\id\ot (\exp(A)-1)):{\bigwedge}^\bullet(I)\ot N\to K(I,M(I)).$$
 
Thus, it is enough to construct a null-homotopy for $\id\ot (\exp(A)-1)$ on ${\bigwedge}^\bullet(I)\ot N$.
By \eqref{uv-hom-eq}, we have
\[
d_I\mathsf H_\alpha+\mathsf H_\alpha d_I=-\id\ot A,
\]
where $\mathsf H_\a=\sum_i (-1)^{\ov{x}_i}\mathsf H_{x_i,y_i}$.
Hence
\[
h=-\mathsf H_\alpha\frac{\exp(A)-1}{A}
\]
is a null-homotopy for $\id\ot (\exp(A)-1)$.


\noindent
{\bf Case $\phi=(0,\phi_2)$}. Since $\phi_2$ is self-dual, it corresponds to an element $c\in S^2(I)\sub S(I)$, namely,
$$c=\frac{1}{2}\sum_i x_i\phi_2(y_i).$$ 
Then
$$\a_{\wt{I'}}-\a_{I'}=\a_{I'}\circ (\id\ot (\exp(C)-1)),$$
where $C$ denotes the endomorphism of $N$ given by the action of $c$.
Now we set
\[
\mathsf H_c=\frac12\sum_i \mathsf H_{x_i,\phi_2(y_i)}, 
\]
so that by \eqref{uv-hom-eq},
\[
d_I\mathsf H_c+\mathsf H_c d_I=-\id\ot C.
\]
Therefore
\[
h=-\mathsf H_c\frac{\exp(C)-1}{C}
\]
is a null-homotopy for $\id\ot(\exp(C)-1)$.

\noindent
(ii) This follows from (i) together with the fact that the complex ${\bigwedge}^\bullet(I)\ot S(I^\vee)$ has cohomology only in degree $-n$ and this
cohomology is isomorphic to $\Ber(I)$ (see \cite[ch.\ 3, Prop.\ 4.8]{Manin}). Note that the choice of an isomorphism in (i) is canonical up to homotopy,
hence, the induced isomorphism on cohomology is independent of choices.
\end{proof}

\medskip

\begin{proof}[Proof of Theorem \ref{Ber-isotropic-thm}]
Using Proposition \ref{cat-eq-prop} and isomorphism \eqref{M-I-main-isom}, we get 
$$M\simeq U(\HH)\ot_{U(\fc(I))} M^I\simeq M(I)\ot^{\bbL}_{U(\ov{\HH})} M^I.$$  
This easily leads to
$$C_I(M)\simeq C_I(M(I))\ot^{\bbL}_{U(\ov{\HH})} M^I.$$
It remains to use Lemma \ref{C-I-M-I-lem}.

The isomorphism \eqref{der-coinv-red-eq} is now derived using the isomorphism $C_J(M)\simeq C_{\ov{J}}(C_I(M))$ (see Proposition \ref{trans-coinv-prop}).
\end{proof}

\subsection{Isotropic reduction II}\label{redII-sec}

Let $I\sub \HH$ be an isotropic subbundle. 
In Sec.\ \ref{redI-sec} we considered isotropic subbundles containing $I$. Another case of reduction 
deals with isotropic subbundles $J\sub \HH$ such that $\fc(I)+J=\HH$
(or equivalently, $\pi(I)^\perp+\pi(J)=\VV$).

\begin{lemma}\label{module-redII-lem}
(i) Let $J\sub \HH$ be an isotropic subbundle such that $\fc(I)+J=\HH$.
Then locally there is a splitting $\HH=\OO_S\oplus \VV$, such that $I,J\sub \VV$, and 
there exists a decomposition $\VV=I\oplus I'\oplus \VV'$, where $I'$ is isotropic,
$\VV'$ is orthogonal to $I\oplus I'$, such that $J=I'\oplus J'$, with $J'\sub \VV'$ an isotropic subbundle in $\VV'$. 

\noindent
(ii) For $J$ as in (i), let $\ov{J}$ denote the image of the embedding $J\cap \fc(I)\to \fc(I)/I=\ov{\HH}$.
Then $\ov{J}$ is an isotropic subbundle in $\ov{\HH}$, and we have natural isomorphisms
$$\fc_{\HH}(I)\cap \fc_{\HH}(J)\rTo{\sim} \fc_{\HH}(J)/J, \ \  \fc_{\HH}(I)\cap \fc_{\HH}(J)\rTo{\sim}\fc_{\ov{\HH}}(\ov{J})/\ov{J}.$$
\end{lemma}

\begin{proof}
(i) Since $J\to \HH/\fc(I)$ is a surjection, locally we can choose a subbundle $I'\sub J$ such that $\HH=I'\oplus \fc(I)$.
This gives an orthogonal decomposition
$$\HH=(I\oplus I')\oplus \ov{\HH},$$
where $\ov{\HH}=(\fc(I)\cap\fc(I'))$ is a Heisenberg extension of $\ov{\VV}=\pi(I)^\perp\cap \pi(I')^\perp$, 
and $I$ and $I'$ are dual to each other via the commutator pairing. 
Then $J\sub \fc(J)=I'\oplus \ov{\HH}$, so with respect to the above decomposition,
$$J=0\oplus I'\oplus J',$$
for some isotropic subbundle $J'\sub \ov{\HH}$.
Furthermore, $\fc(I)=I\oplus 0\oplus \ov{\HH}$, $\fc(I)/I\simeq \ov{\HH}$, and
$$J\cap \fc(I)=0\oplus 0\oplus J',$$
so $\ov{J}$ gets identified with $J'\sub \ov{\HH}$.
It remains to choose a splitting $\ov{\HH}=\OO_S\oplus \ov{\VV}$, so that $J'\sub \ov{\VV}$.
Then $I\oplus I'\oplus \ov{\VV}\sub\HH$ will be the required lifting of $\VV$.

\noindent
(ii) Note that $I\cap J=0$ (since it is contained in $\fc((\fc(I)+J))\cap J=\OO\cap J$). Hence, $\ov{J}$ fits into an exact sequence
$$0\to \ov{J}\to \ov{\HH}\to \HH/J\to 0,$$
which shows that $\ov{J}$ is a subbundle. It is isotropic since $J$ is isotropic.
The last assertion can be checked locally using (i). 
\end{proof}

\begin{prop}\label{redII-prop} 
Let $I,J\sub \HH$ be isotropic subbundles such that $\fc(I)+J=\HH$.
Let $\ov{J}=J\cap \fc(I)$ be the corresponding isotropic subbundle in $\ov{\HH}_I=\fc(I)/I$.

\noindent
(i) One has natural isomorphisms 
\begin{equation}\label{isotr-red-main-isom-1}
C_J(M(I))\simeq M^r(J)\otimes^{\bbL}_{U(\fc(I))} U(\ov{\HH}_I)\simeq M^r_{\ov{\HH}_I}(\ov{J}),
\end{equation}
\begin{equation}\label{isotr-red-main-isom-2}
C_I(M(J))\simeq U(\ov{\HH}_I)\ot^{\bbL}_{U(\fc(I))} M(J)\simeq M_{\ov{\HH}_I}(\ov{J}).
\end{equation}
compatible with the vacuum vectors on both sides.

\noindent
(ii) If, in addition, $\HH=J\oplus \fc(I)$ then $C_J(M)\simeq M^I$ for $M\in U(\HH)-\mod^I$, while
$$C_I(M(J))\simeq U(\ov{\HH}_I),$$
compatibly with the vacuum vectors.
In particular, if $L,L'\sub \HH$ are Lagrangians such that $\VV=L\oplus L'$ then there is a canonical identification
$$C_L(M(L'))\simeq \OO_S.$$
\end{prop}

\begin{proof}
(i) The two lines \eqref{isotr-red-main-isom-1} and \eqref{isotr-red-main-isom-2} are proved similarly, so we will only
consider the first.
The first isomorphism in \eqref{isotr-red-main-isom-1} follows from Lemma \ref{module-red-lem}(i). 
We have a natural map of right $U(\ov{\HH}_I)$-modules
$$U(\ov{\HH}_I)\to M^r(J)\ot_{U(\fc(I))} U(\ov{\HH}_I): x\mapsto 1\ot x.$$
It is clear that it factors through $U(\ov{\HH}_I)/(J\cap \fc(I))\cdot U(\ov{\HH}_I)=M^r_{\ov{H}_I}(\ov{J})$.
To show that we get an isomorphism (and there are no higher $\Tor$'s) we can argue locally.

Let us choose a local 
decomposition $\HH=I\oplus I'\oplus \HH'$,
such that $J=I'\oplus J'$, where $J'\sub \HH'$ (see Lemma \ref{module-redII-lem}(i)).
Then we have an identification $\HH'\rTo{\sim} \ov{\HH}_I$ sending $J'$ to $\ov{J}$, and we have
$$M^r(J)\simeq M^r_{\HH'}(J')\ot_{\OO} S(I)$$
as a right module over $U(\fc(I))=U(\HH')\ot_{\OO} S(I)$. Thus, we get
\begin{align*}
&M^r(J)\ot_{U(\fc(I))}^{\bbL} U(\HH')\simeq (M^r_{\HH'}(J')\ot_{\OO} S(I))\ot_{U(\HH')\ot_{\OO} S(I)}^{\bbL} U(\HH')\\
&\simeq M^r_{\HH'}(J')\ot_\OO (S(I)\ot_{S(I)}^{\bbL} \OO)\simeq M^r_{\HH'}(J'),
\end{align*}
as required.

\noindent
(ii) This follows immediately from (i).
\end{proof}


\begin{proof}[Proof of Theorem \ref{Ber-coinv-thm}]
Let us apply the isomorphism \eqref{der-coinv-red-eq} to $M=M(J')$ and $I:=J\cap J'$.
By Lemma \ref{module-red-lem}(ii), we have $M(J')^I\simeq M_{\ov{\HH}_I}(J'/I)$. Hence, we get
$$C_J(M(J'))\simeq C_{J/I}(M_{\ov{\HH}_I}(J'/I))\ot_{\OO_S} \Ber(J\cap J')[n].$$
We have $(J/I)^\perp=J^\perp/I$ in $I^\perp/I$, so $(J/I)^\perp+J'/I=I^\perp/I$, by assumption.
Now we conclude by applying Proposition \ref{redII-prop} to the isotropic subbundles
$\ov{J}=J/I$ and $\ov{J}'=J'/I$ in $\fc(I)/I$. Note that the reduction $\fc(\ov{J})/\ov{J}$ is naturally identified with $\fc(J)/J$,
so that $\ov{J'}\cap \fc(\ov{J})$ is identified with $J'\cap \fc(J)/J'\cap J$.
\end{proof}

\begin{proof}[Proof of Theorem \ref{coinv-red-thm}]
The fact that $\ov{J}$ is an isotropic subbundle was proved in Lemma \ref{module-redII-lem}(ii).
Applying Lemma \ref{module-red-lem}(ii) and Corollary \ref{cat-eq-cor}, we get
\begin{align*}
&C_J(M)=M^r(J)\otimes^{\bbL}_{U(\HH)}M\simeq
M^r(J)\otimes^{\bbL}_{U(\HH)} (U(\HH)\ot^{\bbL}_{U(\fc(I))}M^I)\\
&\simeq M^r(J)\ot^{\bbL}_{U(\fc(I))}M^I.
\end{align*}
It is easy to see that this isomorphism is given by the natural map of complexes
$${\bigwedge}^\bullet(J)\ot U(\HH)\ot_{U(\fc(I))} M^I\to {\bigwedge}^\bullet(J)\ot M.$$

Next, using \eqref{isotr-red-main-isom-1}, we can rewrite this (with $\ov{\HH}=\ov{\HH}_I$) as
\begin{align*}
&M^r(J)\ot^{\bbL}_{U(\fc(I))}M^I\simeq (M^r(J)\ot^{\bbL}_{U(\fc(I))} U(\ov{\HH}))\ot^{\bbL}_{U(\ov{\HH})} M^I\\
&\simeq M^r_{\ov{\HH}}(\ov{J})\ot^{\bbL}_{U(\ov{\HH})} M^I=C_{\ov{J}}(M^I).
\end{align*}
This isomorphism is induced by the natural map of complexes
$${\bigwedge}^\bullet(J\cap \fc(I))\ot M^I\to {\bigwedge}^\bullet(J)\ot U(\HH)\ot_{U(\fc(I))} M^I$$
(which inserts $1\in U(\HH)$).
\end{proof}

\begin{cor}\label{IJJ'-cor}
Let $I\sub J'\sub \HH$ be isotropic subbundles, and let $J\sub \HH$ be an isotropic subbundle such that 
$\fc(I)+J=\HH$. Then for $\ov{J}=J\cap \fc(I)$ and $\ov{J}':=J'/I$ we have an isomorphism
$$C_{\ov{J}}(M_{\ov{\HH}_I}(\ov{J'}))\rTo{\sim} C_J(M(J')),$$
induced by the chain map of complexes
$${\bigwedge}^\bullet(J\cap I^\perp)\ot M_{\ov{\HH}_I}(\ov{J'})\to {\bigwedge}^\bullet(J)\ot M(J)$$
(where the natural embedding $M_{\ov{\HH}_I}(\ov{J'})\to M(J)$ sends $1$ to $1$).
\end{cor}

\subsection{The case of odd rank and the Pfaffian}\label{odd-rk-Pf-sec}

The following statement is a generalization of the standard fact about Clifford algebras.

\begin{prop}\label{Clifford-prop} 
Let $\VV$ be a symplectic vector super bundle of rank $(0|2m)$, and let $0\to \OO\to \HH\to \VV\to 0$ be an extension
(so $\HH$ has a Heisenberg super Lie algebra structure). 
Then for any $U(\HH)$-module $M$ and any Lagrangian $L\sub \HH$, the natural map
$$M(L)\ot_{\OO} M^L\to M$$
is an isomorphism.
\end{prop}

\begin{proof} Let us first check surjectivity. It is enough to check surjectivity locally for $M=U(\HH)$.
Locally we can assume that $\HH=\OO\oplus \VV$ and $L\sub \VV$. Furthermore, we can assume that
$\VV=L\oplus L^*$, with the standard (super) symplectic structure. Then 
$$\WW(\VV)^L=S^m(L)\ot S(L^*)\sub \WW(\VV)$$
(these are really exterior powers since $L$ has odd rank).
We need to check that this subbundle generates $\WW(\VV)$ as a left module. 
But this follows from the fact that the left action map
$$S(L^*)\ot S^m(L)\to S(L)$$
is an isomorphism.

To check injectivity consider the exact sequence of $U(\HH)$-modules
$$0\to K\to M(L)\ot_{\OO} M^L\to M\to 0$$
Taking $L$-invariants, we get that $K^L=0$. Now surjectivity of $M(L)\ot_{\OO} K^L\to K$ implies that $K=0$.
\end{proof}

\begin{cor}\label{Clifford-cor} 
(i) Let $\VV$ be a symplectic vector super bundle of rank $(0|2m)$, $\HH\to \VV$ a Heisenberg extension.
Then for any $U(\HH)$-module $M$ and any Lagrangian $L\sub \HH$, one has
$$C_L(M)\simeq M_L\simeq M^L\ot_{\OO}\Ber(L).$$

\noindent
(ii) Let $\HH\to \VV$ be as in (i), and let $\VV=L_1\oplus L_2$ be a Lagrangian splitting, where $L_1$ and $L_2$
are equipped with liftings to $\HH$.
Then for any $U(\HH)$-module $M$, we have a natural isomorphism
$$M^{L_1}\simeq C_{L_2}(M),$$
induced by the embedding $M^{L_1}\to M$ (in particular, there are no higher derived coinvariants).
\end{cor} 

\begin{proof}
(i) This follows from Proposition \ref{Clifford-prop} together with Theorem \ref{Ber-coinv-thm}.

\noindent
(ii) Since $L_1$ and $L_2$ are complementary in $\VV$, by Proposition \ref{redII-prop}(ii), we have $C_{L_2}(M(L_1))\simeq \OO$. Hence, using (i) we get 
$$M^{L_1}\simeq C_{L_2}(M(L_1)\ot_{\OO} M^{L_1})\simeq C_{L_2}(M).$$
\end{proof}

\begin{example} Assume we are over the point, $V$ is a super symplectic space of odd rank,
and let $L_1$ and $L_2$ be Lagrangians. If $L_1\cap L_2\neq 0$ then the vacuum vector $1\in M(L_1)$ becomes zero
in $C_{L_2}M(L_1)$. Indeed, we can assume that $V=I_1\oplus I_1^*\oplus I_2\oplus I_2^*$, where $L_1=I_1\oplus I_2$, and
$L_2=I_1\oplus I_2^*$. Then the question reduces to the case $L_2=L_1\neq 0$. Then the vacuum vector is in the image of $L_1\ot L_1^*\cdot 1\sub L_1\ot M(L_1)$. 
\end{example}

Corollary \ref{Clifford-cor}(ii) suggests the following definition, which is a generalization of the Pfaffian.

\begin{definition}\label{Pf-def} 
Let $L$ be a vector bundle of rank $(0|m)$ over a superscheme $S$, 
$\phi:L\to L^\vee$ a super symmetric $\OO_S$-linear map, $\la:L\to \OO$ an $\OO$-linear map.
Let us consider the Fock module $M(L)\simeq S(L^\vee)$ over $\WW(L\oplus L^\vee)$ (where the symplectic form on $L\oplus L^\vee$ is such that $(l,l^*)=\lan l,l^*\ran$), 
and the Lagrangian
$$L(\phi,\la):=\{(l,-\phi(l),-\la(l)) \ |\ l\in L\}\sub \WW(L\oplus L^\vee)_{\le 1}.$$
Then by Corollary \ref{Clifford-cor}(ii), we have a natural isomorphism
$$S^m(L^\vee)\simeq M(L)^{L^\vee}\simeq C_{L(\phi,\la)}M(L).$$
We define $\Pf(\phi,\la)\in S^m(L^\vee)$ as the element corresponding to the image of the vacuum vector of $M(L)$ under this isomorphism.
\end{definition}

Note that we can apply this definition also to locally free modules over sheaves of (super) commutative $\OO_S$-algebras.

For a quasicoherent sheaf $\FF$ over a superscheme $S$, equipped with an increasing filtration $\FF_{\le i}$, we set
\begin{equation}\label{N-Rees-eq}
R(\FF_{\le \bullet},\NN)_{\le m}:=\FF_{\le 0}+\NN\FF_{\le 1}+\NN^2\FF_{\le 2}+\ldots+\NN^m\FF_{\le m}\sub \FF,
\end{equation}
where $\NN\sub \OO_S$ is the ideal generated by all odd functions. We also set
$$R(\FF_{\le \bullet},\NN)=\cup_m R(\FF_{\le \bullet},\NN)_{\le m}\sub \FF.$$
Note that if $\FF$ has an $\OO$-algebra structure,
and $\FF_{\le \bullet}$ is compatible with it, then $R(\FF_{\le\bullet})$ is an $\OO$-subalgebra of $\FF$.

\begin{prop}\label{gen-Pf-prop}
(i) $\Pf(\phi,0)=\Pf(\phi)$, where $\Pf(\phi)$ is the usual Pfaffian;

\noindent
(ii) $\Pf(0,\la)=S^m(\la)\in \Hom_{\OO}(S^m(L),\OO)\simeq S^m(L^\vee)$.

\noindent
(iii) Let $\AA$ be a commutative $\OO_S$-algebra equipped with an algebra filtration $(\AA_{\le \bullet})$,
$L$ a vector bundle over $S$ of rank $(0|m)$, $\phi:L\to \AA_{\le 0}\ot_\OO L^\vee$ a (super-)symmetric map, $\la:L\to \NN\AA_{\le 1}$ an $\OO$-linear map.
Then, viewing $\phi$ and $\la$ as $\AA$-linear structures on $\AA\ot_{\OO} L$, we get
$$\Pf(\phi,\la)\in R(\AA_{\le \bullet},\NN)_{\le m}\ot S^m(L^\vee)\sub \AA\ot S^m(L^\vee)$$
and $\Pf(\phi,\la)\equiv \Pf(\phi) \mod \NN\AA\ot S^m(L^\vee)$.

\noindent
(iv) Assume $L$ has an odd basis $e_1,\ldots,e_m$, and let $(e^*_i)$ be the dual basis (with $\lan e_i,e_j^*\ran=\de_{ij}$), 
so we can write $\phi(e_j)=\sum_{ij}\phi_{ij}e_i^*$ for a skew-symmetric matrix of even functions $(\phi_{ij})$, and let 
$\la(e_i)=\la_i$, for odd functions $\la_i$. Then one has
\begin{equation}\label{Pf-gen-for}
\Pf(\phi,\la)=\sum_{I\sub [1,m]}(-1)^{w(I)+(m-|I|)/2}\la_{I}\Pf(\phi^I)\cdot e_m^*\ldots e_2^*e_1^*,
\end{equation}
where $\phi^I$ is the submatrix of $\phi$ with the rows/columns in $I$ deleted, and for $I=\{i_1<\ldots<i_k\}$, we set
$w(I)=\sum_{j=1}^k (i_j-j)$, $\la_I=\la_{i_1}\ldots \la_{i_k}$.
\end{prop}

\begin{proof}
Parts (i) and (ii) follow from (iv). Note that in (i) the sign that appears for $m$ is even is $(-1)^{{m\choose 2}+m/2}=1$.
Part (iii) follows from the functoriality of the construction with respect to homomorphisms between $\OO$-algebras and from (i), since
our data $(\phi,\la)$ are defined as $R(\AA_{\le \bullet})$-linear objects, and since under the reduction homomorphism
$R(\AA_{\le \bullet},\NN)\to \AA/\NN\AA$, $\la$ maps to $0$.
The proof of (iv) is given in appendix \ref{Pf-app}.
\end{proof}

\begin{example}
For $m=2$, we have 
$$\Pf\bigl(\left[\begin{matrix} 0 & a \\ -a & 0\end{matrix}\right],[\la_1,\la_2]\bigr)=(-a+\la_1\la_2)\cdot e_2^*e_1^*.$$
For $m=3$, we have
$$\Pf\bigl(\left[\begin{matrix} 0 & a_{12} & a_{13} \\ -a_{12} & 0 & a_{23}\\ -a_{13} & -a_{23} & 0\end{matrix}\right],[\la_1,\la_2,\la_3]\bigr)=
(-\la_1a_{23}+\la_2a_{13}+\la_3a_{12}-\la_1\la_2\la_3)\cdot e_3^*e_2^*e_1^*.$$
\end{example}

\subsection{Purely odd intersection}\label{odd-intersection-sec}

In the case when $\VV$ is purely odd, and $S$ is purely even, the usual sheaf of coinvariants of $L$ in $M(L')$
(for arbitrary Lagrangian subbundles $L$ and $L'$) is a line bundle, dual to the Pfaffian line bundles considered in \cite[Sec.\ 4]{BD}.
Part (i) of the following theorem is a generalization of this result. 

Let $\HH\rTo{\pi} \VV$ be a Heisenberg extension, where $\VV$ is a symplectic super vector bundle over a superscheme $S$.

\begin{theorem}\label{Pf-thm}
(i) Let $J\sub \VV$ be an isotropic subbundle, $L\sub \VV$ a Lagrangian subbundle, equipped with lifts to $\HH$. 
Assume that for all $s\in S$, one has 
$J|_s^+\cap L|_s^+=0$ in $\VV|_s^+$. Set $\ov{\HH}_J:=\fc(J)/J$.
Then $C_J(M(L))$ is concentrated in degree $0$ and locally near any $s_0\in S$ there exists an isomorphism
of $U(\ov{\HH}_J)$-modules
\begin{equation}\label{local-Fock-isom}
C_J(M(L))\simeq M_{\ov{\HH}_J}(\ov{L})
\end{equation}
for some Lagrangian subbundle $\ov{L}\sub \ov{\HH}_J$,
such that $\ov{L}|_s^+$ is the image of 
$L|_s^+\cap \fc(J)|_s^+$ in $\fc(J)/J|_s$.

\noindent
(ii) Assume in addition that $\fc(J)/J$ is purely even, and
let $v(J,L)\in C_J(M(L))$ denote the projection of the vacuum vector $1\in M(L)$.
Suppose also that for some $s_0\in S$, one has
$$\dim \pi(J)_{s_0}^-\cap \pi(L)_{s_0}^-=m.$$
Then near $s_0$, there exists an isomorphism $C_J(M(L))\simeq \ov{M}:=M_{\ov{\HH}_J}(\ov{L})$ as in (i), such that
$$v(J,L)=u\cdot \ov{v},$$
for an invertible section $u\in R(U(\ov{\HH}_J)_{\le \bullet},\NN)$ and for some 
$$\ov{v}\in R(\ov{M}_{\le \bullet},\NN)_{\le m},$$
vanishing at $s_0$.
\end{theorem}

\begin{proof}
(i) The statement is local, so we can choose an even subbundle $I\sub J$, such that $I|_s=J|_s^+$.
Then $\fc(I)+L=\HH$, so by Proposition \ref{redII-prop}(i), we have an isomorphism
$$C_I(M(L))\simeq M_{\fc(I)/I}(L\cap \fc(I)).$$
Hence, by the transitivity of the coinvariants (see Proposition \ref{trans-coinv-prop})
$$C_J(M(L))\simeq C_{J/I}(M_{\fc(I)/I}(L\cap \fc(I))),$$
compatibly with the images of the vacuum vectors.
Thus, replacing $\HH$ with $\fc(I)/I$, $J$ with $J/I$ and $L$ with $L\cap \fc(I)$, we reduce to the
case when $J$ has odd rank.

Next, let us choose a Lagrangian $L'\sub \VV$ such that $L'$ is transversal to both $L$ and $J$.
Then the projection $J\to L$ along $L'$ is an embedding of a subbundle, and we denote by $J_1\sub L$ its image.
Let us choose a complement to $J_1$ in $L$, so that we have a decomposition
$L=J_1\oplus J_2$. Then  we get an identification
$$\VV\simeq J_1\oplus J_1^\vee\oplus J_2\oplus J_2^\vee,$$
where the right-hand side has a standard symplectic structure. 
Furthermore, we can choose a splitting $\HH=\VV\oplus \OO$, so that the lift of $L$ to $\HH$ lies in $\VV$.
With respect to the decomposition of $\HH$,
$J\sub V$ is the graph of a map $J_1\rTo{(\phi,\psi,\rho)} J_1^\vee\oplus J_2^\vee\oplus \OO$, where $\phi$ is (super-)symmetric.

Let us consider the embedding
$$i_\psi: J_1\oplus J_1^\vee=:\VV_1\to V=J_1\oplus J_1^\vee\oplus J_2\oplus J_2^\vee:(x,y)\mapsto (x,y,0,\psi(x)),$$
and the induced embedding $\HH_1=\VV_1\oplus \OO\to \VV\oplus \OO=\HH$.
Then $i_\psi$ is compatible with the (standard) symplectic forms, and $J=i_\psi(J')$,
where $J'\sub \HH_1=J_1\oplus J_1^\vee\oplus \OO$ is the graph of $(\phi,\rho)$.
Note that the complex $K_{\HH}(J,M(L))$ representing $C_J(M(L))$ depends only on the action of $J$ on $M(L)$, or equivalently,
on the action of $J'\sub \HH_1$. Hence, it is identified with the complex $K_{\HH_1}(J',M(L))$.
Now $J_1^\vee\sub V_1$ is a Lagrangian complement to $J'$ in $\VV_1$. Since $\VV_1$ is purely odd, by Corollary \ref{Clifford-cor},
the cohomology of $K_{\HH_1}(J',M(L))$ is concentrated in degree $0$, and the embedding
$M(L)^{J_1^\vee}\to M(L)$ induces an isomorphism
\begin{equation}\label{Clifford-ML-J-isom}
M(L)^{J_1^\vee}\rTo{\sim} M(L)_J.
\end{equation}

Furthermore, the embedding
$$i'_\psi:J_2\oplus J_2^\vee\to \VV: (z,t)\mapsto (0,-\psi^*(z),z,t)$$
induces an isomorphism of symplectic bundles $J_2\oplus J_2^\vee\to J^\perp/J$.
In addition, the image of $i'_{\psi}$ is orthogonal to $J_1^\vee$, hence
the isomorphism \eqref{Clifford-ML-J-isom} is compatible with
the $\WW(J_2\oplus J_2^\vee)$-module structures, where the action on the left is via $i'_\psi$, or equivalently
via the standard embedding $J_2\oplus J_2^\vee\to \VV$ (since $J_1^\vee$ acts trivially).

Since $L=J_1\oplus J_2$, we have a natural isomorphism
$$S(J_2^\vee)\ot S^m(J_1^\vee)\rTo{\sim}M_{\VV}(L)^{J_1^\vee},$$
where the rank of $J_1$ is $(0|m)$,
induced by the action on the vacuum vector of $M_{\VV}(L)$.
Now \eqref{local-Fock-isom} follows from the fact that $S(J_2^\vee)\simeq M_{J_2\oplus J_2^\vee}(J_2)$
as a module over $\WW(J_2\oplus J_2^\vee)$.

\noindent
(ii) As in (i) we reduce to the case when 
$$\VV\simeq J_1\oplus J_1^\vee\oplus J_2\oplus J_2^\vee, \ \ \HH=\VV\oplus\OO,$$
$L=J_1\oplus J_2$, $J_1$ has rank $(0|N)$, and
$J\sub \HH$ is the graph of a map $J_1\rTo{(-\phi,-\psi,-\rho)} J_1^\vee\oplus J_2^\vee\oplus \OO$, where $\phi$ is (super-)symmetric.
We have an identification $M(L)\simeq S(J_2^\vee)\ot S(J_1^\vee)$, and the composition
$$S(J_2^\vee)\ot S^m(J_1^\vee)\to M(L)\to M(L)_J$$
is an isomorphism.
We have to rewrite $v(J,L)\in M(L)_J$ as the image of a vector in $S(J_2^\vee)\ot S^m(J_1^\vee)$.

Consider the $\OO$-algebra $\AA=S(J_2^\vee)$ with the standard increasing filtration $(\AA_{\le \bullet})$. 
Then we can view $\phi$ and $\la=(\psi,\rho)$ as a (super-)symmetric $\OO$-linear map $\phi:J_1\to \AA_{\le 0}\ot J_1^\vee$,
and as an $\OO$-linear map $\la:J_1\to \AA_{\le 1}$. Note that since $J_1$ is purely odd and $\AA$ is purely even, $\la$ takes values in $\NN\AA_{\le 1}$.
Now we observe that
$$v(J,L)=\Pf(\phi,\la).$$

By assumption $\dim\ker(\phi|_{s_0})=(0|m)$. Let us choose a complement $W|_{s_0}\sub J_1|_{s_0}$ to $\ker(\phi|_{s_0})$ and extend it
to a subbundle $W\sub J_1$. Then near $s_0$, $\phi$ restricts to a nondegenerate pairing $\phi_W$ on $W$. Hence,
near $s_0$ we have a decomposition
$$J_1=W^\perp\oplus W,$$
where $W^\perp\sub J_1$ is the orthogonal of $W$ with respect to $\phi$. Then we have
$$v(J,L)=\Pf(\phi,\la)=\Pf(\phi_W,\la|_{W})\cdot \Pf(\phi_{W^\perp},\la|_{W^{\perp}}).$$
Now by Proposition \ref{gen-Pf-prop}, both factors are in $R(\AA_{\le \bullet},\NN)$, and
$$u:=\Pf(\phi_W,\la|_W)\equiv \Pf(\phi_W) \mod \NN$$ 
so $u$ is invertible. Furthermore, since $\phi_{W^\perp}$ vanishes at $s_0$, so does $\Pf(\phi_{W^\perp},\la|_{W^{\perp}})$.
Since $\AA$ is a subalgebra of $U(\ov{\HH}_J)$, the assertion follows.
\end{proof}

\subsection{Odd intersection: Lagrangian case}\label{odd-int-sec}

Let us now consider the case of two Lagrangians $L_1,L_2$ in $\VV$, equipped with some liftings to a Heisenberg extension $\HH$.
We have the following Corollary from Theorem \ref{Pf-thm}.

\begin{cor}\label{Lagr-int-cor}
If $L|_s^+\cap L'|_s^+=0$ in $\VV|_s$ for all $s\in S$
then $C_{L_1}(M_\HH(L_2))$ is a line bundle, and $v(L_1,L_2)$ (the image of the vacuum vector in $M(L_2)$)
is its section vanishing exactly at $s\in S$ such that $L_1|_s\cap L_2|_s\neq 0$.
\end{cor}

Following \cite{FKP-reg}, for a pair of Lagrangians $L_1,L_2$, which are generically transversal, we define a rational section
$\th(L_1,L_2)$ of $\Ber(L_1)^{-1}\ot \Ber(L_2)^{-1}$, as the Berezinian of the morphism $L_1\to L_2^\vee$ given by the pairing.
The following result gives a global version of \cite[Thm.\ 4.5]{FKP-reg}.

\begin{theorem}\label{theta-Pf-thm}
Assume that $S$ is smooth and irreducible and the Heisenberg extension $\HH$ is split: $\HH=\OO\oplus \VV$.
Let $L_1,L_2\sub\VV$ be Lagrangians such that
$L_1|_s^+\cap L_2|_s^+=0$ for all $s\in S$, and $L_1$ and $L_2$ are generically transversal.
Then there is a natural isomorphism of line bundles
\begin{equation}\label{coinv-square-isom}
\Ber(L_1)\ot \Ber(L_2)\simeq C_{L_1}(M(L_2))^2,
\end{equation}
such that 
$$\th(L_1,L_2)^{-1}=c\cdot v(L_1,L_2)^2,$$
where $c$ is a nonzero universal constant (possibly depending on the odd dimension of $\VV$).
\end{theorem}

\begin{proof}
It is enough to construct an isomorphism \eqref{coinv-square-isom} locally since it is uniquely determined by the condition
that it maps $\th(L_1,L_2)^{-1}$ to $v(L_1,L_2)^2$, and so will glue into a global one.

Locally we can choose a subbundle $I\sub L_1$, such that $I|_s=L_1|_s^+$ for any $s\in S$. Then $I^\perp/I$ is purely odd and $I^\perp+L_2=\VV$.
Hence, by Propositions \ref{trans-coinv-prop} and \ref{redII-prop}(i), we have an isomorphism
$$C_{L_1}(M(L_2))\simeq C_{\ov{L}_1}C_I(M(L_2))\simeq C_{\ov{L}_1}(M_{I^\perp/I}(\ov{L}_2),$$
where $\ov{L}_1=L_1/I$, $\ov{L}_2=L_2\cap I^\perp$ viewed as a subbundle of $\ov{\VV}=I^\perp/I$.
Furthermore, this isomorphism sends $v(L_1,L_2)$ to $v(\ov{L}_1,\ov{L}_2)$,

On the other hand, since $I$ is transversal to $L_2$, we have a quasi-isomorphism
$$[L_1\to L_2^\vee]\to [L_1/I\to L_2^\vee/I],$$
where $L_2^\vee/I\simeq (L_2\cap I^\perp)^\vee$, so we get an isomorphism
$$\Ber(L_1)\ot \Ber(L_2)\simeq \Ber(\ov{L}_1)\ot \Ber(\ov{L}_2)$$
sending $\th(L_1,L_2)^{-1}$ to $\th(\ov{L},\ov{L}_2)^{-1}$.

Thus, replacing $(\VV,L_1,L_2)$ with $(\ov{\VV},\ov{L}_1,\ov{L}_2)$ we reduce to the case when $\VV$ is purely odd.
Let us consider $\VV^{(2)}=\VV\oplus \VV$ with the form $\om\oplus \om$. Then we have
$$C_{L_1}(M(L_2))^2\simeq C_{L_1\oplus L_1}(M_{\VV^{(2)}}(L_2\oplus L_2)),$$
compatibly with the vacuum vectors.
On the other hand, we have another Lagrangian
$$\De_i(\VV):=\{(x,ix) \ |\ x\in \VV\}\sub \VV^{(2)},$$ 
and by Proposition \ref{Clifford-prop}, we have
$$M_{\VV^{(2)}}(L_2\oplus L_2)\simeq M_{\VV^{(2)}}(\De_i(\VV))\ot_{\OO} M_{\VV^{(2)}}(L_2\oplus L_2)^{\De_i(\VV)}.$$

Therefore,
$$C_{L_1\oplus L_1}(M_{\VV^{(2)}}(L_2\oplus L_2))\simeq C_{L_1\oplus L_1}(M_{\VV^{(2)}}(\De_i(\VV)))\ot_{\OO} M_{\VV^{(2)}}(L_2\oplus L_2)^{\De_i(\VV)}.$$
By Theorem \ref{Ber-coinv-thm}, we have
$$C_{L_1\oplus L_1}(M_{\VV^{(2)}}(\De_i(\VV)))\simeq \Ber(L_1).$$
On the other hand, by Corollary \ref{Clifford-cor} and Theorem \ref{Ber-coinv-thm},
$$M_{\VV^{(2)}}(L_2\oplus L_2)^{\De_i(\VV)}\simeq \Ber^{-1}(\VV)\ot C_{\De_i(\VV)}(M_{\VV^{(2)}}(L_2\oplus L_2))\simeq \Ber^{-1}(\VV)\ot \Ber(L_2).$$
Hence, using the triviality of $\Ber(\VV)$, we get
$$C_{L_1\oplus L_1}(M_{\VV^{(2)}}(L_2\oplus L_2))\simeq \Ber(L_1)\ot \Ber(L_2).$$

To chase what happens with our canonical sections, we will make use of the Lagrangian $\De_{-i}(\VV)$, complementary to $\De_i(\VV)$.
We observe that we have a natural isomorphism induced by the Weyl algebra action,
$$S^{top}(\De_i(\VV)/\De_i(L_2))\rTo{\sim} M_{\VV^{(2)}}(L_2\oplus L_2)^{\De_i(\VV)},$$
$$S^{top}(\De_{-i}(\VV)/\De_{-i}(L_1))\rTo{\sim}C_{L_1\oplus L_1}(M_{\VV^{(2)}}(\De_i(\VV))),$$
so the above isomorphism corresponds to the induced isomorphism
\begin{align*}
&\Ber(L_1)\ot \Ber(L_2)\simeq S^{top}(\De_{-i}(\VV)/\De_{-i}(L_1))\ot S^{top}(\De_i(\VV)/\De_i(L_2))\\ 
&\rTo{\sim} 
C_{L_1\oplus L_1}(M_{\VV^{(2)}}(\De_i(\VV)))\ot M_{\VV^{(2)}}(L_2\oplus L_2)^{\De_i(\VV)}\to C_{L_1\oplus L_1}(M_{\VV^{(2)}}(L_2\oplus L_2)).
\end{align*}

We have to check that under this isomorphism $\th(L_1,L_2)^{-1}$ goes to the vacuum vector (up to a constant).
It is enough to check this locally on an open subset where $L_1$ and $L_2$ are transversal, so we can assume that $(e_\bullet)$ and $(f_\bullet)$ are dual bases
of $L_1$ and $L_2$. The section $\th(L_1,L_2)^{-1}$ corresponds (up to a sign) to 
$(f_1\ldots f_n)\ot (e_1\ldots e_n)\in S^n(\VV/L_1)\ot S^n(\VV/L_2)$. The above isomorphism sends this to
$$\De_{-i}(f_1)\ldots \De_{-i}(f_n)\cdot \De_i(e_1)\ldots \De_i(e_n)\cdot 1\in C_{L_1\oplus L_1}(M_{\VV^{(2)}}(L_2\oplus L_2)).$$ 
Now we can start moving $\De_{-i}(f_j)$ to the right, starting from $\De_{-i}(f_n)$. This will show that the above expression is $c\cdot 1$, for some universal constant $c$.
\end{proof}


\subsection{Coinvariants for degenerate Heisenberg extensions}

Let $\VV$ be a symplectic super vector bundle over $S$, and suppose
\begin{equation}\label{deg-Heis-ext-eq}
0\to \KK\to \wt{\VV}\to \VV\to 0
\end{equation}
is an extension of super vector bundles. We equip $\wt{\VV}$ with the (super) skew symmetric form, pulled back from $\VV$, so that $\KK$ is the kernel of this form.
Then we have a degenerate Weyl algebra $\WW(\wt{\VV})$ with the central subalgebra $S(\KK)$.

Given an isotropic subbundle $\II\sub \wt{\VV}$ and a $\WW(\wt{\VV})$-module $M$, as before, we can consider the sheaf of coinvariants
$$M_{\II}\simeq M^r(\II)\ot_{\WW(\wt{\VV})} M,$$
where
$$M^r(\II)=\WW(\wt{\VV})/\II\cdot \WW(\wt{\VV}).$$

Now given a homomorphism $\chi:\KK\to \OO_S$, we can consider $\WW(\wt{\VV})$-modules with the central character $\chi$, i.e., such that $\KK$ acts via $\chi$.
For example, given an isotropic subbundle $\JJ\sub \wt{\VV}$, projecting isomorphically to a subbundle in $\VV$ (i.e., transversal to $\KK$),
for every $\chi:\KK\to\OO$, we have the corresponding $\WW(\wt{\VV})$-module with the central character $\chi$,
$$M_{\wt{\VV},\chi}(\JJ):=\WW(\wt{\VV})/\WW(\wt{\VV})\cdot(\JJ+I_{\chi}),$$
where $I_{\chi}$ is the span of $r-\chi(r)\in \KK\oplus \OO$. 

Let $\HH_\chi$ denote the push-out of the extension \eqref{deg-Heis-ext-eq} with respect to $\chi:\KK\to \OO$. We can equip $\HH_\chi$ with
the natural structure of the Heisenberg extension. The projection $\wt{\VV}\ot \HH_{\chi}$ induces a surjection $\WW(\wt{\VV})\to U(\HH_\chi)$. 
A $\WW(\wt{\VV})$-module $M$ has $\chi$ as a central character if and only the $\WW(\wt{\VV})$-action on it factors through $U(\HH_\chi)$.

\begin{lemma}\label{central-char-red-lem}
Let $\II\sub \wt{\VV}$ be an isotropic subbundle transversal to $\KK$, and let $\II_\chi\sub \HH_\chi$ denote its image under the projection $\wt{\VV}\to \HH_\chi$.
Then for a $\WW(\wt{\VV})$-module $M$ with the central character $\chi$, there is a natural isomorphism
$$M_\II\simeq C_{\II_\chi}(M),$$
where on the right we view $M$ as an $\HH_\chi$-module.
\end{lemma}

\begin{example} In the application to the moduli spaces of curves, we will have a collection of symplectic vector spaces $(V_i)_{i=1,\ldots,n}$ and of extensions
$$0\to K\to \wt{V}_i\to V_i\to 0,$$
together with Lagrangian subspaces $L_{i,+}\sub \wt{V}_i$ (transversal to $K$), so that for a collection of linear functionals $\chi_i:K\to \C$, we get 
$\WW(\wt{V}_i)$-modules $M_{\wt{V}_i,\chi_i}(L_{i,+})$.
In addition, we will have an isotropic subspace $\wt{I}\sub \wt{V}:=\bigoplus_i \wt{V}_i$, such that $\wt{I}\cap \bigoplus_i K$
is the diagonal $K$. Then, assuming that $\sum_i \chi_i=0$, we will be looking at the coinvariants
$$(\bigotimes_i M_{\wt{V}_ i,\chi_i}(L_{i,+}))_{\wt{I}}\simeq M_{\wt{V},\chi}(L_+)_{\wt{I}},$$
where $\chi=(\chi_i)$, $L_+=\bigoplus_i L_{i,+}$. 
Let $H_\chi$ be the corresponding Heisenberg extension of $\bigoplus_i V_i$ by $\C$. 
Note that $M_{\wt{V},\chi}(L_+)=M_{H_\chi}(L_+)$, where on the right we view $L_+$ as a subspace of $H_\chi$, via the projection.
Hence, using Lemma \ref{central-char-red-lem},
we can rewrite our space of coinvariants as
$$C_{I_\chi}(M_{H_\chi}(L_+)),$$
where $I_\chi$ is the image of $\wt{I}$ in $H_\chi$.
\end{example}


\section{Lie algebroid actions on Heisenberg coinvariants}\label{Lie-algebroid-sec}

\subsection{Lie algebroid associated with a family of isotropic subspaces and coinvariants}\label{constr-main-sec}

Let $V$ be a finite dimensional symplectic super vector space over $\C$, and let $\WW(V)$ be the corresponding super Weyl algebra.
We have a natural increasing filtration $(\WW_{\le n}(V))$, such that $\WW_{\le n}(V)/\WW_{\le n-1}(V)\simeq S^n(V)$.
The adjoint action of $\WW_{\le 2}(V)/\WW_{\le 1}(V)$ on $\WW_{\le 1}(V)/\WW_{\le 0}(V)\simeq V$ gives an identification
of $S^2(V)$ with the Lie super algebra $\sp(V)$ acting on $V$.


Now assume we have an isotropic subbundle 
$$\II\sub \OO_S\ot \WW_{\le 1}(V)$$
over a superscheme $S$
(recall that this means that $\II$ projects isomorphically onto an isotropic subbundle of $\OO\ot V$),
and let $\ov{\HH}_\II=\fc(\II)/\II$ denote the corresponding reduced Heisenberg Lie algebra
Let us consider the right $\OO_S\ot \WW(V)$-module
$$M^r(\II)=\OO\ot \WW(V)/\II\cdot \WW(V).$$
By \eqref{M-I-main-isom}, we have an isomorphism of sheaves of $\OO$-algebras,
$$\und{\End}_{\OO\ot \WW(V)^{op}}(M^r(\II))\simeq U(\ov{\HH}_\II).$$

Note that in the case $\II\sub \OO_S\ot V$, $M^r(\II)$ has a natural $\Z/2$-grading coming from the $\Z/2$-grading of $\WW(V)$ such that $\deg(V)=1$.

\begin{definition}\label{main-def} 
(i) Let $\AA_{\II}$ denote the subsheaf of the sheaf of first order differential operators $D_{\le 1}(M^r(\II),M^r(\II))$, consisting of operators $D$ such that
\begin{itemize}
\item $D$ commutes with the right $\WW(V)$-action (note that no Koszul signs appear in this condition);
\item the symbol of $D$ is in $\TT\sub \TT\ot \und{\End}_{\OO_S\ot \WW(V)^{op}}(M^r(\II))$;
\item $[D,U_{\le 1}(\ov{\HH}_\II)]\sub U_{\le 1}(\ov{\HH}_\II)$.
\end{itemize}
We also denote by $\wt{\AA}_{\II}$ the similar subsheaf subject only to the first two conditions.

\noindent
(ii) In the case $\II\sub \OO_S\ot V$, let $\AA^{ev}_{\II}$ denote the subsheaf of $\AA_{\II}$ consisting of operators preserving the $\Z/2$-grading of $M^r(\II)$.
\end{definition}

Note that $\wt{\AA}_{\II}$ is exactly the sheaf of infinitesimal symmetries of $M^r(\II)$ as a $\WW(V)^{op}$-module.

\begin{prop}\label{anchor-sur-prop} 
(i) The subsheaves $\AA_{\II}\sub \wt{\AA}_{\II}$ (resp., $\AA^{ev}_{\II}$ in the case $\II\sub \OO_S\ot V$) in the Atiyah Lie algebroid of $M^r(\II)$
are sub- Lie algebroids. 
Their anchor maps fit into the exact sequences
$$0\to U(\ov{\HH}_\II)\to \wt{\AA}_{\II}\to \TT_S\to 0,$$
$$0\to U_{\le 2}(\ov{\HH}_\II)\to \AA_{\II}\to \TT_S\to 0,$$
and in the case $\II\sub\OO_S\ot V$,
$$0\to \WW_{\le 2}^{ev}(\II^\perp/\II)\to \AA^{ev}_{\II}\to \TT_S\to 0.$$
The action of $D\in\AA_{\II}$ on $M^r(\II)$ satisfies 
\begin{equation}\label{D-filtration-property}
D(M^r(\II)_{\le i})\sub M^r(\II)_{\le i+2},
\end{equation}
where the filtration $M^r(\II)_{\le \bullet}$ is the image of the natural filtration on $\WW(V)$.

\noindent
(ii) The adjoint action of $\AA_{\II}/\OO$ on $\ov{\HH}_\II$ induces an isomorphism of $\AA_{\II}/\OO$ with the Atiyah Lie algebroid $\AA_{\ov{\HH}_\II}$
of infinitesimal symmetries of $\ov{\HH}_\II$ with a structure of a Heisenberg Lie algebra. 
The adjoint action of $\AA_{\II}/\OO$ on the kernel of the anchor map of $\AA_\II$ coincides with the natural action of
$\AA_{\ov{\HH}_\II}$ on $U_{\le 2}(\ov{\HH}_\II)$. 
In the case $\II\sub \OO_S\ot V$, $\AA^{ev}_{\II}/\OO$ is identified with the Atiyah Lie algebroid $\AA_{\II^\perp/\II}$ of the principal symplectic bundle associated with $\II^\perp/\II$.
\end{prop}

\begin{proof} (i) 
It is enough to show that locally every vector field $v$ lifts to a section of $\AA^{ev}_{\II}$. Let $I_0=\II|_{s_0}\sub V$ for a point $s_0\in S$.
We can fix a (constant) orthogonal decomposition
$$V=(I_0\oplus I_0^*)\oplus V',$$
where $V'$ is a symplectic supervector space, and $I_0\oplus I_0^*$ is equipped with a standard symplectic form (such that $(x,x^*)=\lan x,x^*\ran$ for $x\in I_0$, $x^*\in I_0^*$). 
Then near $s_0$ the isotropic subbundle 
$\II\sub \OO\ot \WW_{\le 1}(V)$ is
a graph of an even morphism 
$$\OO\ot I_0\rTo{(\phi,\psi,\la)} \OO\ot (I_0^*\oplus V'\oplus \C).$$
So $\phi:I_0\to \OO\ot I_0^*$, $\psi:I_0\to \OO\ot V'$ and $\la:I_0\to \OO$ are matrices of local functions on $S$ satisfying
$$-\phi^*+\phi+\psi^{\dagger}\psi=0,$$
where $\psi^{\dagger}:V'\to \OO\ot I_0^*$ is the dual of $\psi$ with respect to the symplectic form $(\cdot,\cdot)_{V'}$ on $V'$
(so that $(v',\psi(x))_{V'}=\lan \psi^{\dagger}(v'),x\ran$).

The embedding $\WW(V')\ot S(I_0^*)\to \WW(V)$, together with an action on the vacuum vector, give an identification
\begin{equation}\label{Mr-I-ident-eq}
\OO\ot \WW(V')\ot S(I_0^*)\rTo{\sim} M^r(\II).
\end{equation}
For an element $x$ of an algebra we denote by $l_x$ and $r_x$ the operators of left and right multiplications by $x$.
It is easy to check that under the isomorphism \eqref{Mr-I-ident-eq},
$x^*\in I_0^*$ (resp., $y\in V'$) acts on the left-hand side by $r_{x^*}$ (resp., $r_y$),
whereas $x\in I_0$ acts by 
$$\pa^r_x-l_{\phi(x)+\psi(x)+\la(x)}J^{\ov{x}},$$
where $J$ is the grading operator (equal to $1$ on the even component and to $-1$ on the odd component), and $\pa^r_x$ is 
the ``right derivation" given by
$$\pa^r_x(f\ot x_1^*\ldots x_m^*)=f\ot x_1^*\ldots x_{m-1}\cdot\lan x_m^*,x\ran+(-1)^{\ov{x}\ov{x}_m^*}f\ot x_1^*\ldots \lan x_m^*,x\ran x_m^*+\ldots,$$
where $f\in \WW(V')$, $x_i^*\in I_0^*$.

We denote the action of a vector field $v$ on the base $S$ on sections of trivial bundles over $S$ simply as $v$.
We need to construct an operator $D_v$ on $\OO\ot \WW(V')\ot S(I_0^*)$, commuting with the right action of $\WW(V)$ and satisfying
$[D_v,f]=v(f)$ for $f\in \OO_S$. We look for $D_v$ in the form
$$D_v=v+l_{F+G},$$
with $F\in \OO\ot (S^2(I_0^*)\oplus I_0^*)$, $G\in \OO\ot I_0^*\ot V'$ (of the same parity as $v$).
Then commutation with the action of $I_0^*$ and $\WW(V')$ is automatic. 
Using the relations
$$v r_a=r_a v+ r_{v(a)} J^{\ov{v}}, \ \ \pa^r_x l_b=l_b \pa^r_x + l_{\pa^r_x(b)} J^{\ov{b}}$$
we can restate the fact that $D_v$ commutes with $\pa^r_x-r_{\phi(x)}-r_{\psi(x)}-r_{\la(x)}$, for $x\in I_0$, as
$$\pa^r_x(F+G)+v(\phi(x)+\psi(x)+\la(x))+[G,\psi(x)]=0
$$
(where we use supercommutator).
We claim that these equations on $F$ and $G$ have a unique solution. 
Namely,
let $(e_i)$ be a basis of $I_0$, $(e_i^*)$ the dual basis of $I_0^*$ (with $\lan e^*_i,e_i\ran=1$), and let $(f_i)$ be a basis of $V'$.
We can write
$$\phi(e_j)=\sum_j a_{ij}e_i^*, \ \ \psi(e_j)=\sum_j b_{ij}f_i, \ \ \la(e_j)=c_j$$
for some functions $(a_{ij})$, $(b_{ij})$ and $(c_i)$ on $S$.
Then the solution for $(F,G)$ is given by
$$F=\frac{1}{2}\sum_{i,j} (A_{ij}-v(a_{ij}))e_i^*e_j^*-\sum_iv(c_i)e_i^*, \ \ G=-\sum_{i,j} v(b_{ij})f_ie_j^*,
$$
where 
$$A_{ij}=\sum_{k,l}(-1)^{\ov{v}(\ov{e}_j+\ov{f}_k)+\ov{e}_i\ov{f}_k}b_{kj}v(b_{li})(f_l,f_k).
$$

The property \eqref{D-filtration-property} holds since it holds for the local lifts $D_v$ above (and since it obviously holds for $D$ in $U_{\le 2}(\ov{\HH}_\II)$.

\noindent 
(ii) The adjoint action of $D\in \AA_\II/\OO$
on $\ov{\HH}_I$ is compatible with the structure of an $\OO_S$-linear Heisenberg Lie algebra on $\ov{\HH}_I$.
In other words, we can view $\ad(D)$ as a section of the Atiyah Lie algebroid $\AA_{\ov{\HH}_I}$.
The obtained map 
$$\AA_{\II}/\OO\to \AA_{\ov{\HH}_I}$$
is a map of Lie algebroids inducing an isomorphism of the kernels of the anchor maps, since
$U_{\le 2}(\ov{\HH}_I)/\OO$ is exactly the sheaf of $\OO_S$-derivations of
$\ov{\HH}_\II$, trivial on $\OO_S$.
Hence, our map of Lie algebroids is an isomorphism.
\end{proof}

\begin{remark}
Part (i) of Proposition \ref{anchor-sur-prop} also follows from another construction of a Lie algebroid on the isotropic Grassmannian of the Heisenberg 
algebra in Sec.\ \ref{another-constr-sec} below.
\end{remark}

Note that by construction, $\wt{\AA}_{\II}$ acts on the left on $M^r(\II)$, commuting with the right $\WW(V)$-action, so that 
$\WW(\ov{\HH}_\II)$ acts naturally.

In the case of a Lagrangian subbundle $\LL\sub \OO_S\ot \WW_{\le 1}(V)$, we have $\ov{\HH}_\LL=\OO_S$, so
$\wt{\AA}_{\LL}=\AA_{\LL}$ is a Picard-Lie algebroid (the kernel of the anchor map is $\OO$).
It is easy to describe the corresponding tdo algebra.

\begin{lemma}
The tdo $U\AA_{\LL}$ (universal enveloping of $\AA_{\LL}$) is isomorphic to the subalgebra $D_{\WW(V)^{op}}(M^r(\LL))$ of differential operators 
$D:M^r(\LL)\to M^r(\LL)$, commuting with the right action of $\WW(V)$.
\end{lemma}

\begin{proof}
The action of $\AA_{\LL}$ on $M^r(\LL)$ gives a homomorphism 
$$U\AA_{\LL}\to \DD:=D_{\WW(V)^{op}}(M^r(\LL)),$$
compatible with the natural filtrations. 
Consider the induced morphism on the associated graded quotients 
$$S^\bullet(\TT)\to \gr^\bullet(\DD).$$
its composition with an embedding 
$$\gr^\bullet(\DD)\hra S^\bullet(\TT)\ot \und{\End}_{\OO\ot\WW(V)^{op}}(M^r(\LL))\simeq S^\bullet(\TT)$$
is the identity map. Hence, it is an isomorphism.
\end{proof}

\begin{example}
Let $L\sub V$ be a fixed Lagrangian. We have a natural family of Lagrangians $L_\phi\sub \WW_{\le 1}(V)$ parametrized by $\phi\in L^*$
($L_\phi$ is the graph of $\phi$ in $\C\oplus V$). Let $\LL\sub \OO\ot \WW_{\le 1}(V)$ be the corresponding Lagrangian subbundle over $S=L^*$.
Then for every choice of a Lagrangian complement $L'\sub V$, we have a flat connection on $\AA_{\LL}$ constructed as follows.
We can identify $M(\LL)$ with $\OO_{L^*}\ot S(L')$. Then upon the identification $L'\simeq L^*$ induced by the pairing, the connection is given by
$$\nabla_{\partial_{l'}}=\partial_{l'}+l'\cdot(-)$$
for $l'\in L'\simeq L^*$,
where the right-hand side is an operator on $\OO_{L^*}\ot S(L')$ (where $\partial_{l'}$ is the differentiation on a constant vector bundle over $L^*$).
\end{example}

Thus, for any $\WW(V)$-module $M$ we have a natural object in the derived category of $\wt{\AA}_{\II}$-modules,
$$C_{\II}^\bullet(M):=M^r_{\OO_X\ot \WW(V)}(\II)\otimes_{\WW(V)}^{\bbL} M.$$
Note that if we forget the full $\wt{\AA}_{\II}$-module structure, but only keep the $U(\ov{\HH}_\II)$-module structure
we get exactly derived coinvariants considered in Sec. \ref{der-coinv-sec} (with the difference that now we only consider the Lie algebra structure on $U(\ov{\HH}_\II$).

\begin{remark} All our Lie algebroids contain $\OO$ as a Lie ideal, and the modules over them we construct (or objects of derived category)
are of weight $1$, i.e., $f\in \OO$ acts by multiplication by $f$. In particular, in case of Lagrangians we get modules over the corresponding tdo.
\end{remark}

\begin{example}\label{Heis-ex-1}
Let $L_+\sub V$ be a Lagrangian such that all isotropic subspaces $I\sub V$ of the form $I=\II|_s$ for $s\in S$, 
satisfy $I^\perp+L_+=V$. Then by Corollary \ref{JL-cor}, the derived coinvariants on $S$, $C_{\II}(M(L_+))$ are concentrated in degree $0$ and we have 
$$C_{\II}(M(L_+))\simeq M_{\ov{\HH}_\II}(\ov{L}_+)$$ 
as an $\ov{\HH}_\II$-module, where $\ov{L}_+$ is the image of $L_+\cap \fc(\II)$ in $\fc(\II)/\II=\ov{\HH}_\II$ (note that $\ov{L}_+$ is Lagrangian in
$\ov{\HH}_\II$).
Thus, our construction gives a structure of an $\AA_{\II}$-module on the sheaf of Fock modules $M(\ov{L}_+)$ over $S$.
\end{example}

\begin{remark}
The Lie algebroid $\wt{\AA}_{\II}$ has a richer structure, which is called in \cite[Sec.\ 1]{BS} an {\it Atiyah algebra} structure 
(more precisely, $U(\ov{\HH}_\II)$-Atiyah algebra), which keeps track
of a structure of an associative $\OO$-algebra on the kernel of the anchor map. The modules we construct over $\wt{\AA}_{\II}$ are actually modules
over an Atiyah algebra (so the kernel of the anchor map also acts as an associative algebra).
\end{remark}

\subsubsection{Pull-backs}

Recall that one has a notion of pull-back for Lie (resp., Picard) algebroids with respect to any morphisms between smooth superschemes $f:Y\to X$ 
(see \cite[Sec.\ 2.2]{BB} for the even case).
More concretely, for a Lie algebroid $\AA_X$ on $X$, one sets 
$$f^\bullet \AA_X:=f^*\AA_X\times_{f^*\TT_X} \TT_Y,$$
with the anchor map $f^\bullet \AA_X\to \TT_Y$ given by the natural projection, and the (super) bracket given by
\begin{align*}
&[(\sum_i \phi_i\cdot a_i,v), (\sum_j\phi'_j\cdot a'_j,v')] \\
&=
(\sum (-1)^{\ov{a_i}\cdot \ov{\phi'_j}}\phi_i\phi'_j\cdot [a_i,a'_j]+\sum_j v(\phi'_j)a'_j-(-1)^{\ov{v}\cdot\ov{v'}}\sum_i v'(\phi_i)a_i,[v,v']),
\end{align*}
where $a_i,a'_j\in f^{-1}\AA_X$, $\phi_i,\phi'_j\in \OO_Y$, $v,v'\in \TT_Y$.


Recall also (see \cite{BB}) that given a Lie algebroid $\AA$ on $X$, and a module $M$ over $\AA$, and a morphism $f:Y\to X$,
the pullback $f^*M$ has a natural structure of $f^\bullet\AA$-module.

\begin{lemma}\label{alg-pull-back-lem}
The formation of the Lie algebroid $\AA_\II$ is compatible with pull-backs: for an isotropic subbundle $\II\sub \OO_S\ot \WW_{\le 1}(V)$,
for a morphism $f:S'\to S$ one has 
$$\AA_{f^*\II}\simeq f^\bullet\AA_\II.$$
The natural isomorphism
$$f^*M^r(\II)\simeq M^r(f^*\II)$$
is compatible with the actions of $\AA_{f^*\II}$.
\end{lemma}

\begin{proof} This easily follows from Proposition \ref{anchor-sur-prop}(i).
\end{proof}

Thus, our Lie algebroids are pull-backs of the universal one on the isotropic Grassmannian of $V$. In Sec.\ \ref{another-constr-sec} below we will give a different construction of
this universal Lie algebroid.



\subsubsection{Pair of varying isotropic subspaces}

Assume now we have two isotropic subbundles $J,J'\sub \OO_S\ot \WW_{\le 1}(V)$.
Then the construction of Sec.\ \ref{constr-main-sec} gives a Lie algebroid $\AA_J$ on $S$, such that $M^r(J)$ has a structure of an
$\AA_J-\WW(V)$-bimodule, as well as a Lie algebroid $\AA_{J'}$, such that $M(J')$ has a structure of a $\WW(V)-\AA_{J'}^{op}$-bimodule.
Hence, the derived coinvariants object $M^r(J)\otimes^{\bbL}_{\WW(V)} M(J')$ has a structure of $\AA_J-\AA_{J'}$-bimodule.


\subsection{Another construction of the universal Lie algebroids}
\label{another-constr-sec}

Let $V$ be a symplectic super vector space, $H=\C\oplus V$ the corresponding Heisenberg Lie algebra.
Let $IG(H)=IG_d(H)$ denote the Grassmannian of isotropic subspaces in $H$ of given dimension $d$. 
Note that $IG(H)$ is a bundle over the usual isotropic Grassmannian $IG(V)$ with the fiber $I^*\simeq V/I$ over $I\sub V$.

We have a transitive action of the symplectic (super)group $\Sp(V)$ on $IG(V)$ and of the affine symplectic group $\Sp(V)\ltimes V$ on $IG(H)$. 
Hence, the Lie algebra $\fg:=\mathfrak{sp}(V)$ (resp., $\fg\ltimes V$) acts on $IG(V)$ (resp., $IG(H)$), 
so we have the corresponding Lie algebroid 
$$\a:\OO\ot \fg \to \TT_{IG(V)} \ ( \text{resp.}, \ \a_H:\OO\ot (\fg\ltimes V)\to \TT_{IG(H)}).$$

Let us fix the base point $I_0\in IG(V)$.
The kernel of the anchor map $K:=\ker(\a)$ is the $\Sp(V)$-equivariant bundle associated with the adjoint representation $\fp_{I_0}$ of the stabilizer subgroup
$P_{I_0}\sub \Sp(V)$. 
Similarly, for the base point $(I_0,0)\in IG(H)$, $K_H:=\ker(\a_H)$ is the $\Sp(V)\ltimes V$-equivariant bundle associated with the adjoint representation
$\fp_{I_0}\ltimes I_0$ of the stabilizer subgroup $P_{I_0}\ltimes I_0$.

We have
$$\fp_{I_0}=S^2(I_0^\perp)+I_0V\sub S^2(V)=\fg,$$
and $I_0V\sub \fp_{I_0}$ is a Lie ideal. Furthermore, the embedding $I_0V\to \mathfrak{sp}(V)$ lifts to a natural $P_{I_0}$-equivariant map
\begin{equation}\label{i-stab-equiv-map}
i=i^l_{I_0}:I_0V\to \wt{\fg}:=\WW^{ev}_{\le 2}(V)
\end{equation}
with the image $I_0\cdot V\sub \wt{\fg}$ (we just take the product of $x\in I_0$ and $v\in V$ in $\WW^{ev}_{\le 2}(V)$ with $x$ on the left).
We also have a $P_{I_0}\ltimes I_0$-equivariant map
\begin{equation}\label{i-stab-equiv-map-bis}
i:I_0V\oplus I_0\to \WW_{\le 2}(V)
\end{equation}
with the image $I_0\cdot V+I_0$.

Let $K_0\sub K$ (resp., $K_{H,0}\sub K_H$) denote the $\Sp(V)$-equivariant subbundle corresponding to the subrepresentation of the stabilizer,
$I_0V\sub \fp_{I_0}$ (resp., $I_0V\oplus I_0\sub \fp_{I_0}\ltimes I_0$).
Note that the fiber of $K_0$ at $I$ (resp., of $K_{V,0}$ at $(I, v\in V/I)$) is naturally identified with $IV$ (resp., $\exp(\ad_v)(IV+I)$).

From the map \eqref{i-stab-equiv-map} (resp., \eqref{i-stab-equiv-map-bis}) we get
an $\Sp(V)$-equivariant (resp., $\Sp(V)\ltimes V$-equivariant) embedding of bundles
$$i:K_0\to \OO\ot \wt{\fg} \ (\text{resp.}, \ i:K_{H,0}\to \OO\ot \WW_{\le 2}(V))$$ 
whose image is contained in the kernel of the anchor map.
We claim that $i(K_0)$ (resp., $i(K_{H,0})$) is actually a Lie ideal in $\OO_X\ot \wt{\fg}$ (resp., $\OO_X\ot \WW_{\le 2}(V)$) with respect
to the Lie algebroid bracket. This follows from the following general result.

\begin{lemma}\label{Lie-ideal-lem}
Let $G$ be an algebraic group acting on a variety $X$, and let $\C\to \wt{\fg}\to \fg$ be a central extension of the Lie algebra of $G$, equipped with a $G$-equivariant structure
(i.e., $G$ acts on $\wt{\fg}$ inducing the adjoint action of $\fg$ on $\wt{\fg}$). Then any $G$-equivariant subbundle of $\ker(\OO\ot \wt{\fg}\to T_X)$ is a Lie ideal in
$\OO\ot \wt{\fg}$.
\end{lemma}

\begin{proof}
First, we recall some standard facts about $G$-equivariant bundles. 
For any $G$-equivariant bundle $E$ on $X$, there is a natural $\fg$-action on the sheaf $E$ that endows $E$ with a structure of a module over the Lie algebroid $\OO_X\ot \fg$.
This construction is compatible with morphisms of $G$-equivariant bundles. Furthermore, if $E=\OO_X\ot V$, where $V$ is a $G$-representation,
then the $\fg$-representation on $\Ga(U,E)=\OO_X(U)\ot_{\C} V$ is the tensor product of the natural $\fg$-representations on $\OO_X(U)$ and on $V$.

We deduce that any $G$-equivariant subbundle $E\sub \OO\ot \wt{\fg}$ is preserved by the action of $\fg$ on $\OO\ot \wt{\fg}$. The latter action is induced by
the Lie bracket with $1\ot X$, where $X\in \wt{\fg}$. If in addition, $E$ lies in the kernel of the anchor map, this implies that $E$ is preserved by the Lie bracket with $f\ot X$,
for any local function $f$ on $X$, as claimed.
\end{proof}

Hence, the quotients
$$\AA'_{IG(V)}:=\OO\ot \WW^{ev}_{\le 2}(V)/i(K_0),$$
$$\AA'_{IG(H)}:=\OO\ot \WW_{\le 2}(V)/i(K_{H,0})$$
have a natural structure of a Lie algebroid over $IG(V)$ and $IG(H)$, respectively. 
Furthermore, their anchor maps fit into exact sequences
$$0\to \WW^{ev}_{\le 2}(\II^\perp/\II)\to \AA'_{IG(V)}\rTo{\a} T_{IG(V)}\to 0$$
$$0\to U_{\le 2}(\fc(\II)/\II)\to \AA'_{IG(H)}\rTo{\a} T_{IG(H)}\to 0$$
where $\II\sub \OO\ot V$ (resp., $\II\sub \OO\ot H$) is the universal isotropic subbundle. 
Note that by construction we have homomorphisms of Lie algebras
$$\WW^{ev}_{\le 2}(V)\to \Ga(IG(V),\AA'_{IG(V)}), \ \ \WW_{\le 2}(V)\to \Ga(IG(H),\AA'_{IG(H)}).$$

Let us denote by $\AA^{ev}_{IG(V)}$ (resp., $\AA_{IG(H)}$) the Lie algebroids associated by the construction of Sec.\ \ref{constr-main-sec}
with the universal isotropic subbundle on $IG(V)$ (resp., on $IG(H)$).

\begin{theorem}\label{two-def-thm}
There are unique isomorphisms of Lie algebroids,
\begin{equation}\label{two-def-isom}
\AA'_{IG(V)}\simeq \AA^{ev}_{IG(V)}, \ \ \AA'_{IG(H)}\simeq \AA_{IG(H)},
\end{equation}
identical on $\WW^{ev}_{\le 2}(\II^\perp/\II)$ (resp., on $U_{\le 2}(\fc(\II)/\II)$), and
such that 
the operator $D_x\in \AA^{ev}_{IG(V)}\sub \End_{\WW(V)^{op}}(M^r(\II))$ associated with $x\in \WW^{ev}_{\le 2}(V)$ is induced by the operator on 
$\OO\ot \WW(V)$,
$$D_x(f\ot y)=v_x(f)\ot y+(-1)^{\ov{f}\ov{x}}f\cdot xy,$$
where $v_x$ is the vector field on $IG(V)$ coming from the projection of $x$ to $\mathfrak{sp}(V)$
(and similarly for $D_x\in\AA_{IG(H)}$ associated with $x\in \WW_{\le 2}(V)$).
\end{theorem}

\begin{proof} We discuss the case of $IG=IG(V)$; the case of $IG(H)$ is similar.
We want to define a structure of an $\AA'_{IG}$-module on $M^r(\II)$.
By definition $\AA'_{IG}$ is the quotient of the Lie algebroid $\OO\ot \wt{\fg}$ by the ideal $\II\cdot V$, so we should start by defining
an $\OO\ot \wt{\fg}$-module structure.

Given a $\WW(V)$-module $M$, we can consider $M$ as a representation of the Lie algebra $\WW_{\le 2}(V)$.
Hence, $M_{IG}:=\OO_{IG}\ot M$ has a natural structure of a module over the Lie algebroid $\OO_{IG}\ot \WW_{\le 2}(V)$ (and hence, of
its subalgebroid $\OO_{IG}\ot \wt{\fg}$), given by
\begin{equation}\label{tensor-product-Lie-alg-action}
(f\ot v)(f'\ot m)=f\cdot v(f')\ot m+(-1)^{\ov{f}'\ov{v}}ff'\ot v(m).
\end{equation}
For $M=\WW(V)$ this gives a structure of $\OO_{IG}\ot \WW_{\le 2}(V)$-module on $\OO_{IG}\ot \WW(V)$, which commutes
with the right action of $\WW(V)$.
 
Now, similarly to Lemma \ref{Lie-ideal-lem}, since the subbundle $\II\sub \OO_{IG}\ot V$ is
$\Sp(V)$-equivariant, it is preserved by the action of the Lie algebra $\fg$. This implies that
$$\II\cdot (\OO_{IG}\ot \WW(V))\sub \OO_{IG}\ot \WW(V)$$
is an $\OO_{IG}\ot \wt{\fg}$-submodule. Hence, 
we get an $\AA'_{IG}$-module structure on
$$M':=\OO_{IG}\ot \WW(V)/(\II\cdot (\OO_{IG}\ot \WW(V))$$
Note that since $\II$ is contained in the kernel of the anchor map, its action on $\OO_{IG}\ot \WW(V)$ is just an obvious $\OO$-linear action
induced by the $\OO$-linear algebra structure on $\OO_{IG}\ot \WW(V)$, hence, we get an identification $M'=M^r(\II)$.

The $\AA'_{IG}$-action preserves the $\Z/2$-grading, so we get
a homomorphism of Lie algebroids $\AA'_{IG}\to \AA^{ev}$.  
The compatibility of $\WW_{\le 2}^{ev}(\II^\perp/\II)$-actions follows immediately from the definitions, using the fact that $\WW_{\le 2}(\II^\perp)\sub \OO\ot \wt{\fg}$
belongs to the kernel of the anchor map.
\end{proof}

\begin{remark}
The isomorphisms \eqref{two-def-isom} give rise to Lie algebra homomorphisms
\begin{equation}\label{global-diff-oper-eq}
\ga^{ev}:\WW_{\le 2}^{ev}(V)\to H^0(IG(V),\AA^{ev}_{IG(V)}), \ \ \ga:\WW_{\le 2}(V)\to H^0(IG(H),\AA_{IG(H)}).
\end{equation}
In particular, the Lie algebra $\WW_{\le 2}(V)$ acts on the sheaves of $\II$-coinvariants over $IG(H)$.
\end{remark}



\subsection{Generalization to a flat Heisenberg bundle}\label{flat-Heis-sec}

Now assume that instead of a fixed symplectic super vector space $V$, we have a Heisenberg extension
$$0\to \OO_S\to \HH\to \VV\to 0$$ 
of a symplectic super vector bundle $\VV$ over $S$. 
We assume that $\HH$ is equipped with a {\it Heisenberg flat connection}, i.e., a flat connection $\nabla$, compatible with the bracket and preserving $\OO_S$, so that the induced
connection on $\OO_S$ is standard. 

Then we can consider the following analog of Definition \ref{main-def}, for an isotropic subbundle $\II\sub \HH$.

\begin{definition}\label{conn-alg-def}
Let $\AA_{\HH,\nabla,\II}$ denote the subsheaf of the Atiyah Lie algebroid of $M^r_{\HH}(\II)$ consisting of operators $D$ such that
\begin{itemize}
\item
$D(m\cdot x)=D(m)\cdot x+(-1)^{\ov{D}\ov{m}}m\cdot \nabla_{\si(D)}(x),$
where $\si(D)\in \TT_X$ is the symbol of $D$.
\item $[D,U_{\le 1}(\ov{\HH}_\II)]\sub U_{\le 1}(\ov{\HH}_\II)$.
\end{itemize}
\end{definition}

Clearly, $\AA_{\HH,\nabla,\II}$ is a sub- Lie algebroid of the Atiyah Lie algebroid of $M^r(\II)$, containing $\OO_X$.
In the case when $\HH=\OO_S\ot \WW_{\le 1}(V)$, with the standard connection, we recover the construction of Sec.\ \ref{constr-main-sec}.

In the case when $\HH$ is split, i.e., $\HH=\OO_S\oplus \VV$ (compatibly with $\nabla$), and $\II\sub \VV$, we define 
$$\AA^{ev}_{\HH,\nabla,\II}\sub \AA_{\HH,\nabla,\II}$$
as the subsheaf of operators which are even with respect to the $\Z/2$-grading on $M^r_{\VV}(\II)$.

\begin{prop}\label{flat-conn-case-prop}
(i) $\AA_{\HH,\nabla,\II}$ is a Lie algebroid, and its anchor map fits into the exact sequence
$$0\to U_{\le 2}(\ov{\HH}_\II)\to \AA_{\HH,\nabla,\II}\to \TT_S\to 0.$$
The action of $D\in\AA_{\HH,\nabla,\II}$ on $M^r_{\HH}(\II)$ satisfies \eqref{D-filtration-property}.

In the split case, $\AA^{ev}_{\HH,\nabla,\II}$ is a sub- Lie algebroid, with the anchor map fitting into the exact sequence
$$0\to \WW^{ev}_{\le 2}(\II^\perp/\II)\to \AA^{ev}_{\HH,\nabla,\II}\to \TT_S\to 0.$$

\noindent
(ii) Let $M$ be a left $U(\HH)$-module, equipped with a flat connection $\nabla^M$, such that
$\nabla^M_v(xm)=\nabla_v(x)m+(-1)^{\ov{v}\ov{x}}x\nabla^M_v(m)$.
Then $M^r_{\HH}(\II)\ot_{U(\HH)} M$ has an $\AA_{\HH,\nabla,\II}$-structure given by
\begin{equation}\label{left-connection-action}
D(m_1\ot m_2)=D(m_1)\ot m_2+(-1)^{\ov{m}_1\ov{D}}m_1\ot \nabla^M_{\si(D)}(m_2).
\end{equation}
The derived coinvariants $C^\bullet_{\II}(M)$ also can be lifted to the derived category of $\AA_{\HH,\nabla,\II}$-modules.
This construction is functorial in $M$ (with respect to morphisms of $U(\HH)$-modules compatible with connections).
\end{prop}

\begin{proof}
(i) To prove surjectivity we can pass to formal completion at a point. Then we can find a horizontal trivialization of a symplectic bundle $\VV$,
and the assertion follows from Proposition \ref{anchor-sur-prop}(i).

\noindent
(ii) The fact that \eqref{left-connection-action} is well defined and gives an $\AA_{\HH,\nabla,\II}$-module structure is a straightforward check.
To lift $C^\bullet_{\II}(M)$ to the derived category of $\AA_{\HH,\nabla,\II}$-modules, we use the bar resolution
$$\ldots\to U(\HH)\ot_{\OO} U(\HH)\ot_{\OO} M\to U(\HH)\ot_{\OO} M\to M\to 0$$
of $M$, and observe that every term has a natural flat connection induced by $\nabla$ and $\nabla^M$, and that the differentials
are compatible with the connections.
\end{proof}

\begin{cor}\label{conf-bl-cor} Let $J\sub \HH$ be an isotropic subbundle, preserved by $\nabla$. Then the natural map
$$M^r_{\HH}(\II)\to C_{\II}(M_{\HH}(J))$$
is compatible with the $\AA_{\HH,\nabla,\II}$-module structures.
\end{cor}

Note that in the case when $\II$ and $J\sub \HH$ are Lagrangian and the images of $J|_s$ and $\II|_s$ in $\VV|_s$ are transversal then the coinvariants sheaf
$C_{\II}(M_{\HH}(J))$ is a line bundle on $S$. In this case we can view the construction of Corollary \ref{conf-bl-cor} as a replacement
for considering (local) horizontal conformal blocks.

\begin{example}\label{I0-algebroid-ex}
In the case $\II=0$, our Lie algebroid is a semi-direct product,
$$\AA_{\HH,\nabla,0}=U_{\le 2}(\HH)\rtimes \TT_S,$$
where $\xi\in \TT_S$ corresponds to the operator $\nabla_\xi$ on $M^r_{\HH}(0)=U(\HH)$.
\end{example}

\subsubsection{Picard algebroid of a virtual Fock module}

Assume now that $M$ is a right $\HH$-module, such that locally there exists a Lagrangian $L\sub \HH$, such that
$M\simeq M^r(L)$. In this case we say that $M$ is a {\it virtual Fock module}.

Repeating Definition \ref{conn-alg-def} for $M$, we get a Picard algebroid $\AA_{\HH,\nabla,M}$ acting on $M$ (commuting with the right action of horizontal sections of $\HH$).

For example, we get an example of a (left) virtual Fock module in the situation of Theorem \ref{Pf-thm}, assuming additionally that $\ov{\HH}_J$ has a flat connection.
This situation occurs for the moduli of supercurves.

\subsection{Compatibility with isotropic reduction}\label{alg-isotr-red-sec}

Let $\HH$ be a Heisenberg extension over a super scheme $S$, equipped with a flat connection $\nabla$, as in Sec.\ \ref{flat-Heis-sec}.
Let $I\sub \HH$ be an isotropic subbundle, preserved by $\nabla$. Then the reduced Heisenberg extension
$\ov{\HH}_I=\fc(I)/I$ has an induced flat connection, which we still denote as $\nabla$. 

Assume that an isotropic subbundle $\JJ\sub \HH$ is such that $\fc(I)+\JJ=\HH$.
Then $\ov{\JJ}:=\fc(I)\cap \JJ$ can be viewed as an isotropic subbundle in $\ov{\HH}_I$, and
\begin{equation}\label{isotr-red-main-isom-repeated}
M^r_{\ov{\HH}_I}(\ov{\JJ})\simeq M^r_{\HH}(\JJ)\ot_{U(\fc(I))}U(\ov{\HH}_I)
\end{equation}
(see \eqref{isotr-red-main-isom-1}).

\begin{prop}\label{isotr-red-Lie-alg-prop}
(i) We have an isomorphism of Lie algebroids,
$$\AA_{\HH,\nabla,\JJ}\to \AA_{\ov{\HH}_I,\nabla,\ov{\JJ}},$$
sending $D$ to the operator
$$m\ot h\mapsto D(m)\ot h+(-1)^{\ov{D}\ov{m}}m\ot \nabla_{\si(D)}(h),$$

\noindent
(ii) For an $\HH$-module $M$, equipped with a flat connection (compatible with the $\HH$-action), such that $I$ acts locally nilpotently on $M$, the isomorphism 
\eqref{redII-conv-isom},
$$C_{\ov{\JJ}}(M^I)\rTo{\sim} C_{\JJ}(M)$$
lifts to the derived category of $\AA_{\HH,\nabla,\JJ}$-modules.
In particular, this holds if $M=M(J')$, for an isotropic subbundle $J'\sub\HH$ containing $I$ and preserved by $\nabla$, in which case $M^I=M_{\ov{\HH}_I}(J'/I)$.
\end{prop}

\begin{proof} 
(i) One immediately checks that this map of Lie algebroids is well defined, and restricts to
the isomorphism $\fc(J)/J\rTo{\sim} \fc(\ov{J})/\ov{J}$ on the kernels of the 
the anchor maps. Hence, it is an isomorphism.

\noindent
(ii) For the usual coinvariants, this is deduced easily using the fact that both the isomorphism \eqref{redII-conv-isom} and the identification of $\JJ$ and $\ov{\JJ}$ coinvariants
are constructed using \eqref{isotr-red-main-isom-repeated}. For derived coinvariants, we first use bar-resolution of $M$ as in the proof of Proposition \ref{flat-conn-case-prop}(ii).
\end{proof}



\subsection{The Picard algebroid associated with a Lagrangian in a split Heisenberg extension}\label{Ber-sec}

Recall that in the case of a Lagrangian subbundle $\LL\sub \HH$, where $\HH$ is equipped with a flat connection, 
our construction gives a Picard algebroid $\AA_{\HH,\nabla,\LL}$ (see Sec.\ \ref{constr-main-sec}).
In the case when $\HH=\VV\oplus \OO$ and $\LL\sub \VV$, 
we can relate $\AA_{\HH,\nabla,\LL}$ to the determinant line bundle $\Ber(\LL)$.

Given a line bundle $L$ and a complex number $\la$, we denote by $\AA_{L^\la}$ the Picard algebroid obtained from the Atiyah Lie algebroid
$\AA_L$ by the push-out $\la:\OO\to \OO$.

\begin{theorem}\label{LG-alg-Ber-thm}
Let $\VV$ be a symplectic super vector bundle, equipped with a flat symplectic connection $\nabla$, and let $\LL\sub \VV$ be a Lagrangian subbundle.
Then
$$\AA_{\VV,\nabla,\LL}\simeq \AA_{\Ber(\LL)^{1/2}}.$$
More generally, if $M$ is a right virtual Fock module over $\WW(\VV)$, then
$$\AA_{\VV,\nabla,M}\simeq \AA_{L(M)^{1/2}},$$
where $L(M)$ is a certain line bundle associated with $M$.
\end{theorem}

\begin{proof} Consider the doubled symplectic bundle $\VV^{(2)}=(\VV\oplus\VV,\om\oplus\om)$ with the induced flat symplectic connection.
Then $\De_i(\VV)=\{(x,ix)\ |\ x\in \VV\}$ is Lagrangian subbundle in $\VV^{(2)}$, preserved by the connection. On the other hand, $\LL^{(2)}\sub \VV^{(2)}$ is also Lagrangian.
Since $\De_i(\VV)\cap \LL^{(2)}\simeq \LL$, by Theorem \ref{Ber-coinv-thm}, we have an isomorphism
$$C_{\LL^{(2)}}(M_{\VV^{(2)}}(\De_i(\VV)))\simeq \Ber(\LL)[n],$$
where $V$ has rank $(2n|2m)$. 
By Proposition \ref{flat-conn-case-prop}(ii), we get an $\AA_{\VV^{(2)},\nabla,\LL^{(2)}}$-module structure on $\Ber(\LL)$, so that $1\in\OO$ acts by the identity.
Hence, we get an isomorphism of Picard algebroids
$$\AA_{\VV^{(2)},\nabla,\LL^{(2)}}\simeq \AA_{\Ber(\LL)}.$$

On the other hand, we can construct a morphism of Lie algebroids
\begin{equation}\label{diagonal-alg-mor}
\AA_{\VV,\nabla,\LL}\to \AA_{\VV^{(2)},\nabla,\LL^{(2)}}
\end{equation}
acting as $2\id$ on $\OO\sub \AA_{\VV,\nabla,\LL}$.
Namely, to an operator $D:M^r(\LL)\to M^r(\LL)$ we associate the operator 
$D^{(2)}:=D\ot\id+\id\ot D$ on 
$$M^r_{\VV^{(2)}}(\LL^{(2)})\simeq M^r_{\VV}(\LL)\ot_{\OO} M^r_{\VV}(\LL).$$
One immediately checks that $D^{(2)}$ is compatible with the right action of $\WW(\VV^{(2)})$ and $\nabla$, and so induces a morphism of Lie algebroids
\eqref{diagonal-alg-mor}, acting as $2\id$ on $\OO$. This gives the requred identification of $\AA_{\VV,\nabla,\LL}$.

In the case of a right virtual Fock module $M$, we can repeat the above construction with $M(\VV^{(2)})$ replaced by $M\ot_{\OO} M$.
Then the above argument shows that $(M\ot_\OO M)\ot_{\WW(\VV^{(2)})} M_{\VV^{(2)}}(\De_i(\VV))$ is a line bundle $L(M)$ in degree $-n$.
and we get an identification of $\AA_{\VV,\nabla,M}$ with $\AA_{L(M)^{1/2}}$. 
\end{proof}

\begin{remark} 1. In the case when there exists a Lagrangian subbundle $L_0\sub \VV$ 
preserved by the connection, such that $\LL$ and $L_0$ have purely odd intersection, the algebroid $\AA_\LL$ acts on
the sheaf of coinvariants $C_{\LL}(M(L_0))$, which is a line bundle, isomorphic to the square root of the relevant Berezinian bundle
(see Theorem \ref{theta-Pf-thm}).

\noindent
2. For arbitrary Heisenberg extension $0\to \OO_S\to \HH\to \VV\to 0$, with a flat connection $\nabla$, and a Lagrangian subbundle $\LL$,
we can define the opposite extension $\HH'$ as the push-forward with respect to the map $-\id:\OO_S\to \OO_S$, so that there is a Lie algebra
isomorphism $\iota:\HH\to \HH'$ acting as $-\id$ on $\OO_S$ (and inducing the identity on $\VV$). Let $\LL'=\iota(\LL)$. Then the argument
of Theorem \ref{LG-alg-Ber-thm} gives an isomorphism for the Baer sum of the Picard algebroids,
$$\AA_{\HH,\nabla,\LL}+\AA_{\HH',\nabla,\LL'}\simeq \AA_{\Ber(\LL)}$$
In the split case we have an additional isomorphism of Heisenberg extensions $\HH\simeq \HH'$ sending $\LL$ to $\LL'$.
In general, $\AA_{\HH,\nabla,\LL}$ is not isomorphic to $\AA_{\Ber(\LL)^{1/2}}$, see appendix \ref{P1-example-sec}.
\end{remark}

\subsection{Picard algebroids for families of isotropic subspaces with a flat reduction}\label{isotropic-PL-sec}

Let $V$ be a symplectic super vector space, and let $\II\sub \HH:=\OO_S\ot \WW_{\le 1}(V)$ an isotropic subbundle.
We have the corresponding Lie algebroid $\AA_{\II}$ on $S$.
Recall (see Proposition \ref{anchor-sur-prop}(ii)) 
that $\AA_\II/\OO_S$ is identified with the Atiyah Lie algebroid $\AA_{\ov{\HH}_\II}$ of infinitesimal symmetries of the Heisenberg Lie algebra
$\ov{\HH}_{\II}$.

Now suppose we have a flat Heisenberg connection $\nabla$ on $\ov{\HH}_{\II}$. 
We can view $\nabla$ as a splitting $\TT\to \AA_{\ov{\HH}_\II}$.
Thus, we can define a Picard algebroid $\AA^{Pic}_\II:=\AA^{Pic}_\II(\nabla)$ on $S$ as the pullback via $\nabla$:
\begin{diagram}
0&\rTo{}& \OO_X&\rTo{}& \AA^{Pic}_\II&\rTo{}&\TT&\rTo{}&0\\
&&\dTo{\id}&&\dTo{}&&\dTo{\nabla}\\
0&\rTo{}& \OO_X&\rTo{}& \AA_{\II}&\rTo{p}&\AA_{\ov{\HH}_\II}&\rTo{}&0
\end{diagram}

\begin{remark}
One can also consider a more general construction of a Picard algebroid starting from
\begin{itemize}
\item a Heisenberg extension $\HH$, equipped with a Heisenberg flat connection;
\item an isotropic subbundle $\II\sub\HH$;
\item a Heisenberg connection on $\ov{\HH}_\II$. 
\end{itemize}
We are not aware of any natural examples requiring such a construction.
\end{remark}

Since we have a homomorphism of Lie algebroids $\AA^{Pic}_\II\to \AA_\II$, we get a structure of an $\AA^{Pic}_\II$-module on the sheaves of derived
$\II$-coinvariants of a $\WW(V)$-module $M$. 

\begin{lemma} For any $\WW(V)$-module $M$, the action of $\AA^{Pic}_\II$ on $C_\II(M)$ 
is compatible with the action of $\ov{\HH}_\II$ as follows
\begin{equation}\label{conn-comp-eq}
v\cdot (x\cdot m)=\nabla_{\a(v)}(x)\cdot m+(-1)^{\ov{v}\ov{x}}x\cdot (v\cdot m),
\end{equation}
where $v$ is a local section of $\AA^{Pic}_\II$, $x$ is a local section of $\ov{\HH}_\II$, $m\in C_\II(M)$.
\end{lemma}

\begin{proof}
It is enough to check that
$$[v,x]=\nabla_{\a(v)}(x)$$
in the action on $M^r_{\HH}(\II)$.
The left-hand side represents the action of $p(v)\in \AA_{\ov{\HH}_\II}$ on $x\in\ov{\HH}_\II$, so by the definition of $\AA^{Pic}_{\II}$ it agrees with the action of $\nabla_{\a(v)}$.
\end{proof}

\begin{prop}\label{virt-Fock-prop} 
In the above situation, assume in addition that we have a Lagrangian $L\sub \HH$, such that $\II|_s^+\cap L|_s^+=0$ for any $s\in S$.
Consider the corresponding virtual Fock module $M:=M_{\HH}(L)_{\II}$ over $\ov{\HH}_\II$ (see Theorem \ref{Pf-thm}). Then 
we have a unique isomorphism of Picard algebroids,
$$\AA^{Pic}_\II\simeq \AA^{op}_{\ov{\HH}_\II,\nabla,M},$$
compatible with the actions on $M$.
\end{prop}

\begin{proof} The identity \eqref{conn-comp-eq} for the action of $\AA^{Pic}_\II$ on $M=C_\II(M_\HH(L))$ 
implies that the action of $\AA^{Pic}_\II$ on $M$ gives a morphism
of Picard-Lie algebroids $\AA^{Pic}_\II\to \AA_{\ov{\HH}_I,\nabla,M}^{op}$, hence, an isomorphism.
\end{proof}

Taking into account Corollary \ref{JL-cor}, we get the following result.

\begin{cor}\label{Pic-alg-cor}
In the situation of Proposition \ref{virt-Fock-prop}, if $\fc(\II)+L=\HH$ then we have
$$\AA^{Pic}_\II\simeq \AA^{op}_{\ov{\HH}_\II,\nabla,\ov{L}},$$
where $\ov{L}=\fc(\II)\cap L$ is the reduced Lagrangian in $\ov{\HH}_\II$,
compatibly with the actions on
$$M_{\HH}(L)_\II\simeq M_{\ov{\HH}}(\ov{L}).$$
\end{cor}

\subsection{Connections associated with a transversal data}\label{local-flat-conn-sec}

Let $\HH\to \VV$ be a Heisenberg extension, equipped with a Heisenberg  flat connection $\nabla$, as in Sec.\ \ref{flat-Heis-sec},
and let $\II\sub \HH$ be an isotropic subbundle. Assume that we have another isotropic subbundle $J\sub \HH$
such that the commutator induces a perfect pairing between $\II$ and $J$.
In this case the images of $\II$ and $J$ in $\VV$ satisfy
$$\VV=\II\oplus J^\perp=\II^\perp\oplus J=\II\oplus (\II^\perp\cap J^\perp)\oplus J.$$
In particular, we get an identification compatible with the Heisenberg extension structures, 
\begin{equation}\label{IperpI-trivialization}
\fc(\II)/\II\simeq \fc(J)/J.
\end{equation}

In this situation we can construct a connection $\nabla_J$ on the Lie algebroid $\AA_{\HH,\nabla,\II}$, i.e., an $\OO_S$-linear splitting of the anchor map,
$$\nabla^J:\TT\to \AA_{\HH,\nabla,\II}.$$
Recall that $\AA_{\HH,\nabla,\II}$ is realized as an algebra of operators on $\MM^r_{\HH}(\II)$.

\begin{prop}\label{transv-flat-connections-prop}
(i) In the above situation there exists a unique connection $\nabla^J$ on $\AA_{\HH,\nabla,\II}$, such that
any $D$ in the image of $\nabla^J$ satisfies
\begin{equation}\label{D1-J-condition}
D(1)\in 1\cdot S^2(J)+U_{\le 1}(\ov{\HH}_\II)\cdot 1\cdot J,
\end{equation}
where $1\in M^r_{\HH}(\II)$ is the vacuum vector. 

\noindent
(ii) Assume that $J$ is preserved by $\nabla$. Then the connection $\nabla^J$ is flat.
The induced flat connections $\nabla^J\mod \OO$ on $\AA_{\HH,\nabla,\II}/\OO$ and
$\nabla^J$ on $\ov{\HH}_\II$ (as on a $\AA_{\HH,\nabla,\II}$-module), both come from the flat Heisenberg connection on $\fc(\II)/\II$
induced by the isomorphism \eqref{IperpI-trivialization} and by $\nabla$.
\end{prop}

\begin{proof}
(i) We know that $D(1)$ is in $M^r_{\HH}(\II)_{\le 2}$ (see \eqref{D-filtration-property}). 
The action on the vacuum vector (where $U(\ov{\HH}_I)$ acts on the left and $S(J)$ on the right) induces an isomorphism
$$\kappa:U(\ov{\HH}_\II)\ot S(J)\rTo{\sim}M^r_{\HH}(\II),$$
compatible with filtrations. Hence, by the construction of Sec.\ \ref{another-constr-sec}, for any $D\in \AA_{\HH,\nabla,\II}$, we have
$$D(1)\in M^r_{\HH}(\II)_{\le 2}=\kappa[S^2(J)+U_{\le 1}(\ov{\HH}_\II)\cdot J+U_{\le 2}(\ov{\HH})].$$
Therefore, there is a unique way to add an element of $U_{\le 2}(\ov{\HH})$ to $D$, to get an element satisfying
\eqref{D1-J-condition}. This shows the existence of a unique $\OO_S$-linear splitting $\nabla^J$ of the anchor map.

\noindent
(ii) We can use the isotropic reduction with respect to $J$.
Namely, by Proposition \ref{isotr-red-Lie-alg-prop}, we have an isomorphism of Lie algebroids,
\begin{equation}\label{A-H-nabla-I-red-isom}
\AA_{\HH,\nabla,\II}\rTo{\sim} \AA_{\ov{\HH}_J,\nabla,0}=U_{\le 2}(\ov{\HH}_J)\rtimes \TT_S,
\end{equation}
where $\TT_S$ correspond to operators $\ov{D}$ on 
$$U(\ov{\HH}_J)\simeq M^r_{\HH}(\II)\ot_{U(\fc(J))} U(\ov{\HH}_J)$$ such that
$\ov{D}(1)=0$
(note that the action on the vacuum vector gives an isomorphism $U(\fc(J))\simeq M^r_{\HH}(\II)$).
For $D\in \AA_{\HH,\nabla,\II}$ the induced operator $\ov{D}$ on $U(\ov{\HH}_J)$ (see Prop. \ref{isotr-red-Lie-alg-prop})
satisfies 
$$\ov{D}(1)=D(1)\ot 1\in M^r_{\HH}(\II)\ot_{U(\fc(J))} U(\ov{\HH}_J).$$ 
It is clear that if $D$ satisfies \eqref{D1-J-condition} then $D(1)\ot 1=0$. 
Hence, our connection $\nabla^J$ corresponds to the natural flat connection on $U_{\le 2}(\ov{\HH}_J)\rtimes \TT_S$.

The second assertion follows easily from the fact that the restriction of the isomorphism \eqref{A-H-nabla-I-red-isom} to
$U_{\le 2}(\ov{\HH}_\II)\sub \AA_{\HH,\nabla,\II}$ is the isomorphism $U_{\le 2}(\fc(\II)/\II)\to U_{\le 2}(\fc(J)/J)$ 
induced by \eqref{IperpI-trivialization}.
\end{proof}


In the Lagrangian case our construction simplifies as follows. Let $\LL\sub \HH$ be a Lagrangian, and let $L'\sub \HH$ be a complementary Lagrangian. 

\begin{cor}\label{Lag-open-conn-cor}
There exists a unique connection $\nabla^{L'}$ on $\AA_{\HH,\nabla,\LL}$, such that
any $D$ in the image of $\nabla^{L'}$ satisfies
$$D(1)\in 1\cdot [S^2(L')+L'].$$
If $L'$ is preserved by $\nabla$ then $\nabla^{L'}$ is flat.
In the split case $\HH=\OO\oplus \VV$, with $\LL\sub \VV$, $L'\sub \VV$, the connection $\nabla^{L'}$ can be viewed as a connection for 
$\AA^{ev}_{\VV,\nabla,\LL}$ such that any $D$ in the image of $\nabla^{L'}$ satisfies
$D(1)\in 1\cdot S^2(L')$.
\end{cor}

\begin{example}
For a fixed Lagrangian $L'$ in a symplectic (super) vector space $V$, consider the open subset $U_{L'}$ in $LG(V)$ of Lagrangians $L$ such that $L\cap L'=0$.
Recall (see Sec.\ \ref{another-constr-sec}) that we can view $\AA^{ev}_{LG(V)}$ as the quotient of $\OO\ot \WW^{ev}_{\le 2}(V)$, and the connection $\nabla^{L'}$
can be seen as the inverse of the composition 
$$\OO\ot S^2(L')\to \OO\ot \WW^{ev}_{\le 2}(V)\to \AA^{ev}_{LG(V)}\to T_{LG(V)},$$
which is an isomorphism. 
\end{example}

Recall that in the case $\HH=\OO\oplus \VV$ (compatibly with $\nabla$) and $\LL\sub\VV$, we have an identification of $\AA_{\VV,\nabla,\LL}$ with the Atyiah algebroid
$\AA_{\Ber(\LL)^{1/2}}$ (see Theorem \ref{LG-alg-Ber-thm}). Hence, in this case connections on $\AA_{\HH,\nabla,\LL}$ can be thought of as connections on $\Ber(\LL)$.
We have the following simple identification of such connections associated with Lagrangian complements $L'$.

\begin{prop}
Let $\LL\sub \VV$ be Lagrangian, and let $L'\sub \VV$ be a Lagrangian complement. Then under the identification $\AA_{\VV,\nabla,\LL}\simeq \AA_{\Ber(\LL)^{1/2}}$,
the connection $\nabla^{L'}$ corresponds to the connection $\Ber(D^\LL)$ on $\Ber(\LL)$ induced by the following connection on $\LL$:
$$D^\LL:\LL\hra \VV\rTo{\nabla}\Om^1\ot \VV=\Om^1\ot (\LL\oplus L')\to \Om^1\ot \LL,$$
where the last arrow is the natural projection.
\end{prop}

\begin{proof}
This is a local question, so we can choose dual bases of $\LL$ and $L'$ and write $\nabla$ as $\nabla=d+\Ga$, where $\Ga$ is a matrix of $1$-forms. Let $a$ denote the $\LL-\LL$ diagonal block of $\Ga$. Then $D^\LL=d+a$, so if $\si$ is the induced basis of $\Ber(\LL)$, then
\begin{equation}\label{str-si-eq}
\Ber(D^\LL)\si=\str(a)\si.
\end{equation}

Now we need to calculate the connection on $\Ber(\LL)$ coming from $\nabla^{L'}$. 
First, let us consider the connection $\DD$ on $\AA_{\VV,\nabla,\LL}$ which associates with a homogeneous vector field $X$ the operator
$$\DD_X=X-R_{q(\Ga_X)}$$
on $M^r(\LL)\simeq S(L')$, where the first term denotes the componentwise differentiation with respect to the chosen basis of $L'$, 
$\Ga_X$ is a symplectic endomorphism of $\VV$, 
$$q:{\mathfrak sp}(\VV)\simeq S^2(V)\to \WW(\VV)_{\le 2}$$
is the symmetric quadratic lift sending $v_1\cdot v_2$ to $(v_1v_2+(-1)^{\ov{v}_1\ov{v}_2}v_2v_1)/2$,
and $R_q(m):=(-1)^{\ov{q}\ov{m}}mq$.

We claim that 
\begin{equation}\label{D-X-1-eq}
\DD_X(1)=-1\cdot q(\Ga_X)\equiv -\frac{1}{2}\str(a_X) \mod S^2(L')+L',
\end{equation}
which is equivalent to
\begin{equation}\label{nabla'-D-rel}
\nabla^{L'}_X=\DD_X+\frac{1}{2}\str(a_X).
\end{equation}
Indeed, we have a decomposition $S^2(\VV)=S^2(\LL)\oplus \LL\cdot L'\oplus S^2(L')$, so
$1\cdot q(\Ga_X)$ modulo $S^2(L')$ depends only on the $\LL\cdot L'$-component of $\Ga_X$.
Let $e_i\in \LL$ and $f_i\in L'$ denote the dual bases. Then $1\cdot q(e_if_j)=0$ for $i\neq j$, while
$$1\cdot q(e_if_i)=1\cdot (e_if_i-1/2)=-1/2.$$
Since $[e_if_i,e_i]=-(-1)^{\ov{e}_i}e_i$, this implies our claim.

Recall that the connection on $\Ber(\LL)$ induced by a connection on $\AA_{\VV,\nabla,\LL}$ comes from the identification
$$(M^r(\LL)\ot M^r(\LL))\ot^{\Bbb L}_{\WW(\VV^{(2)})} M_{\VV^{(2)}}(\De_i)\simeq \Ber(\LL)[n],$$
where $\De_i=\{(v,iv) \ |\ v\in V\}$.
Here we compute the derived tensor product using the two-sided bar resolution and the natural connection on $\WW(\VV^{(2)})$ (coming from $\nabla$).
First, we claim that if we start with $\DD$ then the induced connection on $\Ber(\LL)$ will have $\si$ horizontal.
Indeed, the connection induced by $\DD$ on $M^r(\LL)\ot M^r(\LL)$ is given by $X-R_{Q_X}$, where
$$Q_X=q(\Ga_X)^{(1)}+q(\Ga_X)^{(2)}.$$
On the other hand, one can check that the connection on $M(\De_i)$ is given by $X+L_{Q_X}$ (where $L_{Q_X}$ is the left multiplication by $Q_X$), 
whereas the connection on $\WW(\VV^{(2)})$ is $X+\ad(Q_X)$.
Hence, the difference between the connection on the bar complex and componentwise differentiation is the diagonal inner action, given by \(-R_{Q_X}\) on the right module, \(\operatorname{ad}(Q_X)\) on the algebra factors, and \(L_{Q_X}\) on the left module. This action is null-homotopic by the standard signed insertion of \(Q_X\) into the bar factors.
Hence the induced connection on cohomology is componentwise differentiation, which implies our claim.

Finally, relation \eqref{nabla'-D-rel} implies that the connection on $\Ber(\LL)$ induced by $\nabla^{L'}_X$ differs from the one induced by $\DD$ by
$\str(a_X)$. Comparing with \eqref{str-si-eq} gives the result.
\end{proof}

\subsubsection{Explicit formulas for the $2$-dimensional case}\label{2dim-for-sec}

Let us consider the simplest example when $V$ is the even $2$-dimensional symplectic space
with the basis $e,e^*$, such that $(e^*,e)=1$. 
We think of subspaces in $LG(H)$ as lines $L\sub V$ equipped with functionals $\varphi:L\to \C$, giving a lift of $L$ to $H=\C\oplus V$.
Let $\LL\sub \OO\ot H$ denote the universal Lagrangian, and let $\AA_\LL$ be the corresponding Picard algebroid over $LG(H)$.

Let $U_0\sub LG(H)$ denote the open subset where $e^*\not\in L$.
Let us calculate the flat connection $\nabla_0$ on $\AA_{\LL}|_{U_0}$ given by Corollary \ref{Lag-open-conn-cor}.

We define coordinates $(x,\la)$ on $U_0$
so that $L\sub \C\oplus V$ is spanned by $e+xe^*+\la$
 Over $U_0$ we can identify the right Fock module $M^r(\LL)$ with $1\cdot \OO(U_0)[e^*]$, where $e^*$ acts by multiplication and the right action of $e$ is given by 
$$f\cdot e=\pa_{e^*}f-xe^*f-\la f.$$
 
As in the proof of Proposition \ref{anchor-sur-prop}, we see that over $U_0$, the operators
$\pa_x-(e^*)^2/2$ and $\pa_\la-e^*$ on $M^r(\LL)=\OO(U_0)[e^*]$ are in $\AA_{\LL}$. Therefore,
$$\nabla_0(\pa_x)=\pa_x-(e^*)^2/2, \ \ \nabla_0(\pa_\la)=\pa_\la-e^*.$$
 
We can also calculate the Lie homomorphism (see \eqref{global-diff-oper-eq})
$$\ga:\WW_{\le 2}(V)\to H^0(LG(H),\AA_\LL)\to \AA_\LL(U_0)$$ 
coming from the action of
$\Sp(V)\ltimes V$ on $LG(H)$ (where $v\in V$ sends $(L,\varphi)$ to $(L,\varphi+(v,\cdot))$),
One can check that the corresponding homomorphism of Lie algebras
$$\ov{\ga}:\WW_{\le 2}(V)/\C\to \TT(U_0)$$
is given by
$$\ov{\ga}(e^*)=-\pa_\la, \ \ \ov{\ga}(e)=x\pa_\la,$$
$$\ov{\ga}(\frac{(e^*)^2}{2})=-\pa_x, \ \ \ov{\ga}(ee^*)=2x\pa_x+\la\pa_\la, \ \ \ov{\ga}(\frac{e^2}{2})=-x^2\pa_x-x\la\pa_\la.$$

\begin{prop}\label{ga-calc-prop}
One has 
$$\ga(e^*)=-\nabla_0(\pa_\la), \ \ \ga(e)=\nabla_0(x\pa_\la)-\la,$$
$$\ga(\frac{(e^*)^2}{2})=-\nabla_0(\pa_x), \ \ \ga(ee^*)=\nabla_0(2x\pa_x+\la\pa_\la), \ \ \ga(\frac{e^2}{2})=-\nabla_0(x^2\pa_x+x\la\pa_\la)+\la^2-x.$$
\end{prop}

\begin{proof}
We use the fact that for $x\in \WW_{\le 2}(V)$, the action of $x$ on $M^r(\LL)(U_0)=\OO(U_0)[e^*]$ is given by 
$$f(e^*)\mapsto \ov{\ga}(x)f(e^*)+x\cdot f(e^*)$$
(see Theorem \ref{two-def-thm}). This immediately leads to the formulas for $\ga(e^*)$ and $\ga((e^*)^2)$. In the case when $x$ is either $e$, $ee^*$ or $e^2$, one has
in addition to rewrite $xf(e^*)$ as a polynomial in $e^*$ in $M^r(\LL)$, using the fact that $e\equiv -xe^*-\la \mod(\LL)$. 
\end{proof}
 
\subsubsection{The map to global sections of the coinvariants sheaf}

Let us consider the Grassmannian $LG(H)$ of Lagrangian subspace in the Heisenberg extension $H=\C\oplus V$, where $V$ is a symplectic super vector space.
Let $M=M(L_0)$ be the Fock module over $\WW(V)$ associated with a fixed Lagrangian $L_0$, and let $\LL\sub \OO\ot H$ denote the universal Lagrangian subbundle 
over $LG(H)$. Thus, we can consider the corresponding sheaf of coinvariants $(\OO\ot M)_\LL$.

The natural projection $\OO\ot M\to (\OO\ot M)_\LL$ gives a map 
\begin{equation}\label{I-M-H-map}
I: M\to H^0(LG(H),(\OO\ot M)_\LL),
\end{equation}
It is compatible with the action of $\WW_{\le 2}(V)$, where the action on
the right-hand side comes from the Lie algebra homomorphism $\ga:\WW_{\le 2}(V)\to H^0(\AA_\LL)$ (see \eqref{global-diff-oper-eq}).
It follows that the elements in the image of $I$ satisfy differential equations 
\begin{equation}\label{abstr-heat-eq}
(\ga(v_1)\ga(v_2)-\ga(v_1v_2))I(m)=0.
\end{equation}
Indeed, this follows from the similar relation in $\WW(V)$. 
Note that if we replace $LG(H)$ by $LG(V)$  the map 
$$I_0:M\to H^0(LG(V),C_\LL(M)),$$
will be zero on the odd part of $M$, and we will see only
some consequences of \eqref{abstr-heat-eq}, such as
$$(\ga(l_1l_2)\ga(l_3l_4)-(-1)^{\ov{l}_2\ov{l}_3}\ga(l_1l_3)\ga(l_2l_4))I_0(m)=0,$$
for $l_i\in L$, where $L\sub H$ is a fixed Lagrangian subspace.
The map $I_0$ and the differential equations for the image are well known in the theory of the Weil representation (see \cite{GS}, \cite{Goncharov}). 

Recall that in the case when $V$ is an even $2$-dimensional symplectic space with the standard basis $(e^*,e)$
we calculated the homomorphism $\ga$ explicitly in terms
of the action on $M^r(\LL)(U_0)$, where $U_0\sub LG(H)$ is an open subset of $L$ transversal to $e^*\in H$ (see Proposition \ref{ga-calc-prop}).
For example, 
$$\ga(e^*)=-\pa_\la+e^*, \ \ \ga((e^*)^2)=-2\pa_x+(e^*)^2,$$ 
where $(x,\la)$ are natural coordinates on $U_0$.
Given a $\WW(V)$-module $M$, we have the induced action of $\AA_\LL$ on $(\OO\ot M)_\LL=M^r(\LL)\ot_{\WW(V)}M$.
For example, the action of $\ga(e^*)$ and $\ga((e^*)^2)$ on $(\OO\ot M)_\LL(U_0)$ is still given by the above formulas (where $e^*$ and $(e^*)^2$ act
on $M$ as elements of $\WW(V)$), so the sections of the form $I(m)$ satisfy
$$[2\pa_x-(e^*)^2+(\pa_\la-e^*)^2]I(m)=0.$$

All of this also makes sense in the complex analytic context. If we take $M$ to be the standard representation of $H$ on the space $\SS$ of Schwartz functions $f(t)$ on $\R$,
where $e$ acts by $-d/dt$ and $e^*$ acts by $f\mapsto tf$ (in the classical theory it is customary to introduce $i$ in the Heisenberg relations and let $e$ act by $it$).
Let us consider the domain $D\sub U_0$ where $\operatorname{Re}(x)>0$. Then we have an $\OO(D)$-linear isomorphism 
$$(\OO(D)\hat{\ot} \SS)_\LL\to \OO(D): f(x,\la,t)\mapsto I^{cl}(f):=\int_{\R} f(x,\la,t)\exp(-x\frac{t^2}{2}-\la t)dt.$$
One immediately checks that this map interwines the action of $\pa_x-(e^*)^2/2$ and $\pa_\la-e^*$ on the coinvariant sheaf with the usual operators $\pa_x$ and $\pa_\la$ on
the target. The differential equation above translates into the elementary fact that for $f(x,\la,t)=f(t)$ (i.e., for global sections coming from $\SS$), one has
$$(2\pa_x+\pa_\la^2)I^{cl}(f(t))=0.$$
This makes sense also for some generalized functions. In particular, for $f(t)$ being the delta-function of a lattice, this gives the heat equation for the theta function (written
in slightly non-standard coordinates).
 
\subsection{The class of the Picard algebroid associated with a Lagrangian subbundle}\label{alg-class-sec}
 
 
Let $0\to \OO_S\to \HH\rTo{p} \VV\to 0$ be a Heisenberg extension (where $\VV$ is a symplectic superbundle over a smooth superscheme $S$), equipped with a Lagrangian subbundle $\LL\sub \HH$, and a Heisenberg connection $\nabla$,
and let $\AA=\AA_{\HH,\nabla,\LL}$ denote the corresponding Picard algebroid over $S$.

Let $M^r(\LL)$ denote the right Fock module over $\HH$ associated with $\LL$. By definition, $\AA$ acts on $M^r(\LL)$, compatibly with the right action of $U(\HH)$.
Let 
$$\kappa:\AA\to M^r(\LL)$$ 
denote the map given by the action on the vacuum vector $1\in M^r(\LL)$.
Note that the image of $\kappa$ is contained in $M^r(\LL)_{\le 2}$ (see \eqref{D-filtration-property}).
Note that the action on the vacuum vector gives an isomorphism
$$\HH/\LL\rTo{\sim} M^r(\LL)_{\le 1}.$$
In the next lemma we study $M^r(\LL)_{\le 2}$ more closely.

\begin{lemma}\label{Fock-ext-lem}
(i) There is a natural isomorphism
\begin{equation}\label{M2/O-isom}
M^r(\LL)_{\le 2}/\OO\simeq \Hom^s(\LL,\HH/\LL): x\mapsto (l\mapsto x\cdot l\in M^r(\LL)_{\le 1}\simeq \HH/\LL)
\end{equation}
where $\Hom^s(\LL,\HH/\LL)\sub \Hom(\LL,\HH/\LL)$ is the
preimage of the subspace of supersymmetric maps in $\Hom(\LL,\VV/\LL)$.

\noindent
(ii) Suppose over an open covering $(U_i)$ we have data $(\La_i\sub \VV, \si_i:\VV\to \HH)$, where $\La_i\sub \VV$ is a Lagrangian complement to $\LL$, and 
$\si_i:\VV\to \HH$ is a splitting that agrees with the given splitting over $\LL$. 
We identify $\La_i$ with $\LL^\vee$ via the map $\la\mapsto (l\mapsto (\la,l))$.
Let $\La_j=\{\la+\phi_{ij}(\la) \ |\ \la\in \La_i\}$ over $U_{ij}$, where $\phi_{ij}:\La_i\simeq \VV/\LL\to \LL$ is a supersymmetric homomorphism, and let
$\si_j(v)-\si_i(v)=(v,l_{ij})$, where $l_{ij}\in \LL(U_{ij})$. 
Then over $U_i$ we have an induced isomorphism 
$$\psi_i:\LL^\vee \oplus \Hom^s(\LL,\VV/\LL)\rTo{\sim} \Hom^s(\LL,\HH/\LL):(l^*,f)\mapsto l^*+\si_i\circ f$$ 
and
a splitting $\si(M)_i:\Hom^s(\LL,\HH/\LL)\to M^r(\LL)_{\le 2}$ of the projection
$$M^r(\LL)_{\le 2}\rTo{\pi} M^r(\LL)_{\le 2}/\OO\simeq \Hom^s(\LL,\HH/\LL)$$
with the image $1\cdot (\si_i(\La_i)+\si_i(\La_i)\cdot \si_i(\La_i))$.
Over $U_i\cap U_j$ one has
\begin{equation}\label{si-M-i-si-M-j-eq}
((\si(M)_j-\si(M)_i)\circ \psi_i)(l^*,f)=\lan l^*,l_{ij}\ran+\frac{1}{2}\str(\phi_{ij}f)-\frac{1}{2}(f(l_{ij}),l_{ij}),
\end{equation}
where $l^*\in \LL^\vee$, $f\in \Hom^s(\LL,\VV/\LL)$.

\noindent
(iii) One has $\phi_{ij}+\phi_{jk}=\phi_{ik}$ in $S^2\LL(U_{ijk})$.
\end{lemma}

\begin{proof}
(i) The fact that the map \eqref{M2/O-isom} is well defined follows from the inclusion 
$$M^r(\LL)_{\le i}\cdot \LL\sub M^r(\LL)_{\le i-1}$$ 
which is proved by induction on $i$.
The restriction of \eqref{M2/O-isom} to $M^r(\LL)_{\le 1}/\OO\simeq \VV/\LL$
is the map 
$$\VV/\LL\rTo{\sim} \Hom(\LL,\OO)\sub \Hom^s(\HH/\LL): v\mapsto (l\mapsto (v,l)).$$
The induced map
$$M^r(\LL)_{\le 2}/M^r(\LL)_{\le 1}\to \Hom^s(\LL,\VV/\LL)$$
sends $1\cdot v_1\cdot v_2$ to 
$$l\mapsto 1\cdot v_1\cdot v_2\cdot l\mod \OO=(-1)^{\ov{v}_2\ov{l}}(v_1,l)\cdot v_2+(-1)^{\ov{v}_1(\ov{v}_2+\ov{l})}(v_2,l)\cdot v_1,$$
which corresponds to the standard isomorphism 
$$S^2(\VV/\LL)\rTo{\sim}\Hom^s(\LL,\VV/\LL).$$ 

\noindent
(ii) Let us consider the map 
$$\wt{\si}(M)_i:\La_i\oplus S^2\La_i\to M^r(\LL)_{\le 2}:(\la,\la_1\la_2)\mapsto 1\cdot (\si_i(\la)+\si_i(\la_1)\si_i(\la_2)).$$
Then the composition $\pi\wt{\si}(M)_i:\La_i\oplus S^2\La_i\to \Hom^s(\LL,\HH/\LL)$ corresponds to $\psi_i$ via the isomorphisms
$$s_i:\La_i\rTo{\sim} \LL^\vee:\la\mapsto (l\mapsto (\la,l))$$
\begin{align*}
&s_i^{(2)}:S^2\La_i\rTo{\sim} S^2(V/\LL)\simeq \Hom^s(\LL,\VV/\LL):\\
&\la_1\la_2\mapsto (l\mapsto (-1)^{\ov{\la}_2\ov{l}}(\la_1,l)\la_2+(-1)^{\ov{\la}_1(\ov{\la}_2+\ov{l})}(\la_2,l)\la_1\mod \LL).
\end{align*}
Hence, using these isomorphisms we get from $\wt{\si}(M)_i$
the splitting $\si(M)_i$.
Explicitly, for $\la,\la_1,\la_2\in \La_i$ we have
$$\si(M)_i(s_i(\la))=1\cdot \si_i(\la), \ \ (\si(M)_i\psi_i)(s_i^{(2)}(\la_1\la_2))=1\cdot \si_i(\la_1)\si_i(\la_2).$$

Over $U_i\cap U_j$, for $\la\in \La_i$, set $\la'=\la+\phi_{ij}(\la)\in \La_j$. Then 
$$s_j(\la')(l)=(\la+\phi_{ij}(\la),l)=(\la,l)=s_i(\la)(l).$$
Hence,
$$\si(M)_j(s_i(\la))=\si(M)_j(s_j(\la'))=1\cdot \si_j(\la')=(\la,l_{ij})+\si(M)_i(s_i(\la)),$$
or equivalently,
$$\si(M)_j(l^*)=\lan l^*,l_{ij}\ran+\si(M)_i(l^*),$$
which gives the first component of \eqref{si-M-i-si-M-j-eq}.

On the other hand, for $\la_1,\la_2\in \La_i$ set 
$$c_a=(\la_a,l_{ij}), \ \  \la'_a=\la_a+\phi_{ij}(\la_a), \ \ p_a=\ov{\la}_a, \ \text{ for } a=1,2.$$ 
We have
\begin{align}\label{psi-i-psi-j-eq}
&\psi_j(s_j^{(2)}(\la'_1\la'_2))(l)=\si_j(\la'_1)(\la_2,l)+(-1)^{p_1p_2}\si_j(\la'_2)(\la_1,l) \nonumber\\
&\equiv c_1\cdot (\la_2,l)+(-1)^{p_1p_2}c_2\cdot (\la_1,l)+
\psi_i(s_i^{(2)}(\la_1\la_2)) \mod \LL.
\end{align}
Also, we have
\begin{align*}
&\si(M)_j\psi_j(s_j^{(2)}(\la'_1\la'_2))=1\cdot \si_j(\la'_1)\si_j(\la'_2)\\
&=1\cdot (\si_i(\la_1)+c_1)(\si_i(\la_2)+c_2+\phi_{ij}(\la_2))\\
&=\si(M)_i\psi_i(s_i^{(2)}(\la_1\la_2))+1\cdot (c_1\cdot\si_i(\la_2)+(-1)^{p_1p_2}c_2\cdot\si_i(\la_1)+(\la_1,\phi_{ij}(\la_2)))\\
&+c_1c_2.
\end{align*}
Combining this with
$$\si(M)_j(\psi_j(s_j^{(2)}(\la'_1\la'_2))-\psi_i(s_i^{(2)}(\la_1\la_2)))
=1\cdot(c_1\cdot \si_i(\la_2)+(-1)^{p_1p_2}c_2\cdot \si_i(\la_1)+2c_1c_2)$$
(which follows from \eqref{psi-i-psi-j-eq}),
we get
$$(\si(M)_j-\si(M)_i)\psi_i(s_i^{(2)}(\la_1\la_2))=(\la_1,\phi_{ij}(\la_2))-c_1c_2,$$
which is equivalent to the second component of \eqref{si-M-i-si-M-j-eq}.

\noindent
(iii) This is straightforward.
\end{proof}

By Lemma \ref{Fock-ext-lem},
we have a morphism of exact sequences,
\begin{diagram}
0&\rTo{}&\OO&\rTo{}& \AA&\rTo{}& \TT_S&\rTo{}& 0\\
&&\dTo{\id}&&\dTo{\kappa}&&\dTo{\ov{\kappa}}\\
0&\rTo{}& \OO&\rTo{}& M^r(\LL)_{\le 2}&\rTo{\pi}& \Hom^s(L,\HH/L)&\rTo{}& 0$$
\end{diagram}

Thus, the extension $\AA\to \TT_S$ is obtained by the pullback from the extension associated with $M^r(\LL)_{\le 2}$.
Since we understand the cocycle corresponding to the latter extension, due to Lemma \ref{Fock-ext-lem}, we can now
proceed to compute the cocycle associated with $\AA$.

\begin{lemma}\label{ov-kappa-lem} 
The map $\ov{\kappa}$, induced by $\kappa$, is given by $\xi\mapsto (l\mapsto -\nabla_\xi(l)\mod \LL)$.
\end{lemma}

\begin{proof}
Let us lift (locally) a vector field $\xi$ to a section $\wt{\xi}$ of $\AA$.
We have
$$\ov{\kappa}(\xi)(l)=\pi(\xi(1))(l)=\xi(1)\cdot l \in M^r(\LL)_{\le 1}\simeq \HH/\LL.$$
Applying the defining property of the action of $\xi$ on $M^r(\LL)$, we get
$$0=\xi(1\cdot l)=\xi(1)\cdot l+1\cdot \nabla_\xi(l),$$
which implies our assertion.
\end{proof}

Consider the composed map, which is $\OO_S$-linear,
$$\nabla_\LL:\LL\to \HH\rTo{\nabla} \HH\ot \Om_S\to \HH/\LL\ot \Om_S,$$
and let
$$\de=p\circ \nabla_\LL:\LL\to \VV/\LL\ot \Om_S\simeq \LL^\vee\ot \Om_S$$
(where the isomorphism $\VV/\LL\to \LL^\vee$ is given by $v\mapsto (l\mapsto (v,l))$).
For a vector field $\xi$, we have the induced maps
$$\nabla_{\LL,\xi}:\LL\to \HH/\LL, \ \ \de_\xi:\LL\to \VV/\LL\simeq \LL^\vee.$$

Note that compatibility of $\nabla$ with the Heisenberg extension structure, together with the fact that $\nabla$ is Lagrangian, imply
the following symmetry of $\de$:
$$(\de_\xi(l_1),l_2)=(-1)^{\ov{l}_1\ov{l}_2}(\de_\xi(l_2),l_1),$$
where $l_1,l_2\in \LL$. Equivalently, $\de$ is an even global section of $S^2\LL^\vee\ot \Om^1_S$.

\begin{lemma}\label{class-calc-lem}
Let $(\La_i,\si_i)$ be local data over a covering $(U_i)$ as in Lemma \ref{Fock-ext-lem}(ii).
Let us define sections $\a_i\in \LL^\vee\ot \Om^1(U_i)$ by 
$$\a_i(l)(\xi)=\nabla_{\LL,\xi}(l)-\si_i(\de_\xi(l)),$$
where $\xi$ is a local vector field. 

\noindent 
(i) Over $U_{ij}$ one has
\begin{equation}\label{alpha-i-j-eq}
\lan \a_j-\a_i,l\ran=-(\de(l),l_{ij}).
\end{equation}

\noindent
(ii) Let $\nabla_i:\TT_S\to \AA$ denote the local connection over $U_i$ associated with $\si_i(\La_i)$ (see Proposition \ref{transv-flat-connections-prop}). 
Then the corresponding transition $1$-forms over $U_i\cap U_j$ are given by
\begin{equation}\label{om-ij-eq}
\om_{ij}:=\nabla_j-\nabla_i=-\lan \a_i,l_{ij}\ran-\frac{1}{2}\str(\phi_{ij}\de)+\frac{1}{2}\lan \de(l_{ij}),l_{ij}\ran,
\end{equation}
where $\phi_{ij}$ and $l_{ij}$ are as in Lemma \ref{Fock-ext-lem}(ii). Here $\str(\phi_{ij}\de)(\xi)=\str(\phi_{ij}\de_\xi)$.

\noindent
(iii) Assume that $\La_i$ is preserved by the connection $\nabla$ on $\VV$, and let $\nabla^{\LL}_i$ denote the connection on $\LL$
coming from the induced connection on $\VV/\La_i$ via the isomorphism $\LL\simeq \VV/\La_i$.
Then
$$\nabla^{\LL}_j-\nabla^{\LL}_i=-\phi_{ij}\circ \de,$$
where we view $\phi_{ij}$ as a morphism $\VV/\LL\to \LL$. 
\end{lemma}

\begin{proof} (i) The formula \eqref{alpha-i-j-eq} follows immediately from the identity
$$\nabla_{\LL,\xi}(l)\equiv \lan\a_j,l\ot\xi\ran+\si_j(\de_\xi(l))\mod \LL\equiv \lan \a_j,l\ot\xi\ran+(\de_\xi(l),l_{ij})+\si_i(\de_\xi(l))\mod\LL.$$

\noindent
(ii) Since the splitting $\si(M)_i$ of $\pi:M^r(\LL)_{\le 2}\to \Hom^s(\LL,\HH/\LL)$ over $U_i$ 
has the image $1\cdot (\si_i(\La_i)+\si_i(\La_i)\cdot\si_i(\La_i))$,
the induced connection $\TT_S\to \AA$ is exactly $\nabla_i$.
Thus, for a local vector field, one has
$$\lan \om_{ij},\xi\ran=(\si(M)_j-\si(M)_i)(\ov{\kappa}(\xi)).$$
We can write
$$\ov{\kappa}(\xi)(l)\equiv \lan l^*,l\ran+\si_i(f(l)) \mod \LL,$$
for some $l^*\in \LL^\vee$ and a symmetric map $f:\LL\to \VV/\LL$,
where $l\in \LL$. By Lemma \ref{ov-kappa-lem}, this means that
$$l^*=-\lan \a_i,\xi\ran \ \text{ and } \ f=-\de_\xi.$$

Hence, applying \eqref{si-M-i-si-M-j-eq}, we get
$$\lan \om_{ij},\xi\ran=-\lan \a_i,l_{ij}\ot \xi\ran-\frac{1}{2}\str(\phi_{ij}\de_\xi)+\frac{1}{2}(\de_\xi(l_{ij}),l_{ij}),$$
as claimed.

\noindent
(iii) Let $\de_i:\LL\to \La_i\ot \Om_S\sub \VV\ot \Om_S$ correspond to $\de$ via the isomorphism $\La_i\simeq \VV/\LL$.
Then we have an equality of maps $\LL\to \VV\ot \Om_S$,
$$\de_j=\de_i+\phi_{ij}\de_i,$$
and for $l\in \LL$, one has the equality of sections of $\VV\ot \Om_S$,
$$\nabla(l)=\de_i(l)+\nabla^{\LL}_i(l)=\de_j(l)+\nabla^{\LL}_j(l).$$
This immediately gives the result.
\end{proof}

Suppose we have a vector bundle $\LL$, and a symmetric $\OO_S$-linear morphism $\de:\LL\to \LL^\vee\ot\Om^1_S$.
Let us consider the complex concentrated in degrees $0$ and $1$,
$$C(\LL,\de):=[\LL\rTo{\de} \LL^\vee\ot \Om^1_S].$$

\begin{definition} {\it Extension data} $(\EE,\de_\EE)$ for $C(\LL,\de)$ consists of an extension 
\begin{equation}\label{E-L-ext-eq}
0\to \OO_S\to \EE\rTo{\pi} \LL^\vee\to 0
\end{equation} 
and an $\OO_S$-linear morphism $\de_\EE:\LL\to \EE\ot\Om^1_S$,
such that $(\pi\ot \id)\de_\EE=\de$. 
\end{definition}

\noindent
{\bf Construction}. Given extension data $(\EE,\de_\EE)$
we construct a class $\kappa(\EE,\de_\EE)\in H^1(S,\Om^1_S)$ as follows.
Dualizing, we get an extension
$$0\to \LL\to \EE^\vee\to \OO_S\to 0$$
and a morphism
$$\de'_\EE:\EE^\vee\ot \LL\to \Om^1.$$
The fact that $\de$ is symmetric implies that $\de'_\EE$ vanishes on the image of ${\bigwedge}^2(\LL)$.
Now we define $\kappa(\EE,\de_\EE)$ as the class of the extension of $\OO_S$ by $\Om^1_S$, obtained as the push-out of the natural extension
$$0\to \EE^\vee\ot \LL/{\bigwedge}^2(\LL)\to S^2(\EE^\vee)\to \OO_S\to 0$$
with respect to $\de'_\EE$.

\begin{definition}
(i) Consider a complex $C(\LL,\de)$ associated with some vector bundle $\LL$ and a symmetric map $\de$.
We say that it is equipped with a {\it flat structure} with respect to an open covering $(U_i)$
if over $U_i$ we are given flat connections $\nabla^{\LL}_i$ on $\LL$, and we have sections $\phi_{ij}\in S^2\LL(U_{ij})$, such that 
\begin{itemize}
\item
the map $\de$ is horizontal as a section of $(\LL^\vee)^{\ot 2}\ot \Om^1$ with respect to each connection
$\nabla^{\LL}_i$;
\item
over $U_i\cap U_j$
the maps $R_{ij}:=\nabla^\LL_j-\nabla^{\LL}_i$ and $R^\vee_{ij}:=\nabla^{\LL^\vee}_j-\nabla^{\LL^\vee}_i$ define a chain map
\begin{diagram}
\LL&\rTo{}&\LL^\vee\ot\Om^1_S\\
\dTo{R_{ij}}&&\dTo{R^\vee_{ij}}\\
\LL\ot\Om^1_S&\rTo{}&\LL^\vee\ot \Om^2_S
\end{diagram}
which is homotopic to zero via a map $\LL^\vee\ot\Om^1_S\to \LL\ot\Om^1_S$ induced by $-\phi_{ij}$;
\item over $U_{ijk}$ one has $\phi_{ik}=\phi_{ij}+\phi_{jk}$.
\end{itemize}

\noindent
(ii) Assume a flat structure on $C(\LL,\de)$ is given. {\it Flat extension data} $(\EE,\de_\EE,\nabla^\EE_i)$, compatible with the flat structure $(\nabla^{\LL}_i)$,
is extension data $(\EE,\de_\EE)$ together with flat connections $\nabla^{\EE}_i$ on $\EE|_{U_i}$ such that 
\begin{itemize}
\item
all maps in \eqref{E-L-ext-eq} are horizontal (where $\OO$ has the standard connection and $\LL^\vee$ is equipped with the dual connection to $\nabla^\LL_i$); 
\item
$\de_\EE$ is horizontal;
\item 
$\nabla^{\EE}_j-\nabla^{\EE}_i=\de_\EE\circ\phi_{ij}\circ\pi$ over $U_{ij}$.
\end{itemize}
The last condition is equivalent to 
\begin{equation}\label{nabla-E-overlap-eq}
\nabla^{\EE^\vee}_j-\nabla^{\EE^\vee}_i=-\iota\circ\phi_{ij}\circ\de''_{\EE},
\end{equation}
where 
$\iota:\LL\hra \EE^\vee$ is the natural inclusion and $\de''_{\EE}:\EE^\vee\to \LL^\vee\ot\Om^1_S$ is obtained from $\de_\EE$ by dualization.
(More precisely, $\iota(l)(e)=(-1)^{\ov{l}\ov{e}}\pi(e)(l)$.)

\noindent
(iii) Assume a flat structure on $C(\LL,\de)$ is given. A Cech $1$-cocycle $(l_{ij},\a_i)$, where $l_{ij}\in \LL(U_{ij})$, $\a_i\in \LL^\vee\ot \Om^1(U_i)$,
$\de(l_{ij})=\a_i-\a_j$, is called horizontal if $\nabla^{\LL}_i(l_{ij})=-\phi_{ij}\a_j$ and $\nabla^{\LL^\vee}_i(\a_i)=0$ in $\LL^\vee\ot \Om^2$. 
Such a cocycle is a horizontal coboundary if $l_{ij}=l_i-l_j$, $\a_i=\de(l_i)$, where $l_i\in \LL(U_i)$ are $\nabla^{\LL}_i$-horizontal.
We denote by $H^1(S,C(\LL,\de))^h$ the group of horizontal $1$-cocycles modulo horizontal coboundary.
\end{definition}


Assume that over $U_i$ we have chosen a Lagrangian complement $\La_i\sub \VV$ preserved by the connection induced by $\nabla$.
Then 
$$\HH=\LL\oplus \KK_i,$$
where $\KK_i:=p^{-1}(\La_i)\sub \HH$.
The connection on $\HH$ has block form
\begin{equation}\label{nabla-block-form-eq}
\nabla=\left(\begin{matrix} \nabla^\LL_i & 0 \\ \de_\EE & \nabla_i^\EE\end{matrix}\right),
\end{equation}
for some flat connection $\nabla^\EE_i$ on $\EE=\HH/\LL\simeq \KK_i$ defined over $U_i$. The equation $\nabla^2=0$ implies in addition that
$\de_\EE$ is horizontal as a section of $\EE\ot \LL^\vee\ot\Om^1_S$ (with respect to $\nabla_i^\EE$ and $\nabla_i^\LL$).
The projection $\pi:\EE\to \VV/\LL\simeq \LL^\vee$ is horizontal since the connection on $\La_i$ is dual to $\nabla^\LL_i$ under the symplectic pairing.
Hence, $\de=\pi\circ \de_\EE$ is horizontal as a section of $(\LL^\vee)^{\ot 2}\ot\Om^1_S$.

On $U_{ij}$, set
$$a_{ij}:=\phi_{ij}\circ\pi:\EE\to \LL.$$
Then $\KK_j$ is the graph of $a_{ij}|_{\KK_i}:\KK_i\to \LL$. Rewriting the block decomposition
\eqref{nabla-block-form-eq} with respect to $\HH=\LL\oplus\KK_j$ gives
$$\nabla^\LL_j-\nabla^\LL_i=-a_{ij}\circ\de_\EE=-\phi_{ij}\circ\de,$$
$$\nabla^\EE_j-\nabla^\EE_i=\de_\EE\circ a_{ij}=\de_\EE\circ\phi_{ij}\circ \pi.$$
Together with Lemma \ref{Fock-ext-lem}(iii) this implies that $(\nabla^\LL_i,\phi_{ij})$ define a flat structure on $C(\LL,\de)$ and that $(\EE,\de_\EE,\nabla^\EE_i)$ 
are compatible flat extension data.

Note that a choice of horizontal splitting $s_i:\La_i\to \KK_i$ of the projection $\KK_i\to \La_i$ is equivalent to a choice of a $\nabla^\EE_i$-horizontal retraction
$r_i:\EE\to \OO_S$. 

\begin{lemma}\label{ext-data-classes-lem}
(i) For extension data $(\EE,\de_\EE)$, consider the chain map 
$$[\LL\to \EE^\vee]\to C(\LL,\de)=[\LL\to \LL^\vee\ot \Om^1]$$ 
with the components $\id_\LL$ and $\de''_\EE:\EE^\vee\to \LL^\vee\ot \Om^1$. 
Since $[\LL\to \EE^\vee]\simeq \OO[-1]$, this gives a well defined class $c(\EE,\de_\EE)\in H^1(C(\LL,\de))$.
Suppose we have local retractions $r_i:\EE\to \OO_S$ over $U_i$. Set $\a_i=r_i\de_\EE\in \LL^\vee\ot \Om^1(U_i)$, and
define $l_{ij}\in \LL(U_{ij})$ by $\iota(l_{ij})=r_i-r_j$.
Then $(l_{ij},\a_i)$ is a Cech $1$-cocycle representing $c(\EE,\de_\EE)$.

We have a well defined chain map
$$\tau: S^2(C(\LL,\de)[1])\to \Om^1_S[1],$$
such that 
\begin{equation}\label{kappa-tau-c-id}
\kappa(\EE,\de_\EE)=\tau(c(\EE,\de_\EE)^2),
\end{equation}
where in the right-hand side we use the induced map
$S^2H^1(S,C(\LL,\de))\to H^1(S,\Om^1_S)$. In terms of Cech cocycles we have 
$$\tau((l_{ij},\a_i)^2)=[\lan \de(l_{ij}),l_{ij}\ran-2\lan \a_i,l_{ij}\ran],$$
and the corresponding symmetric pairing is given by
\begin{equation}\label{T-sym-eq}
T((l_{ij},\a_i),(l'_{ij},\a'_i))=\lan \de(l_{ij}),l'_{ij}\ran-\lan \a_i,l'_{ij}\ran-(-1)^{pq}\lan \a'_i,l_{ij}\ran,
\end{equation}
where $(l_{ij},\a_i)$ and $(l'_{ij},\a'_i)$ are cocycles of parities $p$ and $q$.

\noindent
(ii) Assume a flat structure on $C(\LL,\de)$ is given. 
If $(\EE,\de_\EE,\nabla^\EE_i)$ are compatible flat extension data
and $r_i:\EE\to \OO_S$ are $\nabla^\EE_i$-horizontal local retractions over $U_i$, then the corresponding Cech cocycle
$(l_{ij},\a_i)$ defined in (i) is horizontal and gives a well defined class $c^h(\EE,\de_\EE,\nabla^\EE_i)\in H^1(S,C(\LL,\de))^h$.
We also have a well defined map
$$\tau^h:S^2H^1(S,C(\LL,\de))^h\to H^1(S,\Om^{1,cl})$$
given by the same formula as in (i).
\end{lemma}

\begin{proof}
(i) The fact that $c(\EE,\de_\EE)$ is represented by the Cech $1$-cocycle $(l_{ij},\a_i)$ 
followis immediately from the fact that $(l_{ij},r_i)$ is the Cech $1$-cocycle in $[\LL\to \EE^\vee]$ corresponding to the quasi-isomorphism with 
$\OO[-1]$.

The chain map $\tau$ is given by  
\begin{diagram}
{\bigwedge}^2\LL&\rTo{}& \LL\ot \LL^\vee\ot\Om^1&\rTo{}& S^2(\LL^\vee\ot \Om^1)\\
&&\dTo{\tau_0}\\
&&\Om^1_S
\end{diagram}
where $\tau_0$ is the natural supertrace map. The fact that it is a chain map follows from the symmetry of $\de$.

Composing $\tau$ with the natural map $S^2([\LL\to\EE^\vee][1])\to S^2([\LL\to \LL^\vee\ot\Om^1][1])$, we get a diagram
\begin{diagram}
\EE^\vee\ot \LL/{\bigwedge}^2\LL&\rTo{}& S^2(\EE^\vee)\\
\dTo{}&&\dTo{}\\
\LL\ot \LL^\vee\ot\Om^1/{\bigwedge}^2\LL&\rTo{}& S^2(\LL^\vee\ot \Om^1)\\
\dTo{\tau_0}\\
\Om^1_S
\end{diagram}
where the complex in the first row is quasi-isomorphic to $\OO$. By definition, the composed chain map represents $\kappa(\EE,\de_E)$,
so this proves \eqref{kappa-tau-c-id}.

To get the formula in terms of Cech cocycles we start with the fact that the Cech cocycle with values in
$[\EE^\vee\ot\LL/{\bigwedge}^2\LL\to S^2(\EE^\vee)]$, corresponding to the quasi-isomorphism with $\OO[-1]$, is given
by 
$$(r_j^2-r_i^2,r_i^2)=(l_{ij}\ot l_{ij}-2r_i\ot l_{ij},r_i^2).$$
Hence, the corresponding class in $H^0(S^2(C(\LL,\de)[1]))$ is represented by
$$(\de(l_{ij})\ot l_{ij}-2\a_i\ot l_{ij}, \de''_\EE(r_i)^2).$$
Finally, taking the trace to the first component we get the claimed formula.

\noindent
(ii) 
Since $r_i$ and $\de_\EE$ are horizontal, we deduce that $\a_i=r_i\de_\EE$ is also horizontal. Moreover,
$$\de(l_{ij})=\de''_\EE(r_i-r_j)=\a_i-\a_j.$$
Next, using \eqref{nabla-E-overlap-eq} we get
\begin{align*}
&\nabla^\LL_i(l_{ij})=\nabla^{\EE^\vee}_i(r_i-r_j)=-\nabla^{\EE^\vee}_i(r_j)\\
&=(\nabla^{\EE^\vee}_j-\nabla^{\EE^\vee}_i)(r_j)=-\phi_{ij}\de''_{\EE}(r_j)=-\phi_{ij}\a_j.
\end{align*}
Thus, $(l_{ij},\a_i)$ is a horizontal Cech 1-cocycle.
If we choose different horizontal retractions $r'_i\in\EE^\vee(U_i)$, then $r'_i-r_i=l_i\in \LL(U_i)$ are $\nabla^\LL_i$-horizontal,
so our horizontal 1-cocycle will change by a horizontal coboundary.

Next, let us check that there is a well quadratic map $\tau^h$ sending the class of a horizontal $1$-cocycle $(l_{ij},\a_i)$ to the class of $(\b_{ij})$,
where 
$$\b_{ij}:=\lan \de(l_{ij}),l_{ij}\ran-2\lan \a_i,l_{ij}\ran$$ 
By part (i), we know that $(\b_{ij})$ is a Cech $1$-cocycle. Let us now check that the 1-form $\b_{ij}$ is closed.
Let $D_i$ denote the covariant exterior derivative associated with $\nabla^\LL_i$. Since $\de$ and $\a_i$ are horizontal,
using symmetry of $\de$ we get
$$d\lan \de(l_{ij}),l_{ij}\ran=-2\lan\de(l_{ij}),D_il_{ij}\ran,$$
while
$$d\lan \a_i,l_{ij}\ran=-\lan \a_i,D_il_{ij}\ran.$$
Combining all this with $\a_j=\a_i-\de(l_{ij})$, we obtain
$$d\b_{ij}=2\lan\a_j,D_il_{ij}\ran=-2\lan \a_j,\phi_{ij}\a_j\ran=0$$
in $\Om_S^2$, where the last equality follows from the symmetry of $\phi_{ij}$.

Finally, we need to check that the symmetric pairing \eqref{T-sym-eq} associated with $\tau$, when applied to horizontal $1$-cocycles, 
descends to the quotients by horizontal coboundaries.
Let $h_i\in \LL(U_i)$ be $\nabla^\LL_i$-horizontal of parity $p$. The corresponding horizontal coboundary is $(h_i-h_j,\de(h_i))$.
Now our assertion follows from the fact that for a horizontal $1$-cocycle $(l'_{ij},\a'_i)$ of parity $q$, the $1$-cocycle
\begin{align*}
&\lan \de(h_i-h_j),l'_{ij}\ran-\lan \de(h_i),l'_{ij}\ran-(-1)^{pq}\lan \a'_i,h_i-h_j\ran\\
&=-(-1)^{pq}\lan \de(l'_{ij}),h_j\ran+(-1)^{pq}\lan\a'_i,h_j\ran-(-1)^{pq}\lan \a'_i,h_i\ran\\
&=(-1)^{pq}\lan \a'_j,h_j\ran-(-1)^{pq}\lan \a'_i,h_i\ran
\end{align*}
is a coboundary.
\end{proof}

\begin{theorem}\label{alg-class-thm}
(i) Given a Heisenberg extension $\HH\to \VV$ with a Heisenberg connection $\nabla$ and a Lagrangian subbundle $\LL\sub \HH$, set $\EE:=\HH/\LL$,
which is naturally an extension of $\VV/\LL\simeq \LL^\vee$ by $\OO$. Then we have $\OO$-linear maps $\de_\EE:=\nabla_\LL:\LL\to \EE\ot \Om^1_S$ and
$\de:\LL\to  \LL^\vee\ot \Om^1_S$ induced by the connection, and hence the corresponding class $c(\EE,\de_\EE)\in H^1(S,C(\LL,\de))$.
Then the image $c(\AA)'$ of the class $c(\AA)$ in $H^1(S,\Om^1_S)$ is given by
$$c(\AA)'=\frac{1}{2}c_1(\Ber(\LL))+\frac{1}{2}\tau(c(\EE,\de_\EE)^2).$$

\noindent
(ii) Suppose in addition we have Lagrangian complements $\La_i\sub \VV$ preserved by $\nabla$, and horizontal retractions $r_i:\EE=\HH/\LL\to \OO_S$ over $U_i$,
so that $C(\LL,\de)$ gets equipped with a flat structure, and we have the corresponding flat extension data $(\EE,\de_\EE,\nabla^\EE_i)$.
Then we have the same formula for $c(\AA)\in H^1(S,\Om^{1,cl})$ as above:
$$c(\AA)=\frac{1}{2}c_1(\Ber(\LL))+\frac{1}{2}\tau^h(c^h(\EE,\de_\EE,\nabla^\EE_i)^2).$$
\end{theorem}

\begin{proof}
(i) Let us choose an open covering $(U_i)$ together with the data $(\La_i\sub \VV,\si_i\VV\to \HH)$ as in Lemma \ref{Fock-ext-lem}(ii).
We have the induced local splittings $\si_i:\LL^\vee\simeq \VV/\LL\to \HH/\LL$, which are equivalent to local retractions $r_i:\HH/\LL\to \OO$.
Then $\si_j-\si_i\in (\LL^\vee)^\vee\simeq \LL$ corresponds to $l_{ij}=r_i-r_j$, and
$\a_i=r_i\de_\EE\in \LL^\vee\ot \Om^1(U_i)$ agree with $\a_i$ defined in Lemma \ref{class-calc-lem}.
Therefore, formula \eqref{om-ij-eq} gives the $1$-cocycle $(\om_{ij})$ with the class $c(\AA)'$:
$$\om_{ij}=-\frac{1}{2}\str(\phi_{ij}\de)+\frac{1}{2}\tau((l_{ij},\a_i)^2),$$
where $\de_j-\de_i=\phi_{ij}\de_i$, with $\phi_{ij}\in S^2\LL(U_{ij})$.

It remains to observe that Lemma \ref{class-calc-lem}(iii) implies that the $1$-cocycle $-\str(\phi_{ij}\de)$ represents the class $c_1(\Ber(\LL))$.

\noindent
(ii) Using the same data as in (i), but with $\La_i$ preserved by $\nabla$ and $r_i$ horizontal, we will get that $c(\AA)$ is the class of the $1$-cocycle $(\om_{ij})$,
where now the $1$-forms $\om_{ij}$ are closed (since the corresponding connections $\nabla_i:\TT_S\to \AA$ are flat by Proposition \ref{transv-flat-connections-prop}(ii)).
Hence, our formula for $c(\AA)$ again follows essentially from \eqref{om-ij-eq}. The term given by the closed $1$-form $-\frac{1}{2}\str(\phi_{ij}\de)$ corresponds to $\frac{1}{2}c_1(\Ber(\LL))$. On the orher hand, the remaining term represents $\frac{1}{2}\tau^h(c^h(\EE,\de_\EE,\nabla^\EE_i)^2)$, as follows from Lemma 
\ref{ext-data-classes-lem}(ii).  
\end{proof}

\begin{remark}
1. In the case when the Heisenberg extension is split: $\HH=\OO_S\oplus \VV$, $\LL\sub \VV\sub \HH$, and the connection $\nabla$ preserves $\VV$, 
the class $c^h(\EE,\de_\EE,\nabla^\EE_i)$ is zero, so we recover the result of Theorem \ref{LG-alg-Ber-thm}.

\noindent
2. Existence of local Lagrangian complements $\La_i$ preserved by $\nabla$, as well as of local horizontal retractions $r_i$ is guaranteed if we work with varieties over $\C$ and use classical topology.
\end{remark}

\subsection{Holonomicity and vanishing of higher derived coinvariants}\label{holonomicity-sec}

Using our twisted $D$-module structures on coinvariants, we will
prove the following criterion for vanishing of the higher derived coinvariants.

\begin{theorem}\label{codim-thm}
Let $\pi:\HH\to \VV$ be a super Heisenberg extension over a smooth superscheme $S$, and let
$\LL_1,\LL_2\sub \HH$ be a pair of Lagrangian subbundles. 
Assume that  for every $i\ge 1$, one has
$$\codim\{s\in S_{\bos} \ |\ \dim (\pi(\LL_1)|_s^+\cap \pi(\LL_2)|_s^+)\ge i\}\ge i.$$
Then the coinvariants object $C_{\LL_1}(M_{\HH}(\LL_2))$ has cohomology concentrated in degree $0$.
In the case when $S$ is even, $\HH$ has a Heisenberg flat connection $\nabla$ and $\LL_2$ is preserved by $\nabla$, 
the coinvariants sheaf, viewed as an $\AA_{\HH,\nabla,\LL_1}$-module, is holonomic.
\end{theorem}

\begin{proof}
{\bf Step 1}. The statement is local, so we can choose a splitting $\HH=\OO\oplus \VV$, such that $\LL_2\sub \VV$. Furthermore, locally there exist compatible trivializations
of $\LL_2$ and $\VV$, so we can assume
$\VV=\OO\ot V$, $\LL_2=\OO\ot L_2$, where $L_2\sub V$ is a constant Lagrangian.

\medskip

\noindent
{\bf Step 2}. Reduction to the case when $V$ is purely even. Set $I:=L_2^-$. Then $\ov{V}:=I^\perp/I=V^+$ is purely even, and by Corollary \ref{odd-rk-cat-eq-cor}, we have an equivalence of categories
$$\WW(V)^{op}-\mod\to \WW(\ov{V})^{op}:M\to M^I.$$
Furthermore, by Proposition \ref{trans-coinv-prop}, we have
$$C_{\LL_1}(M(L_2))=M^r(\LL_1)\ot^{\bbL}_{\WW(V)}M(L_2)=C_{L_2}(M^r(\LL_1))\simeq 
C_I(M^r(\LL_1))\ot^{\bbL}_{\WW(\ov{V})} M_{\ov{V}}(\ov{L_2}),$$
where $\ov{L}_2=L_2/I\sub \ov{V}$. Now applying Theorem \ref{Pf-thm}(i), locally we have an isomorphism
$$C_I(M^r(\LL_1))\simeq M_{\ov{V}}^r(\ov{\LL}_1),$$
for some Lagrangian $\ov{\LL}_1\sub \OO\ot \ov{V}$, such that $\ov{\LL}_1|_s=\LL_1|_s^+$.
Hence, the assumption still holds for the data $(\ov{V},\ov{\LL}_1,\ov{L}_2)$.

\medskip

\noindent
{\bf Step 3. Reduction to an even base}.
Let $i:S_{\bos}\to S$ be the natural embedding.
By the twisted version of Penkov's equivalence (see \cite{Penkov}), the restriction functor 
$$i^*:M\mapsto \OO_{S_{\bos}}\ot_{\OO_S} M$$
is an exact equivalence from  
the category of $\AA_{\LL_1}$-modules of weight $1$ to the similar category for $i^\bullet \AA_{\LL_1}\simeq \AA_{i^*\LL_1}$
(the last isomorphism is given by Lemma \ref{alg-pull-back-lem}). Note that the twisted version follows from the untwisted one using local connections on $\AA_{\LL_1}$ (see Corollary \ref{Lag-open-conn-cor}).
Computing the derived coinvariants using standard resolutions shows that $i^*$ sends
$C_{\LL_1}(M(L_2))$ to the derived coinvariants formed from the restricted data on $S_{\bos}$.
Since $i^*$ is exact and faithful, concentration in degree zero can be checked after restriction.

\medskip

\noindent
{\bf Step 4}. Now we assume that $S$ is even. Note that $M^r(\LL_1)$ is an $\AA_{\LL_1}-\WW(V)$-bimodule. Hence, upon choosing a decomposition $V=L_2\oplus L_2^*$, 
we can think of $M^r(\LL_1)$ as a twisted $D$-module $\MM$ on $L_2\times S$. We claim that $\MM$ is holonomic. This is a local statement, so we can choose a decomposition
$V=L_0\oplus L_0^*$ trasversal to all other Lagrangians. 
Then $\pi(\LL_1)$ is a graph of a symmetric map $L_0\to L_0^*$. 
The lift $\LL_1$ to $\HH$ may also have a central linear term. In the $L_0$-polarization (i.e., viewed as a twisted $D$-module over $L_0\times S$), $\MM$ is therefore a rank-one
exponential connection associated with a fiberwise quadratic-linear function.

Passing from the $L_0$-polarization to the $L_2$-polarization is implemented by a fixed linear symplectic transformation; on $D$-modules this is a composition of linear coordinate changes, exponential twists, and partial Fourier transforms. These operations preserve holonomicity. Thus $\MM$ is holonomic in the $L_2$-polarization as well.

The standard complex computing $M^r(\LL_1)\ot^{\bbL}_{\WW(V)} M(L_2)$, up to a twist with a line bundle on $S$, 
can be identified with the relative de Rham complex for the projection $p:L_2\times S\to S$, so 
$$F:=C_{\LL_1}(M(L_2))\simeq p_+(\MM)\ot \xi,$$ 
where $\xi$ is a line bundle. Since the push-forward preserves holonomicity,
we deduce that $F$ is a holonomic complex of twisted $D$-modules on $S$.

More generally, we can prove holonomicity of $F$ for a given connection $\nabla$ on $\HH$ preserving $\LL_2$ as follows. 
After descent to a finitely generated characteristic-zero subfield, we may work over \(\mathbb C\). For a coherent algebraic twisted $D$-module, holonomicity is equivalent to holonomicity of its analytification and can therefore be checked in the classical topology. Locally in that topology, a horizontal trivialization of \(\mathcal H\) adapted to \(\mathcal L_2\) reduces the problem to the constant case. More precisely, the analytified coinvariant complex is the derived inverse image of the universal algebraic constant-case coinvariant complex
(over the Lagrangian Grassmannian). Hence it is holonomic.

\medskip

\noindent
{\bf Step 5}. 
Now we use a well-known criterion for a holonomic complex $F$ to be in $D^{\ge 0}$: it is enough to check that there exists a smooth stratification $S=\sqcup_i S_i$ such that for each locally closed embedding $j_i:S_i\to S$ one has 
$j_i^!F\in D^{\ge 0}$ and $j_i^!F$ has $\OO$-coherent cohomology (where we use Bernstein's notation $j^!$, see \cite{Bernstein}).

In our case, we strart with the stratification of $S$ by the dimension of intersection of $\LL_1|_s$ with $L_2$. 
On each stratum $Z$, $K:=\LL_1|_Z\cap\LL_2|_Z$ is a vector bundle of rank $r$. Let $\ell:K\to\OO_Z$ measure the difference of the two lifts of $K$ to $\HH$: one coming from 
$\LL_1$ and another from $\LL_2$. Next, we partition $Z$ into two strata: $Z_0$ which is the vanishing locus of $\ell$ and $Z_1=Z\setminus Z_0$. Finally, we refine this to a smooth stratification.

Now let $j:Z\to S$ be one of the strata of the obtained stratification such that the rank of the intersection of $\LL_1$ and $\LL_2$ over $Z$ is $r$, and let $c=\codim_S Z$.
By assumption, $c\ge r$.
$D$-module base change for the cartesian square
\begin{diagram}
L_2\times Z &\rTo{j'}& L_2\times S\\
\dTo{p_Z}&&\dTo{}\\
Z&\rTo{j}&S
\end{diagram}
gives
$$j^!F\simeq p_{Z,+}(j^{\prime !}\MM$$
(up to a harmless invertible line factor which we suppress).
Since $M^r(\LL_1)$ is $\OO_S$-flat, 
$$j^{\prime !}\MM\simeq \MM|_{L_{2,0}\times Z}[-c]$$
(up to the standard normal-line factor). Consequently,
$$j^!F\simeq F_Z[-c],$$
where $F_Z$ is the derived coinvariant complex for the pair $(\LL_1|_Z,\LL_2|_Z)$ over $Z$. 
If $\ell$ is everywhere nonzero on $Z$ then $F_Z=0$. Otherwise, $\ell=0$, so by Corollary \ref{JL-cor}(ii),
we get that $F_Z\simeq \det(K)[r]$. Hence, $j^!F$ is either zero or $\det(K)[r-c]\in D^{\ge 0}$.
Therefore, $F\in D^{\ge 0}$. Since the Koszul complex computing $F$ lies in $D^{\le 0}$, we deduce that the cohomology of $F$
is concentrated in degree $0$.
\end{proof}

For example, conditions of Theorem \ref{codim-thm} are satisfied for $\LL_1=\LL$ and $\LL_2=\OO\ot L_0$, 
where $\LL$ is the universal Lagrangian over $LG(H)$ or $LG(V)$, and $L_0\sub V$ is a fixed Lagrangian,
so in this case the derived coinvariants $C_\LL(M(L_0))$ are concentrated in degree $0$.

\section{Equivariant structures}\label{equiv-sec}

\subsection{Equivariant structure on the Lie algebroid}

Below we use the notion of a $G$-equivariant structure on a Lie algebroid from \cite[Sec.\ 1.8.4]{BB} (in loc.\ cit.\ this structure is referred to as that of
{\it Harish-Chandra Lie algebroid}).
Recall that if an algebraic group $G$ acts on $S$, and $\AA$ is a Lie algebroid over $S$, then a $G$-equivariant structure consists of
\begin{itemize}
\item a structure of a $G$-equivariant bundle on $\AA$ and
\item a homomorphism of Lie algebras $i_{\fg}:\fg={\operatorname{Lie}}(G)\to H^0(X,\AA)$,
\end{itemize}
subject to the following compatibilities:
\begin{itemize}
\item the $G$-action on $\AA$ is compatible with the Lie bracket and with the anchor map;
\item $i_{\fg}$ is compatible with the $G$-actions (where $G$ acts on $\fg$ via the adjoint representation) and with the anchor map;
\item the $\fg$-action on $\AA$ coming from the $G$-action coincides with the adjoint action via the homomorphism $i_{\fg}$.
\end{itemize}

We will describe a setup for constructing equivariant structures on the Lie algebroids associated with isotropic subbundles.
Let $V$ be a symplectic (super) vector space, $\II\sub \OO_S\ot \WW_{\le 1}(V)$, an isotropic subbundle, and let $\AA_\II$ denote the corresponding Lie algebroid on $S$.

Suppose we have the following extra data:
\begin{itemize}
\item an action of an algebraic (super) group $G$ on $S$;
\item a homomorphism $\rho:G\to \Sp(V)$;
\item a Lie homomorphism $\wt{d\rho}:\fg=\Lie(G)\to \WW_{\le 2}(V)$ lifting $d\rho:\fg\to \mathfrak{sp}(V)$.
\end{itemize}

\begin{prop}\label{G-equiv-prop} 
Assume the subbundle $\II\sub \OO_S\ot \WW_{\le 1}(V)$ is $G$-equivariant (where we use the action of $G$  on $S$ and on $V$ via $\rho$)
and that $\wt{d\rho}$ is $G$-equivariant. Then there is a well defined homomorphism of Lie algebras
$$i_{\fg}:\fg\to H^0(S,\AA_\II)$$
such that for $X\in \fg$, the operator $i_\fg(X)$ on $M^r(\II)$ is given by
$$i_\fg(X)(m)=X(m)+(-1)^{\ov{X}\ov{m}}m\cdot \wt{d\rho}(X),$$
where $m\mapsto X(m)$ denotes the action of $\fg$ on $M^r(\II)$ induced by the $G$-equivariant structure.
The natural $G$-action on $\AA_\II$ together with $i_{\fg}$ define a $G$-equivariant structure on $\AA_\II$.
\end{prop}

\begin{proof}
First, we need to check that $i_\fg(X)$ belongs to $\AA_\II$, i.e., commutes with the right action of $\WW(V)$ 
and is compatible with the left action of $\ov{\HH}_I$.
For $v\in V$, we have
\begin{align*}
&i_\fg(X)(mv)=X(mv)+(-1)^{\ov{X}(\ov{m}+\ov{v})}mv\cdot(\wt{d\rho})(X)\\
&=X(m)v+(-1)^{\ov{X}\ov{m}}m\cdot (d\rho)(X)(v)+(-1)^{\ov{X}(\ov{m}+\ov{v})}mv\cdot (\wt{d\rho})(X)\\
&=[X(m)+(-1)^{\ov{X}\ov{m}}m\cdot (\wt{d\rho})(X)]\cdot v,
\end{align*}
where we used the fact that for any $\wt{y}\in \WW_{\le 2}(V)$ lifting $y\in\mathfrak{sp}(V)$, one has
$$\wt{y}v-(-1)^{\ov{y}\ov{v}}v\wt{y}=y(v).$$

Next, given $h\in \ov{\HH}_I$, and $X\in \fg$, we have
$$X(h\cdot m)=X(h)\cdot m+(-1)^{\ov{X}\ov{h}}h\cdot X(m),$$
where $m\in M^r(\II)$, for the action of $\fg$ on $\ov{\HH}_I$ coming from the equivariant structure.
This easily implies that
$$i_\fg(X)(hm)=X(h)\cdot m+(-1)^{\ov{X}\ov{h}}h\cdot i_\fg(X)(m),$$
hence $i_{\fg}(X)$ is a section of $\AA_\II$.

It is easy to check that for $X,Y\in \fg$, one has
$$[i_\fg(X),i_\fg(Y)]=i_{\fg}([X,Y])$$
(one has to use that $((d\rho)(X))(\wt{d\rho}(Y))=\wt{d\rho}([X,Y])$ which follows from the $G$-equivariance of $\wt{d\rho}$).

The compatibility of $i_\fg$ with the $G$-action and with the anchor map is clear. To check the last condition in the definition of a $G$-equivariant structure
we have to compare $X(A)$ with $[i_\fg(X),A]$ for a section $A$ of $\AA_\II$ and $X\in \fg$.
We have for $m\in M^r(\II)$,
\begin{align*}
&i_{\fg}(X)(A(m))-(-1)^{\ov{X}\ov{A}}A i_\fg(X)(m)=X(A(m))+(-1)^{\ov{X}(\ov{A}+\ov{m})}A(m)\cdot \wt{d\rho}(X)\\
&-(-1)^{\ov{X}\ov{A}}A(X(m))-(-1)^{\ov{X}(\ov{A}+\ov{m})}A(m\cdot \wt{d\rho}(X)).
\end{align*}
Using the fact that $A$ commutes with the right action of $\wt{d\rho}(X)$ this simplifies to $X(A(m))-(-1)^{\ov{X}\ov{A}}A(X(m))$, which is exactly $X(A)(m)$.
\end{proof}


\subsection{Localization of Harish-Chandra modules}

Suppose $\II\sub \OO_S\ot V$ is an isotropic subbundle, $G$ acts on $S$, $\rho:G\to \Sp(V)$ is a homomorphism,
$\wt{d\rho}:\fg\to \WW_{\le 1}(V)$ is a Lie homomorphism, such that the conditions of Proposition \ref{G-equiv-prop} are satisfied.

Given a (left) $\WW(V)$-module $M$, the coinvariants $C_\II(M)=M^r(\II)\ot_{\WW(V)} M$ get equipped with an $\AA_\II$-module structure.
Suppose $G$ acts on $M$ compatibly with the $\WW(V)$-action (where $G$ acts on $\WW(V)$ via $\rho$).
Then $C_\II(M)$ has a $G$-equivariant structure as an $\OO$-module, and the $\AA_\II$-action is compatible with the $G$-action.
Thus, $C_\II(M)$ is a weak $(\AA_\II,G)$-module (see \cite[Sec.\ 1.8.5]{BB}). Recall that a weak $(\AA_\II,G)$-module is a Harish-Chandra module
if in addition the $\fg$-action induced by the $G$-action coincides with the one given by the homomorphism $i_\fg:\fg\to \AA_\II$.

\begin{lemma}\label{HC-mod-lem1}
Assume that the $G$-action on $M$ satisfies in addition the following property: the induced $\fg$-action coincides with the
one induced by the $W(V)_{\le 2}$-action via $\wt{d\rho}:\fg\to W(V)_{\le 2}$.
Then $C_\II(M)$ is a Harish-Chandra module for $(\AA_\II,G)$.
\end{lemma}

\begin{proof}
The $\fg$-action on $C_\II(M)$ coming from the $G$-action satisfies
$$X(m_1\ot m)=X(m_1)\ot m+(-1)^{\ov{X}\ov{m}_1}m_1\ot X(m),$$
where $X\in \fg$, $m_1\in M^r(\II)$, $m\in M$. By assumption, this is equal to
\begin{align*}
&X(m_1)\ot m+(-1)^{\ov{X}\ov{m}_1}m_1\ot (\wt{d\rho}(X)\cdot m)=X(m_1)\ot m+(-1)^{\ov{X}\ov{m}_1}(m_1\cdot \wt{d\rho}(X))\ot m\\
&=i_\fg(X)(m_1)\ot m,
\end{align*}
as required.
\end{proof}

Let $L_+$ be a fixed Lagrangian subspace of $V$. Then
$P=P(L_+)=\operatorname{Stab}(L_+)$ be the corresponding maximal
parabolic subgroup of $\operatorname{Sp}(V)$. Its Lie algebra $\mathfrak p(L_+)$ is the
subalgebra $L_+V$ of $\mathfrak{sp}\,V\cong S^2V$. We
can lift $\mathfrak p(L_+)$ to a Lie subalgebra of $\WW_{\le 2}(V)$ via the homomorphism
$$i_+=i^r_{L_+}\colon \mathfrak{p}(L_+)\to \WW_{\le 2}(V)$$ 
sending $\ell v\in \mathfrak p(L_+)$ to $v\cdot\ell$ (the product in $\WW_{\le 2}(V)$) for $\ell\in L_+$ and $v\in V$ (note that $\ell$ is put on the right of $v$). 
Together with the
embedding $L_+\subset V\subset \WW_{\le 2}(V)$ of the abelian Lie
algebra $L_+$ we extend $i_+$ to a lift
\[
  i_+\colon \mathfrak p(L_+)\ltimes L_+\hookrightarrow \hat {\mathfrak
  g}\subset \WW_{\le 2}(V),
\]
of the Lie algebra of $P(L_+)\ltimes L_+$.

\begin{definition}\label{d-1} Let $\tilde P(L_+)=P(L_+)\ltimes L_+\subset \operatorname{Sp}V\ltimes V$ and
  $\tilde{\mathfrak{p}}=\mathfrak{p}(L_+)\ltimes L_+$ its Lie algebra.  An
  $\WW(V)$-module $M$ is called a \textit{Harish-Chandra module} for the
  group $\tilde P(L_+)$ if the action of $i_+\tilde{\mathfrak p}(L_+)$
  on $M$ integrates to an action of $\tilde P(L_+)$.
\end{definition}


\begin{lemma}\label{HC-mod-lem2} 
Let $M$ be a Harish-Chandra $(\WW(V),\tilde P(L_+))$-module. Then 
the action of $\tilde P(L_+)$ on $M$ is compatible with the $G$-action on $\WW(V)$ and
the $\WW(V)$-module structure on $M$:
$$g(h\cdot m)=g(h)\cdot g(m),$$
for $g\in \tilde P(L_+)$, $h\in \WW(V)$, $m\in M$.
\end{lemma}

\begin{proof} 
Since $\tilde P(L_+)$ is connected it is enough to check the compatibility of the $\tilde{\mathfrak p}(L_+)$-actions.
For $X\in \tilde{\mathfrak p}(L_+)$ and $v\in V$, we have
$$i_+(X)\cdot (v\cdot m)=[i_+(X),v]\cdot m+(-1)^{\ov{X}\ov{v}}v\cdot (i_+(X)\cdot m),$$
so the assertion follows from the fact that $v\mapsto [i_+(X),v]$ coincides with the natural action of $X\in \mathfrak{sp}(V)$ on $V$.
\end{proof}

Recall that we denote by $\WW(V)-\mod^{L_+}$ the category of $\WW(V)$-modules on which $L_+$ acts locally nilpotently.

\begin{lemma}\label{F-HC-mod} 
Every module in $\WW(V)-\mod^{L_+}$ is a
  Harish-Chandra module over the pair $(\WW(V),\tilde P(L_+))$.
\end{lemma}

\begin{proof} Recall that $M\simeq M(L_+)\ot M^{L_+}$ (see Proposition \ref{cat-eq-prop}, Corollary \ref{cat-eq-cor}) compatibly with
  $i_+\tilde{\mathfrak p}(L_+)$-action (which acts naturally on the Fock module and trivially on $M^{L_+}$).
  Thus, it is enough to consider the case $M=M(L_+)$.

It is clear that the action of the subalgebra $L_+$
  integrates to an action of the group because $L_+$ acts locally
  nilpotently (for any $v\in M$ there exists an $n$ such that
  $L_+^nv=0$) so that the exponential series terminates. To analyze the action of
  $i_+\mathfrak p(L_+)$, 
   let $V=L_+\oplus L_-$ be a decomposition into the
  sum of Lagrangian subspaces. Then the action of
  $L_-\subset \WW(V)$ on the vacuum vector $1\in M^{L_+}$ of
  $M$ defines an isomorphism $M\cong S^\bullet L_-$, Moreover $L_+L_+$
  is the radical of $\mathfrak p(L_+)\cong L_+V$ and $L_+L_-$ is a
  Levi factor.  Since $L_+$ acts locally nilpotently on $M$ the action
  of the radical integrates to the group.  With the choice of lift
  $i_+$ the Levi factor acts trivially on the generator of $M$.  The
  action of the Levi factor $L_+L_-$ on $V$ preserves $L_-$ and its
  action on $L_-$ integrates to the group of linear automorphisms of
  $L_-$.  Thus the action on $M$ integrates to the induced action of
  $\operatorname{Aut}(L_-)$ on $M\cong S^\bullet L_-$.
\end{proof}


Now let $\II\sub \OO_S\ot \WW_{\le 1}(V)$ be an isotropic subbundle over $S$.

\begin{prop}\label{G-equiv-main-prop}
Assume that a connected algebraic (super) group $G$ acts on $S$ and we have a homomorphism $\rho:G\to \tilde{P}(L_+)$, such that
$\II\sub \OO_S\ot \WW_{\le 1}(V)$ is $G$-equivariant. Set $\wt{d\rho}=i_+\circ d\rho$. Then the construction of Proposition \ref{G-equiv-prop} 
gives a $G$-equivariant structure on $\AA_\II$,
such that for any Harish-Chandra module $M$ for $(\WW(V),\tilde P(L_+))$, the coinvariants sheaf $C_\II(M)$ is a Harish-Chandra module for $(\AA_\II,G)$.
In particular, $C_\II(M(L_+))$ is a Harish-Chandra module for $(\AA_\II,G)$.
\end{prop}

\begin{proof}
To check the assumptions of Proposition \ref{G-equiv-prop} we need to know that $i_+\circ d\rho$ is $G$-equivariant. But this follows from the fact that $d\rho$ is
$G$-equivariant, and $i_+$ is $\tilde{P}(L_+)$-equivariant. Now the assertion follows from
Lemmas \ref{HC-mod-lem1} and \ref{HC-mod-lem2}. Note that
the assumption in Lemma \ref{HC-mod-lem1} is satisfied by our definition of $\wt{d\rho}$.
\end{proof}


\appendix

\section{The class of the Picard algebroid associated with the universal Lagrangian for the $3$-dimensional Heisenberg algebra}\label{P1-example-sec}

In this section we consider the (even) $2$-dimensional symplectic vector space $V$. Let $H=\C\oplus V$ be the corresponding Heisenberg algebra.
Then the Lagrangian Grassmannian $LG(H)$ of $H$ is simply $\P(H)\setminus p$, the space of $1$-dimensional subspaces in $H$, different from the center. 
The projection $\pi:LG(H)\to \P(V)$ identifies it with the total space of $\OO(1)$ over $\P(V)=\P^1$.
Let $\LL\sub \OO\ot H$ (resp., $\LL_0\sub \OO\ot V$) denote the universal Lagrangian over $LG(H)$ (resp., over $\P(V)$). 
We are going to compute the class of the Picard algebroid $\AA_\LL$
and compare it with the pull-back of the class of  $\AA_{\LL_0}$. 

Recall (see \cite{BB}) that a locally trivial Picard algebroid $\AA$ over $X$ has a class $c(\AA)$ in $H^1(X,\Om^{1,cl})$, where $\Om^{1,cl}$ is the sheaf of closed $1$-forms.
The de Rham differential gives a map $d:H^1(X,\OO_X)\to H^1(X,\Om^{1,cl})$.

\begin{prop} Let $X=LG(H)$. Then $d:H^1(X,\OO_X)\to H^1(X,\Om^{1,cl})$ is injective and $c(\AA_\LL)-c(\AA_{\LL_0})=d(e)$, where $e$ is a
generator of the $1$-dimensional space $H^1(\P^1,\OO(-2))$
with respect to the decomposition $H^1(X,\OO_X)=\bigoplus_{i\ge 2} H^1(\P^1,\OO(-i))$.
In particular, $c(\AA_\LL)-c(\AA_{\LL_0})\neq 0$ in $H^1(X_{an},\Om^{1,cl})$ (where we use the classical topology on $X$).
\end{prop}

\begin{proof}
As in Sec.\ \ref{2dim-for-sec} we think of subspaces in $LG(H)$ as pairs $(L,\varphi)$, where $L\sub V$ is a line, and $\varphi:L\to \C$ is a functional.
We use a basis $e,e^*$ of $V$ such that $(e^*,e)=1$. The space $X=LG(H)$ is covered by two affine open subsets, $U_0$ and $U_1$, each isomorphic to
$\A^2$, where $U_0$ (resp., $U_1$) consists of $(L,\varphi)$ such that $e^*\not\in L$ (resp., $e\not\in L$). We have coordinates $(x,\la)$ on $U_0$ (resp., $(y,\mu)$ on $U_1$,
so that $L\sub \C\oplus V$ is spanned by $e+xe^*+\la$ (resp., $-ye+e^*+\mu$). On the intersection $U_{01}$ the coordinate transformation is
$$y=-\frac{1}{x}, \ \ \mu=\frac{\la}{x}.$$

In Sec.\ \ref{2dim-for-sec} we calculated the flat connection $\nabla_0$ on $\AA_\LL|_{U_0}$: 
$$\nabla_0(\pa_x)=\pa_x-(e^*)^2/2, \ \ \nabla_0(\pa_\la)=\pa_\la-e^*.$$
Now we also need a similar connection 
$\nabla_1$ over $U_1$ (coming from the Lagrangian complements 
$\lan e\ran$).
Over $U_1$ we have an identification $M^r(L)=1\cdot \OO(U_1)[e]$, where $e$ acts by multiplication and
$$g\cdot e^*=-\pa_e g+yeg-\mu g.$$
Using this we get 
$$\nabla_1(\pa_y)=\pa_y-e^2/2, \ \ \nabla_1(\pa_\mu)=\pa_\mu+e.$$

Over $U_{01}=U_0\cap U_1$, we want to compute the difference $\nabla_1-\nabla_0$.
For example, we know that
$$\nabla_1(\pa_y)-\nabla_0(\pa_y)=f(x,\la),$$
for some $f(x,\la)\in\OO(U_{01})$. To find this function, it suffices to apply both sides to $1$.
We compute in the model $M^r(L)=S(e^*)$:
$$\nabla_1(\pa_y)(1)=-\frac{1}{2}\cdot (\pa_{e^*}-xe^*-\la)^2(1)=-\frac{1}{2}x^2(e^*)^2-x\la e^*+\frac{x-\la^2}{2},$$
$$\nabla_0(\pa_y)(1)=\nabla_0(x^2\pa_x+x\la\pa_\la)(1)=[x^2(\pa_x-\frac{1}{2}(e^*)^2)+x\la(\pa_\la-e^*)](1)=-\frac{1}{2}x^2(e^*)^2-x\la e^*.$$
This shows that $f(x,\la)=(x-\la^2)/2$.

Similarly,
$$\nabla_1(\pa_\mu)(1)=(\pa_\mu+\pa_e-xe^*-\la)(1)=-xe^*-\la,$$
$$\nabla_0(\pa_\mu)(1)=\nabla_0(x\pa_\la)=(x\pa_\la-xe^*)(1)=-xe^*.$$
Hence,
$$\nabla_1(\pa_\mu)-\nabla_0(\pa_\mu)=-\la.$$

Thus, the difference $\nabla_1-\nabla_0$ on $U_{01}$ is given by the closed $1$-form
$$\om=\frac{x-\la^2}{2}dy-\la d\mu=\frac{x+\la^2}{2}\cdot \frac{dx}{x^2}-\frac{\la d\la}{x}.$$
Thus, the class of $\AA_{\LL}$ is given by the Cech $1$-cocycle $c_{01}=\om$. 

Analogous calculation for $\AA_{\LL_0}$ gives 
$$\om_0=\frac{dx}{2x}$$
(obtained by setting $\la=0$ in $\om$).
We have
$$\om-\om_0=\frac{\la^2 dx}{2x^2}-\frac{\la d\la}{x}=d(-\frac{\la^2}{2x}).$$

It is easy to check that the cocycle $f_{01}=\la^2/x$ corresponds to a generator in $H^1(\P^1,\OO(-2))\sub H^1(X,\OO_X)$.
Finally, the exact sequence of sheaves in classical topology,
$$0\to \C_X\to \OO_X\to \Om^{1,cl}\to 0$$
gives rise to a long exact sequence
$$H^1(X,\C)\to H^1(X_{an},\OO)\to H^1(X_{an},\Om^{1,cl})\to\ldots$$
Since, $H^1(X,\C)=0$, we see that the map 
$H^1(X_{an},\OO)\to H^1(X_{an},\Om^{1,cl})$ is injective.

Now the last assertion follows from Lemma \ref{H1-an-lem} below.
\end{proof}

\begin{lemma}\label{H1-an-lem}
The natural composition $H^1(\P^1,\OO(-2))\to H^1(X,\OO)\to H^1(X_{an},\OO)$ is nonzero.
\end{lemma}

\begin{proof} 
We have a natural identification $H^1(X_{an},\OO)\simeq H^1(\P^1_{an},\pi_*\OO_X)$, and we have a filtration
$$\pi_*\OO_X\supset \pi_*(\OO_X(-Z))\supset \pi_*(\OO_X(-2Z))\supset\ldots,$$
with the associated graded quotients $\pi_*(\OO_X(-iZ))/\pi_*(\OO_X(-(i+1)Z))\simeq \OO_{\P^1}(-i)$.
The image of a nonzero class in $H^1(\P^1,\OO(-2))$ lives in $H^1(\P^1_{an},\pi_*(\OO_X(-2Z)))$ and has a nonzero projection to
$H^1(\P^1,\OO(-2))$ (the class we started from). Hence, it remains to check injectivity of the composition
$$H^1(\P^1,\pi_*(\OO_X(-2Z)))\to H^1(\P^1,\pi_*(\OO_X(-Z)))\to H^1(\P^1,\OO_X).$$
But this follows easily from the long exact sequence of cohomology given the identification of the relevant associated graded quotients with $\OO_{\P^1}(-1)$ and $\OO_{\P^1}$.
\end{proof}

\section{Proof of Proposition \ref{gen-Pf-prop}}
\label{Pf-app}

We use induction on $m$. Considering the action of the element $(e_1,-\phi(e_1),-\la_1)\in L(\phi,\la)$ on $e_1^*\in S(L^\vee)=M(L)$, we see that
$$1=e_1\cdot e_1^*\equiv \la_1\cdot e_1^*+\sum_{i>1}\phi_{i1}e_i^*e_1^*$$
in the space of $L(\phi,\la)$-coinvariants.

For each subset $I\sub [1,m]$, we denote by $\phi^{I}$ (resp., $\la^{I}$) the submatrix of $\phi$ with the rows/columns in $I$ deleted (resp., the subvector of $\la$ with the entries in $I$ deleted).
It is easy to see that the natural embedding
$$S(L^\vee/(e_1^*))\cdot e_1^*\to S(L^\vee)$$
sends $L(\phi^{\{1\}},\la^{\{1\}})\cdot S(L^\vee/(e_1^*))\cdot e_1^*$ to $L(\phi,\la)\cdot S(L^\vee)$ (where we use $e_2^*,\ldots,e_m^*$ as a basis in $L^\vee/(e_1^*)$).
Similarly, multiplying with $e_i^*e_1^*$ on the right allows to reduce to the pair $(\phi^{\{1,i\}},\la^{\{1,i\}})$.
Therefore, 
\begin{equation}\label{Pf-ind-eq}
1\equiv \la_1\cdot \Pf(\phi^{\{1\}},\la^{\{1\}})\cdot e_1^*+\sum_{i>1}\Pf(\phi^{\{1,i\}},\la^{\{1,i\}})]\cdot\phi_{i1}\cdot e_i^*e_1^*
\end{equation}
in the space of $L(\phi,\la)$-coinvariants. 

Let us write the sign in \eqref{Pf-gen-for} as $\eps([1,m],I)$.
Applying the induction assumption, we can rewrite the right-hand side of \eqref{Pf-ind-eq} as
\begin{align*}
&[\sum_{I'\sub [2,m]}\eps([2,m],I')\la_1\la_{I'}\Pf(\phi^{\{1\}\cup I'})\\
&+\sum_{i>1,J\sub [2,m]-i}(-1)^i\eps([2,m]-i,J)\la_J\Pf(\phi^{\{1,i\}\cup J})\phi_{i1}]\cdot e_m^*\ldots e_1^*
\end{align*}
(the sign $(-1)^i$ arises from moving $e_i^*$ through $e_2^*,\ldots,e_{i-1}^*$).
It is easy to see that the first sum is simply $\sum_{1\in I\sub [1,m]}\eps([1,m],I)\la_I\Pf(\phi^I)$.
As for the second sum, one can rewrite it as
$$\sum_{J\sub [2,m]}\la_J\cdot\sum_{i\in [2,m]\setminus J}(-1)^i\eps([2,m]-i,J)\Pf((f^J)^{\{1,i\}})\phi_{i,1}.
$$
Using the formula expanding $\Pf(f^J)$ into Pfaffian minors corresponding to $\{1,i\}$ (one can check that the signs work out),
we can rewrite the right-hand side of \eqref{Pf-ind-eq} as
$$\sum_{1\in I\sub [1,m]}\eps([1,m],I)\la_I\Pf(\phi^I)+\sum_{J\sub [2,m]}\eps([1,m],J)\la_J\Pf(\phi^J),$$
which is the right-hand side of \eqref{Pf-gen-for}.

\end{document}